%% file: Katori_Shirai_annulusPDPP_e3.tex
\documentclass[a4paper,12pt]{article}
\input{katori-shirai_v2d.tex}

\usepackage{amsmath}	
\begin{document}
\def\rB{\textrm{B}}
\def\Conf{\textrm{Conf}}
\def\arg{\textrm{arg }}
\def\Re{\textrm{Re}}
\def\Im{\textrm{Im}}
\def\DetI{\mathop{ \mathrm{Det}_{\mathcal{I}} }}
\def\DetJ{\mathop{ \mathrm{Det}_{\{2,3\}} }}
\def\perdet{\mathop{ \mathrm{perdet} }}
\def\zhat{\widehat{z}}
\def\alphahat{\widehat{\alpha}}
\def\green#1{{\color{green} #1}}
\title{
Zeros of the i.i.d.~Gaussian Laurent series \\
on an annulus:
weighted Szeg\H{o} kernels and 
permanental-determinantal point processes
}
\author{
\katoriname \, \shirainame
}
\date{16 February 2022}
\pagestyle{plain}
\maketitle

\begin{abstract}
On an annulus 
${\mathbb{A}}_q :=\{z \in {\mathbb{C}}: q < |z| < 1\}$
with a fixed $q \in (0, 1)$,
we study a Gaussian analytic function (GAF) and 
its zero set which defines a point process on ${\mathbb{A}}_q$
called the zero point process of the GAF.
The GAF is defined by the 
i.i.d.~Gaussian Laurent series such that
the covariance kernel parameterized by $r >0$ is identified with
the weighted Szeg\H{o} kernel of ${\mathbb{A}}_q$ 
with the weight parameter $r$
studied by Mccullough and Shen.
The GAF and the zero point process 
are rotationally invariant and have a symmetry 
associated with the $q$-inversion of 
coordinate $z \leftrightarrow q/z$ and
the parameter change $r \leftrightarrow q^2/r$.
When $r=q$ they are invariant 
under conformal transformations which preserve ${\mathbb{A}}_q$.
Conditioning the GAF by adding zeros, 
new GAFs are induced such that
the covariance kernels are also given by
the weighted Szeg\H{o} kernel of Mccullough and Shen 
but the weight parameter $r$
is changed depending on the added zeros. 

We also prove that the 
zero point process of the GAF provides
a permanental-determinantal point process (PDPP)
in which each correlation function is expressed by
a permanent multiplied by a determinant.
Dependence on $r$ of the unfolded 2-correlation function
of the PDPP is studied. 
If we take the limit $q \to 0$, 
a simpler but still non-trivial 
PDPP is obtained on the unit disk $\D$. 
We observe that the limit PDPP
indexed by $r \in (0, \infty)$ can be 
regarded as an interpolation between 
the determinantal point process (DPP)
on ${\mathbb{D}}$ studied by
Peres and Vir\'ag ($r \to 0$) and 
that DPP of Peres and Vir\'ag with
a deterministic zero added at the origin ($r \to \infty$). 

\vskip 0.2cm



\end{abstract}

\SSC{Introduction and Main Results}
\label{sec:introduction}
\subsection{Weighted Szeg\H{o} kernel and GAF on 
an annulus}
\label{sec:GAF}

For a domain $D \subset \C$, let $X$ be a random variable
on a probability space which takes values 
in the space of analytic functions on $D$. 
If $(X(z_1), \dots, X(z_n))$ follows a mean zero
complex Gaussian distribution for every $n \in \N$
and every $z_1, \dots, z_n \in D$,
$X$ is said to be a \textit{Gaussian analytic function} (GAF)
\cite{HKPV09}.
In the present paper the zero set of $X$ is regarded as 
a point process on $D$ 
denoted by a nonnegative-integer-valued 
Radon measure
$\cZ_{X}=\sum_{z \in D: X(z)=0} \delta_z$, 
and it is simply called the
\textit{zero point process} of the GAF.
Zero-point processes of GAFs have been extensively studied 
in quantum and statistical physics as solvable models of 
quantum chaotic systems and 
interacting particle systems
\cite{BBL92,BBL96,Han96,Leb99,Leb00,For10,CHSDS06}.
Many important characterizations of their probability laws 
have been reported in probability theory
\cite{EK95,BSZ00,ST04,PV05,HKPV09,Shi12,MS13}.

A typical example of GAF 
is provided by the i.i.d.~Gaussian power series
defined on the unit disk $\D :=\{z \in \C: |z| < 1\}$: 
Let $\N_0 :=\{0,1,2,\dots\}$ and $\{\zeta_n\}_{n \in \N_0}$ 
be i.i.d.~standard complex Gaussian random variables 
with density $e^{-|z|^2}/\pi$
and consider a random power series, 
\begin{equation}
X_{\D}(z)=\sum_{n=0}^{\infty} \zeta_n z^n,
\label{eqn:GAF_D}
\end{equation}
which defines an analytic function on $\D$ a.s.
This gives a GAF on $\D$ with a covariance kernel 
\begin{equation}
\bE[X_{\D}(z) \overline{X_{\D}(w)}]
= \frac{1}{1-z \wbar}=:S_{\D}(z, w), \quad z, w \in \D.
\label{eqn:SD}
\end{equation}
This kernel is identified with 
the reproducing kernel of the Hardy space $H^2(\D)$
called the \textit{Szeg\H{o} kernel} of $\D$
\cite{Neh52,Ber70,AM02,Bel16}.
Peres and Vir\'ag \cite{PV05} proved that 
$\cZ_{X_{\D}}$ is a \textit{determinantal point process} (DPP) 
such that the correlation kernel 
is given by 
$S_{\D}(z,w)^2=(1-z \wbar)^{-2}$, 
$z,w \in \D$ 
with respect to the reference measure 
$\lambda=m/\pi$.
Here $m$ represents the Lebesgue measure on $\C$; 
$m(dz) :=dx dy$, $z=x+\sqrt{-1} y \in \C$.
(See Theorem \ref{thm:Peres_Virag} 
in Section \ref{sec:Peres_Virag} below). 
This correlation kernel is identified with
the reproducing kernel of the Bergman space on $\D$,
which is called the \textit{Bergman kernel} of $\D$ and 
denoted here by 
$K_{\D}(z, w), z, w \in \D$
\cite{Neh52,Ber70,HKZ00,AM02,Bel16}.
Thus the study of Peres and Vir\'ag on
$X_{\D}$ and $\cZ_{X_{\D}}$
is associated with the following relationship 
between kernels on $\D$ \cite{PV05},
\begin{equation}
\bE[X_{\D}(z) \overline{X_{\D}(w)}]^2
=S_{\D}(z, w)^2
=K_{\D}(z, w), 
\quad z, w \in \D.
\label{eqn:SKC_D}
\end{equation}
(A brief review of reproducing kernels will be given 
in Section \ref{sec:RK}.) 

Let $q \in (0, 1)$ be a fixed number and 
we consider the annulus
$\A_q:=\{z \in \C: q < |z| < 1 \}$. 
In the present paper we will report 
the fact that, if we consider a GAF 
given by the i.i.d.~Gaussian Laurent series $X_{\A_q}$ on $\A_q$, 
we will observe interesting new phenomena
related with $X_{\A_q}$ 
and its zero point process $\cZ_{X_{\A_q}}$. 
The present results are reduced to those by 
Peres and Vir\'ag \cite{PV05} in the limit $q \to 0$.
Conversely, the point processes
associated with $X_{\D}$ are extended to
those associated with $X_{\A_q}$ in this paper.
The obtained new point processes can be regarded
as \textit{elliptic extensions}
of the previous ones, 
since expressions for the former given by
polynomials and rational functions
of arguments are replaced by those of
the theta functions with the arguments and 
the nome $p=q^2$ for the latter
\cite{TV97,Spi02,War02,KN03,GR04,Kra05,RS06,Rai10}. 
Moreover, we will introduce another parameter $r >0$
in addition to $q$, 
and one-parameter families of
GAFs, $\{X_{\A_q}^r : r >0 \}$ and zero point processes, 
$\{\cZ_{X_{\A_q}^r} : r > 0\}$ 
will be constructed on $\A_q$. 
Here put $X_{\A_q} := X_{\A_q}^q$ and $\cZ_{X_{\A_q}} := \cZ_{X_{\A_q}^q}$.
Construction of a model on an annulus
will serve as a solid starting point for arguing
general theory on multiply connected domains.
Even if the models are different, studies in this direction
provide useful hints for us to proceed the generalization
\cite{Pee93,Car02,Zha04,Car06,For06,BF08,HL08,
HBB10,Izy17,Rem18,Kat19b,BKT18}.

Consider the Hilbert space of analytic functions on $\A_q$ 
equipped with the inner product
\[
\bra f, g \ket_{H^2_r(\A_q)} = 
\frac{1}{2\pi} \int_{\gamma_1 \cup \gamma_q}
f(z) \overline{g(z)} \sigma_r(dz),
\quad f, g \in H^2_r(\A_q)
\]
with
\[
\sigma_r(d z)
=\begin{cases}
d \phi, & \mbox{if $z \in \gamma_1
:=\{e^{\sqrt{-1} \phi}: \phi \in [0, 2 \pi) \}$},
\cr
r d \phi, & \mbox{if $z \in \gamma_q 
:=\{q e^{\sqrt{-1} \phi}: \phi \in [0, 2 \pi) \}$},
\end{cases}
\]
which we write as $H_r^2(\A_q)$.
A complete orthonormal system (CONS) of $H_r^2(\A_q)$ 
is given by $\{e^{(q, r)}_n\}_{n \in \Z}$ with
\[
e^{(q, r)}_n(z) = \frac{z^n}{\sqrt{1+r q^{2n}}},
\quad z \in \A_q, \quad n \in \Z,
\]
and the reproducing kernel is given by \cite{MS94}
\begin{equation}
S_{\A_q}(z, w; r)
= \sum_{n \in \Z} e^{(q,r)}_n(z) \overline{e^{(q,r)}_n(w)}
= \sum_{n=-\infty}^{\infty} 
\frac{(z \wbar)^n}{1+r q^{2n}}. 
\label{eqn:SAqr1}
\end{equation}
This infinite series converges absolutely for $z,w \in \A_q$.
When $r=q$, this Hilbert function space is known as
the Hardy space on $\A_q$
denoted by $H^2(\A_q)$ and the
reproducing kernel 
$S_{\A_q}(\cdot, \cdot) :=S_{\A_q}(\cdot, \cdot; q)$ is called the
Szeg\H{o} kernel of $\A_q$ \cite{Neh52,Sar65}.
The kernel (\ref{eqn:SAqr1}) 
with a parameter $r >0$ is considered as a 
\textit{weighted Szeg\H{o} kernel} of $\A_q$ \cite{Neh52b}
and $H_r^2(\A_q)$ is the
\textit{reproducing kernel Hilbert space}
(RKHS) \cite{Aro50} with respect to 
$S_{\A_q}(\cdot, \cdot; r)$ \cite{MS94,MS12}.
We call $r$ the \textit{weight parameter} 
in this paper. 
We note that (\ref{eqn:SAqr1}) implies that
$S_{\A_q}(z, z; r)$ is a monotonically decreasing
function of the weight parameter
$r \in (0, \infty)$ for each fixed
$z \in \A_q$.

Associated with $H_r^2(\A_q)$, 
we consider the Gaussian Laurent series 
\begin{equation}
X_{\A_q}^r(z):= \sum_{n \in \Z} \zeta_n e^{(q, r)}_n(z)
=\sum_{n=-\infty}^{\infty}
\zeta_n \frac{z^n}{\sqrt{1+r q^{2n}}},
\label{eqn:GAF_Aqr}
\end{equation}
where $\{\zeta_n\}_{n \in \Z}$ are i.i.d.~standard complex
Gaussian random variables with density $e^{-|z|^2}/\pi$. 
Since $\lim_{n \to \infty} |\zeta_n|^{1/n} = 1$ a.s., 
we apply the Cauchy--Hadamard
criterion to the positive 
and negative powers of $X_{\A_q}^r(z)$ separately
to conclude that this random Laurent series converges 
a.s. whenever $z \in \A_q$. 
Moreover, since the distribution $\zeta_n$ is symmetric, 
both of $\gamma_1$ and $\gamma_q$ are natural 
boundaries \cite[p.40]{K85}.
Hence $X_{\A_q}^r$ provides a GAF on $\A_q$ whose 
covariance kernel is given by 
the weighted Szeg\H{o} kernel of $\A_q$, 
\[
\bE[X_{\A_q}^r(z) \overline{X_{\A_q}^r(w)} ]
=S_{\A_q}(z, w; r), \quad z, w \in \A_q,
\]
and the zero point process is denoted by
$\cZ_{X_{\A_q}^r} :=\sum_{z \in \A_q: X^r_{\A_q}(z)=0} \delta_{z}$.
In particular, we write
$X_{\A_q}(z) :=X_{\A_q}^q(z), z \in \A_q$ and
$\cZ_{X_{\A_q}} := \cZ_{X_{\A_q}^q}$
as mentioned above. 

We recall \textit{Schottky's theorem}
(see, for instance, \cite{AIMO08}):
The group of conformal 
(i.e., angle-preserving one-to-one) transformations
from $\A_q$ to itself 
is generated by the rotations 
and the $q$-inversions $T_q(z) := q/z$. 
The invariance of the present 
GAF and its zero point process under rotation is obvious. 
Using the properties of $S_{\A_q}$, 
we can prove the following.
\begin{prop}
\label{thm:conformal_Aq}
\begin{description}
\item{\rm (i)} 
The GAF $X_{\A_q}^r$ given by (\ref{eqn:GAF_Aqr}) 
has the $(q, r)$-inversion symmetry in the sense that
\[
\Big\{ (T_q'(z))^{1/2} X_{\A_q}^r (T_q(z)) \Big\}
\dis=
\Big\{ \sqrt{\frac{q}{r}} X_{\A_q}^{q^2/r}(z) \Big\},
\quad z \in \A_q,
\]
where $T_q'(z) := \frac{d T_q}{dz}(z)=-q/z^2$.
\item{\rm (ii)} 
For $\cZ_{X_{\A_q}^r} = \sum_i \delta_{Z_i}$,
let 
$T_q^{*}\cZ_{X_{\A_q}^r} := \sum_i \delta_{T_q^{-1}(Z_i)}$.
Then
$T_q^{*}\cZ_{X_{\A_q}^r} \dis= \cZ_{X_{\A_q}^{q^2/r}}$.
\item{\rm (iii)} 
In particular, when $r=q$, 
the GAF $X_{\A_q}$ is invariant under
conformal transformations which preserve $\A_q$,
and so is its zero point process $\cZ_{X_{\A_q}}$.
\end{description}
\end{prop}
\noindent
This result should be compared with
the conformal invariance of the 
DPP of Peres and Vir\'ag on $\D$
stated as Proposition \ref{thm:Peres_Virag2}
in Section \ref{sec:Peres_Virag} below. 
The proof of Proposition \ref{thm:conformal_Aq}
is given in Section \ref{sec:proof_conformal_Aq}.

\begin{rem}
\label{rem:parameter_L}
Note that $(T_q'(z))^{1/2}=\sqrt{-1} q^{1/2}/z$ is single valued
and non-vanishing in $\A_q$, and so is $(T_q'(z))^{L/2}$
if $L \in \N$.
By the calculation given in Section \ref{sec:proof_conformal_Aq},
we have the equality,
\[
(T_q'(z))^{L/2} \overline{(T_q'(w))^{L/2}}
S_{\A_q}(T_q(z), T_q(w); r)^L
=\left( \frac{q}{r} \right)^L
S_{\A_q}(z, w; q^2/r)^L.
\]
We define $X_{\A_q}^{r, (L)}$ as
the centered GAF with the covariance kernel
$S_{\A_q}(z,w; r)^L$ on $\A_q$, $L \in \N$.
Then it is rotationally invariant and having the 
$(q, r)$-inversion symmetry in the sense 
\[
\Big\{ (T_q'(z))^{L/2} X_{\A_q}^{r, (L)} (T_q(z)) \Big\}
\dis=
\Big\{ \left(\frac{q}{r}\right)^{L/2} X_{\A_q}^{q^2/r, (L)}(z) \Big\},
\quad z \in \A_q.
\]
This implies that the zero point process of $X_{\A_q}^{r, (L)}$
is also rotationally invariant and symmetric under
the $(q, r)$-inversion.
In particular, the GAF $X_{\A_q}^{(L)}:=X_{\A_q}^{q, (L)}$
and its zero point process are invariant under 
conformal transformations which preserve $\A_q$.
By definition $X_{\A_q}^{(1)}=X_{\A_q}$
given by (\ref{eqn:GAF_Aqr}) with $r=q$.
The formula (\ref{eqn:K_Aq}) 
and Proposition \ref{thm:relation}
in Appendix \ref{sec:KAq_SAq}
imply that $X_{\A_q}^{(2)}$ is realized by
$X_{\A_q}^{(2)}(z)= \sum_{n \in \Z} \zeta_n c^{(2)}_n z^n$,
$z \in \A_q$, 
where 
$c^{(2)}_{-1}=\sqrt{a-1/(2 \log q)}$
with $a=a(q)$ given by (\ref{eqn:aq}),
$c^{(2)}_n=\sqrt{(n+1)/(1-q^{2(n+1)})}$, $n \in \Z \setminus \{-1\}$, 
and $\{\zeta_n\}_{n \in \Z}$ are i.i.d.~standard complex 
Gaussian random variables with density $e^{-|z|^2}/\pi$.
We do not know explicit expressions for
the Gaussian Laurent series of $X_{\A_q}^{(L)}$ for
$L=3, 4, \dots$, but it is expected that
$\lim_{q \to 0} X_{\A_q}^{(L)}(z) \dis= X_{\D}^{(L)}(z)
:=\sum_{n \in \N_0} \zeta_n 
\frac{\sqrt{L(L+1) \cdots (L+n-1)}}{\sqrt{n!}} z^n$, $z \in \D$, 
and $X_{\A_q}^{(1)}$ and $X_{\A_q}^{(2)}$ given above
indeed satisfy such limit transitions.
Here $\{X_{\D}^{(L)} : L >0 \}$ is the family of GAFs on $\D$
studied in \cite[Sections 2.3 and 5.4]{HKPV09}
which are invariant under conformal transformations
mapping $\D$ to itself.
\end{rem}

Let $\theta(\cdot) :=\theta(\cdot; q^2)$ 
be the theta function, 
whose definition and basic properties
are given in Section \ref{sec:S_qt_theta}.
Following the standard way \cite{GR04,RS06}, 
we put $\theta(z_1, \dots, z_n):=\prod_{i=1}^n \theta(z_{i})$.
Then (\ref{eqn:SAqr1}) is expressed as \cite{MS94}
\begin{equation}
S_{\A_q}(z, w; r)
= \frac{q_0^2 \theta(-r z\wbar)}{\theta(-r, z \wbar)},
\quad z, w \in \A_q
\label{eqn:SAqt2}
\end{equation}
with $q_0:=\prod_{n=1}^{\infty} (1-q^{2n})$,
as proved in Section \ref{sec:S_qt_JK}.

\begin{rem}
\label{rem:norm_r}
Consider an operator $(U_q f)(z) :=f(q^2 z)$
acting on holomorphic functions $f$ on 
$\C^{\times}$. 
For $n \in \N$, 
Rosengren and Schlosser \cite{RS06} called 
$f$ an \textit{$A_{n-1}$ theta function of norm} 
$a \in \C^{\times}$ if
\[
(U_q f)(z)= \frac{(-1)^n}{a z^n} f(z).
\]
It is shown that $f$ is an $A_{n-1}$ theta function of norm $a$
if and only if there exist $C$, $b_1, \dots, b_n$ such that
$\prod_{\ell=1}^n b_{\ell}=a$ and
$f(z)=C \theta(b_1 z, \dots, b_n z)$ 
\cite[Lemma 3.2]{RS06}.
In the following, 
given $n$ points $z_1, \dots, z_n \in \A_q$, 
we will evaluate the 
weighted Szeg\H{o} kernel 
at these points.
In this case, the weight parameter
$r$ for $H_r^2(\A_q)$
can be related to a norm 
for $A_{n-1}$ theta functions as explained below.
Put $a=-r \prod_{\ell=1}^n \overline{z_{\ell}}$
and let $\Theta^{(n, a)}_j(z) :=C \theta(-r z \overline{z_j})
\prod_{1 \leq \ell \leq n, \ell \not=j} \theta(z \overline{z_{\ell}})$,
$z \in \A_q$, 
$j =1, \dots, n$. Then 
$\{\Theta^{(n, a)}_j(z) \}_{j=1}^n$ form a basis of
the $n$-dimensional space of the $A_{n-1}$ theta functions of
norm $a$.
If we choose 
$C=q_0^2/\theta(-r)$, then 
evaluations of the weighted Szeg\H{o} kernel 
at the $n$ points are
expressed as
$S_{\A_q}(z_i, z_j; r)=\Theta^{(n, a)}_j(z_i)/\prod_{\ell=1}^n
\theta(z_i \overline{z_{\ell}})$, $i = 1, \dots, n$. 
Multivariate extensions of such elliptic function spaces
were studied in \cite{TV97}.
\end{rem}

\subsection{Mccullough-Shen formula for the
conditional Szeg\H{o} kernel}
\label{sec:MS_formula}

For any non-empty set $D$, given a positive definite kernel 
$k(z,w)$ on $D \times
D$, we can define a centered Gaussian process on $D$, $X_D$, 
such that the covariance kernel is given by
$\bE[X_D(z) \overline{X_D(w)}]=k(z, w)$,
$z, w \in D$. 
The kernel $k$ induces RKHS 
$\cH_k$ realized as a
function space having $k$  
as the reproducing kernel \cite{Aro50}. 
Now we define a \textit{conditional kernel} 
\begin{equation}
 k^{\alpha}(z,w) = k(z,w) 
 - \frac{k(z,\alpha) k(\alpha, w)}{k(\alpha, \alpha)}, 
 \quad z, w \in D, 
\label{eqn:conditionS}
\end{equation}
for $\alpha \in D$ such that 
$k(\alpha, \alpha) > 0$. 
Then, $k^{\alpha}$ is a reproducing kernel for the
Hilbert subspace 
$\cH_{k}^{\alpha} := \{f \in \cH_{k} :f(\alpha)=0\}$.
The corresponding centered Gaussian process on $D$
whose covariance kernel is given by 
$k^{\alpha}$ is equal in law to 
$X_D$ given that $X_D(\alpha)=0$. 

We can verify that if $D \subsetneq \C$ 
is a simply connected domain 
with $\cC^{\infty}$ smooth boundary
and the Szeg\H{o} kernel $S_D$ can be defined on it,
Riemann's mapping theorem implies 
the equality \cite{Ahl79,Bel95} 
\begin{equation}
S_D^{\alpha}(z, w)=S_D(z,w) h_{\alpha}(z) \overline{h_{\alpha}(w)},
\quad z, w, \alpha \in D,
\label{eqn:SaD}
\end{equation}
where $h_{\alpha}$ is the Riemann mapping function;
the unique conformal map from $D$ to $\D$
satisfying $h_{\alpha}(\alpha)=0$ and $h_{\alpha}'(\alpha) >0$.
Actually (\ref{eqn:SaD}) is equivalent with
(\ref{eqn:SD_formula}) derived 
from Riemann's mapping theorem in Section \ref{sec:RK} below.
In particular, when $D=\D$, $h_{\alpha}$ is 
the M\"{o}bius transformation
$\D \to \D$ sending $\alpha$ to the origin,
\begin{equation}
h_{\alpha}(z)=\frac{z-\alpha}{1-\overline{\alpha} z}
=z \frac{1-\alpha/z}{1-\alphabar z},
\quad z, \alpha \in \D.
\label{eqn:Mob1}
\end{equation}

\begin{figure}[ht]
\begin{center}
\includegraphics[scale=0.3]{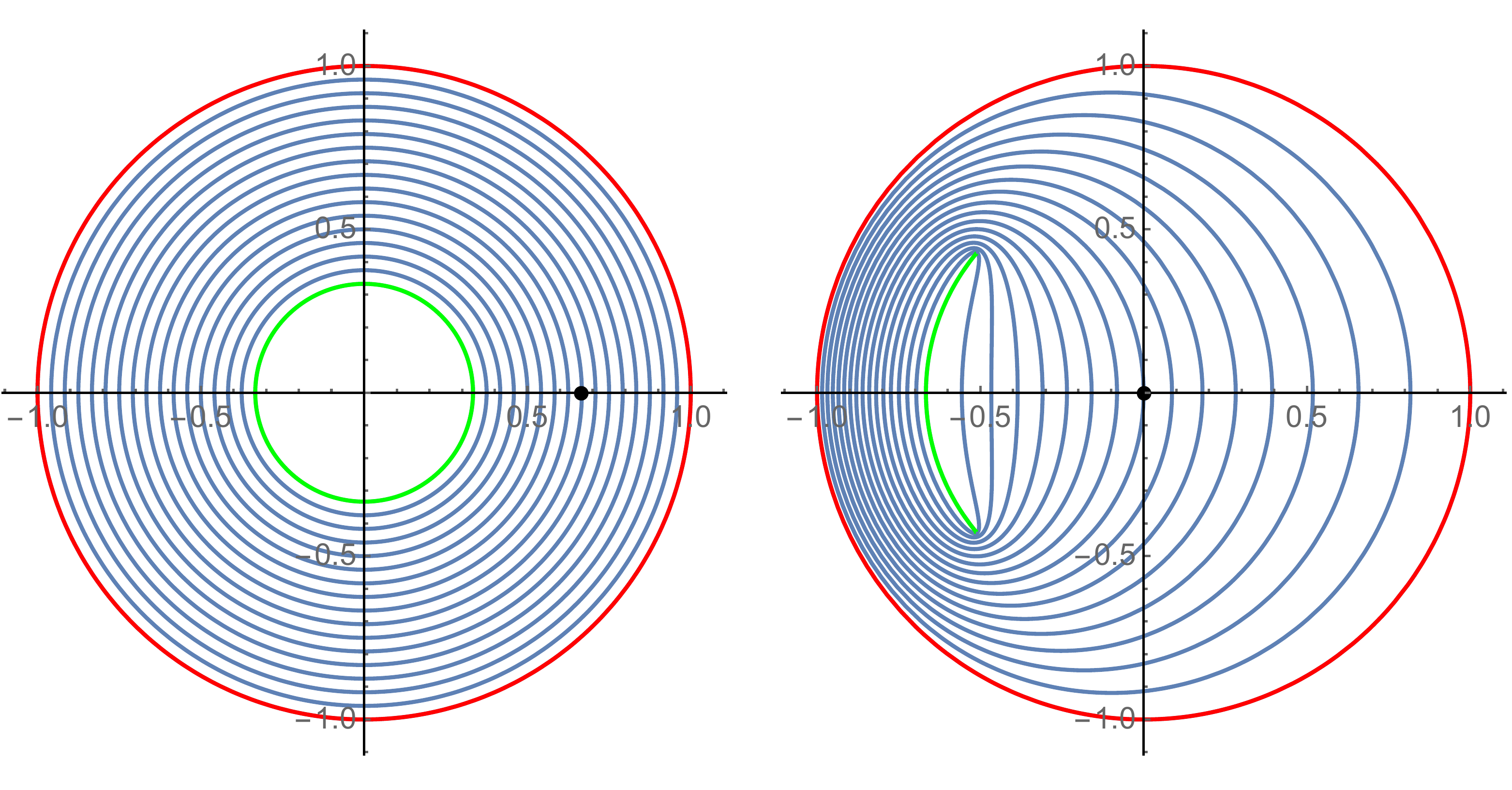}
\end{center}
\caption{Conformal map $h^{q}_{\alpha} : 
\A_{q} \to \D \setminus \{\mbox{a circular slit}\}$
is illustrated for $q=1/3$ and $\alpha=2/3$. 
The point $\alpha=2/3$ in $\A_{1/3}$ 
is mapped to the origin.
The outer boundary $\gamma_1$ of $\A_{1/3}$ 
(denoted by a red circle) is mapped to a unit circle
(a red circle) making the boundary of $\D$. 
The inner boundary $\gamma_{1/3}$ of $\A_{1/3}$
(a green circle) is mapped to a circular slit (denoted by a green arc)
which is a part of the circle with radius $\alpha=2/3$,
where the map is two-to-one
except the two points on $\gamma_{1/3}$
mapped to the two edges of the circular slit.}
\label{fig:h_alpha}
\end{figure}

Since the theta function $\theta(z)$ can be regarded
as an elliptic extension of $1-z$ as suggested by
the formula $\lim_{q \to 0} \theta(z; q^2)=1-z$
given by (\ref{eqn:theta_p0}) below, 
we can think of the following function
as an elliptic extension of (\ref{eqn:Mob1});
\begin{equation}
h_{\alpha}^q(z):=
z \frac{\theta(\alpha/z)}
{\theta(\alphabar z)}
=- \alpha \frac{\theta(z/\alpha)}{\theta(z \alphabar)}, 
\quad z, \alpha \in \A_q.
\label{eqn:haq}
\end{equation}
We can prove that
$h_{\alpha}^q$ is identified with a conformal map
from $\A_{q}$ to the unit disk with a circular slit in it,
in which $\alpha \in \A_q$ is sent to the origin \cite{MS94}.
See Figure \ref{fig:h_alpha} 
and Lemma \ref{thm:Blaschke}
in Section \ref{sec:conditional_S}.
Mccullough and Shen proved the equality 
\begin{equation}
S_{\A_q}^{\alpha}(z, w; r)
=S_{\A_q}(z, w; r |\alpha|^2)
h_{\alpha}^q(z) \overline{h_{\alpha}^q(w)},
\quad z, w, \alpha \in \A_q, 
\label{eqn:SaAq}
\end{equation}
as an extension of (\ref{eqn:SaD}) \cite{MS94}.
See Section \ref{sec:conditional_S} below
for a direct proof of this equality by
Weierstrass' addition formula 
of the theta function (\ref{eqn:Weierstrass_add1}).
Up to the factor $h_{\alpha}^q(z) \overline{h_{\alpha}^q(w)}$
the conditional kernel $S_{\A_q}^{\alpha}(z, w; r)$ 
remains the weighted Szeg\H{o} kernel, but 
the weight parameter should be changed from $r$ to
$r |\alpha|^2$.

Following (\ref{eqn:conditionS}), conditional kernels
$k^{\alpha_1, \dots, \alpha_n}$ are inductively defined as
\begin{equation}
k^{\alpha_1, \dots, \alpha_n}(z, w)
=(k^{\alpha_1, \dots, \alpha_{n-1}})^{\alpha_n}(z,w),
\quad
z, w, 
\alpha_1, \dots, \alpha_n \in D,
\quad n=2,3,\dots.
\label{eqn:conditionS2}
\end{equation}
The kernels $k^{\alpha_1, \dots, \alpha_n}$,
$n=2, 3, \dots$, 
will construct Hilbert subspaces
$\cH_{k}^{\alpha_1, \dots, \alpha_n} 
:= \{f \in \cH_{k} :f(\alpha_1)=\cdots=f(\alpha_n)=0\}$.

For $n \in \N$, $\alpha_1, \dots, \alpha_n \in \A_q$, define
\begin{equation}
\gamma^q_{\{\alpha_{\ell}\}_{\ell=1}^n}(z) 
:= \prod_{\ell=1}^n h_{\alpha_{\ell}}^q(z),
\quad z \in \A_q.
\label{eqn:gamma}
\end{equation}
Then the Mccullough and Shen formula (\ref{eqn:SaAq}) 
\cite{MS94} is generalized as 
\begin{equation}
S_{\A_q}^{\alpha_1, \dots, \alpha_n}(z, w; r)
=S_{\A_q} \Big(z, w; r \prod_{\ell=1}^n |\alpha_{\ell}|^2 \Big)
\gamma^q_{\{\alpha_{\ell}\}_{\ell=1}^n}(z) 
\overline{
\gamma^q_{\{\alpha_{\ell}\}_{\ell=1}^n}(w)
}, \quad z, w \in \A_q,
\label{eqn:MS_general}
\end{equation}
for $n \in \N$, $\alpha_1, \dots, \alpha_n \in \A_q$. 
We can give probabilistic interpretations
of the above facts as follows.
\begin{prop}
\label{thm:equivalence2}
For any $\alpha_1, \dots, \alpha_n \in \A_q$, $n \in \N$, 
the following hold. 
\begin{description}
\item{\rm (i)} \, 
The following equality is established,
\begin{align*}
&\mbox{$\{X_{\A_q}^r(z) : z \in \A_q \}$ 
given $\{X_{\A_q}^r(\alpha_1)= \cdots =X_{\A_q}^r(\alpha_n)=0\}$}
\nonumber\\
& \hskip 4cm
\dis= 
\left\{ \gamma^q_{\{ \alpha_{\ell} \}_{\ell=1}^n}(z) 
X_{\A_q}^{r \prod_{\ell=1}^n |\alpha_{\ell}|^2}(z)
 : z \in \A_q \right\}.
\end{align*}
\item{\rm (ii)} \, 
Let $\cZ_{X_{\A_q}^r}^{\alpha_1, \dots, \alpha_n}$ 
denote the zero point process
of the GAF $X_{\A_q}^r(z)$
given $\{X_{\A_q}^r(\alpha_1)=\cdots = X_{\A_q}^r(\alpha_n)=0\}$. 
Then, 
$\cZ_{X_{\A_q}^r}^{\alpha_1, \dots, \alpha_n} 
\dis= \cZ_{X_{\A_q}^{r \prod_{\ell=1}^n |\alpha_{\ell}|^2}}
+ \sum_{i=1}^n \delta_{\alpha_i}$.
\end{description}
\end{prop}

\begin{rem}
\label{rem:Remark_condition1}
For the GAF on $\D$
studied by Peres and Vir\'ag \cite{PV05},
$\{X_{\D}(z) : z \in \D \}$ 
given $\{X_{\D}(\alpha)=0\}$
is equal in law to $\{h_{\alpha}(z) X_{\D}(z) : z \in \D \}$, 
$\forall \alpha \in \D$, 
where $h_{\alpha}$ is given by (\ref{eqn:Mob1}),
and then, in the notation 
used in Proposition \ref{thm:equivalence2}, 
$\cZ_{X_{\D}}^{\alpha} \dis= \cZ_{X_{\D}}+\delta_{\alpha}$, 
$\forall \alpha \in \D$.
Hence, no new GAF nor new zero point process 
appear by conditioning of zeros.
For the present GAF on $\A_q$, however, 
conditioning of zeros induces new GAFs and 
new zero point processes
as shown by Proposition \ref{thm:equivalence2}.
Actually, by (\ref{eqn:SAqr1}) 
the covariance of the induced GAF
$X_{\A_q}^{r \prod_{\ell=1}^n |\alpha_{\ell}|^2}$ 
is expressed by
$S_{\A_q}(z, w; r \prod_{\ell=1}^n |\alpha_{\ell}|^2)
=\sum_{n=-\infty}^{\infty} (z \wbar)^n/
(1+r \prod_{\ell=1}^n |\alpha_{\ell}|^2 q^{2n})$.
Since $q < |\alpha_{\ell}| <1$,
as increasing the number of conditioning zeros, 
the variance of induced GAF monotonically
increases, in which the increment is a decreasing function of
$|\alpha_{\ell}| \in (q, 1)$. 
\end{rem}

\subsection{Correlation functions of the zero point process}
\label{sec:correlations}
We introduce the following 
notation. For an $n \times n$ matrix
$M=(m_{ij})_{1 \leq i, j \leq n}$,
\begin{equation}
\perdet M = \perdet_{1 \leq i, j \leq n} [m_{ij}]
:=\per M \det M,
\label{eqn:perdet}
\end{equation}
that is, $\perdet M$ denotes
$\per M$ multiplied by $\det M$.
Note that
$\perdet$ is a special case
of \textit{hyperdeterminants} introduced
by Gegenbauer following Cayley
(see \cite{Mat08,EG09,LV10} and references therein). 
If $M$ is a positive semidefinite hermitian matrix, then
$\per M \geq \det M \geq 0$ 
\cite[Section II.4]{MM92}
\cite[Theorem 4.2]{Min78},
and hence $\perdet M \geq 0$ by the definition (\ref{eqn:perdet}).

The following will be proved in Section \ref{sec:proof_mainA1}.

\begin{thm}
\label{thm:mainA1}
Consider the zero point process $\cZ_{X_{\A_q}^r}$ 
on $\A_q$. 
Then, it is a permanental-determinantal point process 
(PDPP) in the sense that 
it has correlation functions
$\{\rho^{n}_{\A_q}\}_{n \in \N}$ 
given by
\begin{equation}
\rho^{n}_{\A_q}(z_1,\dots,z_n; r) 
=
\frac{\theta(-r)}{\theta( -r \prod_{k=1}^n |z_k|^4)}
\perdet_{1 \leq i, j \leq n}
\Big[
S_{\A_q} \Big(z_{i}, z_{j}; r \prod_{\ell=1}^n |z_{\ell}|^2 \Big) 
\Big]
\label{eqn:rho_mainA1}
\end{equation}
for every $n \in \N$ and
$z_1, \dots, z_n \in \A_q$
with respect to $m/\pi$.
\end{thm}
\noindent
In Appendix \ref{sec:hDPP} we rewrite this theorem
using the notion of hyperdeterminants 
(Theorem \ref{thm:main_hDPP}).

\begin{rem}\label{rem:PDPP_property}
(i) The PDPP with correlation functions \eqref{eqn:rho_mainA1} 
turns out to be a \textit{simple point process}, 
i.e., there is no multiple point a.s., 
due to the existence of two-point correlation 
function with respect to
the Lebesgue measure $m/\pi$ \cite[Lemma 2.7]{Kal17}. 
(ii) Using the explicit expression (\ref{eqn:rho_mainA1}) together 
with the Frobenius determinantal formula
\eqref{eqn:Frobenius_formula}, 
we can verify that for every $n \in \N$, 
the $n$-point correlation $\rho^n_{\A_q}(z_1, \dots, z_n) >0$
if all coordinates $z_1, \dots, z_n \in \A_q$ are 
different from each other, 
and that $\rho_{\A_q}^n(z_1, \dots, z_n) =0$ if
some of  $z_1, \dots, z_n$ coincide; 
e.g., $z_i = z_j, i \not=j$,  
by the determinantal factor in $\perdet$ (\ref{eqn:perdet}).
\end{rem}

\begin{rem}
\label{rem:pdet_correlation}
The determinantal point processes (DPPs) 
and the permanental point processes (PPPs)
have the $n$-correlation functions of the forms  
\begin{align*}
\rho_{\mathrm{DPP}}^n(z_1,\dots,z_n) 
= \det_{1 \le i,j \le n}[K(z_i,z_j)], \quad 
\rho_{\mathrm{PPP}}^n(z_1,\dots,z_n) 
= \per_{1 \le i,j \le n}[K(z_i,z_j)], 
\end{align*}
respectively (cf. \cite{ST03a, MM06, HKPV09}). 
Due to Hadamard's inequality for the determinant \cite[Section II.4]{MM92}
and Lieb's inequality for the permanent \cite{Lie66}, we have 
\begin{align*}
 \rho_{\mathrm{DPP}}^2(z_1,z_2) \le \rho_{\mathrm{DPP}}^1(z_1) 
 \rho_{\mathrm{DPP}}^1(z_2), \quad 
 \rho_{\mathrm{PPP}}^2(z_1,z_2) \ge \rho_{\mathrm{PPP}}^1(z_1) 
 \rho_{\mathrm{PPP}}^1(z_2), 
\end{align*}
in other words, the unfolded $2$-correlation functions 
are $\le 1$ or $\ge 1$, respectively (see Section~\ref{sec:two_point}).
These correlation inequalities suggest a repulsive nature
(negative correlation) for DPPs and 
an attractive nature (positive correlation) for PPPs. 
Some related topics are discussed in \cite{Shi07}. 
Since $\perdet$ is considered to have intermediate nature 
between determinant and permanent, 
PDPPs are expected to exhibit both repulsive and attractive characters, 
depending on the position of two points $z_1$ and $z_2$. 
For example, Remark~\ref{rem:PDPP_property} (ii) shows 
the repulsive nature inherited from the DPP side. 
The two-sidedness of the present PDPP will be 
clearly described in Theorem~\ref{thm:critical_line} given below. 
\end{rem}
\begin{rem}
\label{rem:Remark_condition2}
Since correlation functions are transformed
as in Lemma \ref{thm:transformation_rho}
given in Section \ref{sec:Peres_Virag}, 
Proposition \ref{thm:conformal_Aq} (ii) is 
rephrased using correlation functions as
\begin{equation}
\rho^{n}_{\A_q}(T_q(z_1),\dots, T_q(z_n); r) 
\prod_{\ell=1}^n |T_q'(z_{\ell})|^2
= \rho^{n}_{\A_q}(z_1, \dots, z_n; q^2/r)
\label{eqn:rho_inversion0}
\end{equation}
for any $n \in \N$ and $z_1, \dots, z_n \in \A_q$,
where $T_q(z)=q/z$ and $|T_q'(z)|^2=q^2/|z|^4$.
In the correlation functions 
$\{\rho_{\A_q}^n\}_{n \in \N}$ given by
Theorem \ref{thm:mainA1}, we see 
an \textit{inductive structure} such that
the functional form of the 
permanental-determinantal correlation kernel
$S_{\A_q}(\cdot, \cdot; r \prod_{\ell=1}^n |z_{\ell}|^2)$ 
is depending on the
points $\{z_1, \dots, z_n\}$, which we intend to measure
by $\rho_{\A_q}^n$, via the weight parameter
$r \prod_{\ell=1}^n |z_{\ell}|^2$.
This is due to the inductive structure 
of the induced GAFs generated in conditioning of zeros as
mentioned in Remark \ref{rem:Remark_condition1}.
In addition, the reference measure $m/\pi$ is
also weighted by 
$\theta(-r)/\theta(-r \prod_{k=1}^n |z_{k}|^4)$.
As demonstrated by a direct proof of 
(\ref{eqn:rho_inversion0}) given in 
Section \ref{sec:proof_inversion_symmetry},
such a \textit{hierarchical structure}
of correlation functions and reference measures 
is necessary to realize
the $(q,r)$-inversion symmetry (\ref{eqn:rho_inversion0})
(and the invariance under conformal transformations
preserving $\A_q$ when $r=q$). 
\end{rem}
\begin{rem}
\label{rem:pair_zero}
The nonexistence of zero in $\D$ 
of $S_{\D}(\cdot, \alpha), \alpha \in \D$ and the
uniqueness of zero in $\A_q$ of $S_{\A_q}(\cdot, \alpha)$, 
$\alpha \in \A_q$ are concluded from a general consideration
(see, for instance, \cite[Section 27]{Bel16}).
Define
\begin{equation}
\alphahat := - \frac{q}{\alphabar}, \quad \alpha \in \A_q.
\label{eqn:alphabar}
\end{equation}
The fact
$S_{\A_q}(\alphahat, \alpha)=0$, $\alpha \in \A_q$ was proved 
as Theorem 1 in \cite{TT99} by direct calculation, 
for which a simpler proof will be given 
below (Lemma \ref{thm:zero_S})
using theta functions (\ref{eqn:SAqt2}). 
For the GAF $X_{\D}$ studied by Peres and Vir\'ag \cite{PV05},
all points in $\D$ are correlated, while 
the GAF $X_{\A_q}$ 
has a pair structure of independent points
$\{ \{\alpha, \alphahat \} : \alpha \in \A_q\}$
(Proposition \ref{thm:pair_structure}). 
As a special case of (\ref{eqn:MS_general}), we have
\[
S_{\A_q}^{\alpha, \alphahat}(z, w)
=S_{\A_q}(z, w) f^q_{\alpha}(z) \overline{f^q_{\alpha}(w)},
\quad z, w \in \A_q
\]
with
\begin{align*}
f^q_{\alpha}(z) &:= \frac{1}{z} h^q_{\alpha}(z) h^q_{\alphahat}(z)
\nonumber\\
&=z 
\frac{\theta(-q z \alphabar, \alpha/z)}
{\theta(-q z/\alpha, \alphabar z)}
=-\alpha \frac{\theta(-q z \alphabar, z/\alpha)}
{\theta(-qz/\alpha, z \alphabar)}.
\end{align*}
We notice that $f^q_{\alpha}$ is identified with the
\textit{Ahlfors map} from $\A_q$ to $\D$, 
that is, it is holomorphic and
gives the two-to-one map 
from $\A_q$ to $\D$ satisfying
$f^q_{\alpha}(\alpha)=f^q_{\alpha}(\alphahat)=0$.
The Ahlfors map has been extensively studied 
(see, for instance, \cite[Section 13]{Bel16}),
and the above explicit expression using theta functions
will be useful.
We can verify that if we especially consider 
the $2n$-correlation of $n$-pairs 
$\{\{z_i, \widehat{z_i}\} \}_{i=1}^n$
of points in the zero point process $\cZ_{X_{\A_q}}$,
the hierarchical structure mentioned above vanishes
and the formula (\ref{eqn:rho_mainA1}) of Theorem \ref{thm:mainA1}
is simplified as
\[
\rho^{2n}_{\A_q}(z_1, \widehat{z_1}, \dots, z_n, \widehat{z_n}; q) 
\nonumber\\
=\perdet_{1 \leq i, j \leq n}
\begin{bmatrix}
S_{\A_q}(z_i, z_j) 
& S_{\A_q}(z_i, \widehat{z_j}) \\
S_{\A_q}(\widehat{z_i}, z_j) 
& S_{\A_q}(\widehat{z_i}, \widehat{z_j})
\end{bmatrix}
\]
for any $n \in \N$.
In other words,  we need the hierarchical structure
of correlation functions and reference measure
in order to describe the probability distributions of 
general configurations of the 
zero point process $\cZ_{X_{\A_q}^r}$.
\end{rem}

The density of zeros on $\A_q$ 
with respect to $m/\pi$ is given by 
\begin{equation}
\rho^1_{\A_q}(z; r) = 
\frac{\theta(-r)}{\theta(-r |z|^4)}
S_{\A_q}(z,z; r|z|^2)^2
=
\frac{q_0^4 \theta(-r, -r |z|^4)}
{\theta(-r |z|^2, |z|^2)^2},
\quad z \in \A_q,
\label{eqn:density_Aqt}
\end{equation}
which is always positive. 
Since $\rho^1_{\A_q}(z; r)$ depends only on 
the modulus of the coordinate $|z| \in (q, 1)$, 
the PDPP is rotationally invariant. 
As shown by 
(\ref{eqn:theta_signs})--(\ref{eqn:max_theta}) in Section \ref{sec:S_qt_theta}, 
in the interval $x \in (-\infty, 0)$, $\theta(x)$ is positive and
strictly convex with $\lim_{x \downarrow -\infty} \theta(x)
=\lim_{x \uparrow 0} \theta(x)=+\infty$, while
in the interval $x \in (q^2,1)$, $\theta(x)$ is positive
and strictly concave 
with
$\theta(x) \sim
q_0^2(x-q^2)/q^2$ as $x \downarrow q^2$
and $\theta(x) \sim q_0^2(1-x)$ as $x \uparrow 1$.
Therefore, the density shows divergence both at the inner
and outer boundaries as
\begin{equation}
\rho^1_{\A_q}(z; r) \sim
\begin{cases}
\displaystyle{
\frac{q^2}{ (|z|^2-q^2)^2}
},
& |z| \downarrow q,
\cr
\displaystyle{
\frac{1}{(1-|z|^2)^2}
},
& |z| \uparrow 1,
\end{cases}
\label{eqn:density_edges}
\end{equation}
which is independent of $r$ and implies $\bE[\cZ_{X^r_{\A_q}}(\A_q)]=\infty$. 
If $M$ is a $2 \times 2$ matrix, we see that
$\perdet M=\det (M \circ M)$,
where $M \circ M$ denotes the Hadamard product of $M$, 
i.e., entrywise multiplication, $(M \circ M)_{ij}=M_{ij} M_{ij}$. 
Then the two-point correlation is expressed by a single 
determinant as
\begin{equation}
\rho^{2}_{\A_q}(z_1, z_2; r)
= \frac{\theta(-r)}{\theta( -r |z_1|^4|z_2|^4)}
\det_{1 \leq i, j \leq 2} 
\left[S_{\A_q}(z_i, z_j ; r |z_1|^2 |z_2|^2)^2
\right], 
\quad z_1, z_2 \in \A_q.
\label{eqn:rho2_Aqt}
\end{equation}

The above GAF and the PDPP induce the following limiting cases.
With fixed $r >0$ we take the limit $q \to 0$.
By the reason explained in Remark \ref{rem:q0limit} below, 
in this limiting procedure, we should consider the 
point processes $\{\cZ_{X^r_{\A_q}} : q>0\}$ to be defined on 
the punctured unit disk 
$\D^{\times} :=\{z \in \C : 0 < |z| < 1\}$
instead of $\D$.
Although the limit point process is given on $\D^{\times}$ 
by definition, 
it can be naturally viewed as a point process
defined on $\D$, which we will introduce below.  
Let $H_r^2(\D)$ be the Hardy space on $\D$ with the
weight parameter $r>0$,
whose inner product is given by
\[
\bra f, g \ket_{H^2_r(\D)} 
=\frac{1}{2\pi} \int_0^{2 \pi} f(e^{\sqrt{-1} \phi})
\overline{g(e^{\sqrt{-1} \phi})} d \phi
+ r f(0) g(0), 
\quad f, g \in H^2_r(\D). 
\]
The reproducing kernel of $H^2_r(\D)$ is given by
\begin{align}
S_{\D}(z, w; r)
&= \sum_{n=0}^{\infty} e^{(0,r)}_n(z) \overline{e^{(0,r)}_n(w)}
= \frac{1}{1+r} + \sum_{n=1}^{\infty} (z \wbar)^n 
\nonumber\\
&=
\frac{1+r z \wbar}{(1+r)(1-z \wbar)}, 
\quad z, w \in \D.
\label{eqn:SA0t}
\end{align}
The GAF associated with $H_r^2(\D)$ is then
defined by
\begin{equation}
X_{\D}^r(z)
= \frac{\zeta_0}{\sqrt{1+r}} 
+  \sum_{n=1}^{\infty} \zeta_n z^n,
\quad z \in \D
\label{eqn:GAF_A0t}
\end{equation}
so that the covariance kernel is given by 
$\bE[X_{\D}^r(z) 
\overline{ X_{\D}^r(w)} ]
=S_{\D}(z,w;r)$, 
$z, w \in \D$.
For the conditional GAF given a zero at $\alpha \in \D$, 
the covariance kernel is given by
\[
S_{\D}^{\alpha}(z, w; r)
=S_{\D}(z, w; r |\alpha|^2)
h_{\alpha}(z) \overline{h_{\alpha}(w)},
\quad z, w, \alpha \in \D,
\]
where the replacement of the weight parameter
$r$ by $r |\alpha|^2$ should be done, even though
the factor $h_{\alpha}(z)$ is simply given by
the M\"{o}bius transformation (\ref{eqn:Mob1}).

For the zero point process
Theorem \ref{thm:mainA1} is reduced to
the following by the formula $\lim_{q \to 0} \theta(z; q^2)=1-z$.
\begin{cor}
\label{thm:mainA2}
Assume that $r >0$.
Then $\cZ_{X_{\D}^r}$ is 
a PDPP on $\D$ with the correlation functions
\begin{equation}
\rho^{n}_{\D}(z_1,\dots,z_n; r) 
=
\frac{1+r}{1+r \prod_{k=1}^n |z_k|^4}
\perdet_{1 \leq i, j \leq n}
\Big[S_{\D}\Big(z_{i}, z_{j}; r \prod_{\ell=1}^n |z_{\ell}|^2 \Big) \Big]
\label{eqn:rho_mainA2}
\end{equation} 
for every $n \in \N$ and
$z_1, \dots, z_n \in \D$
with respect to $m/\pi$.
In particular, the density of zeros on $\D$
is given by
\begin{equation}
\rho_{\D}^1(z;r)
= \frac{(1+r)(1+r|z|^4)}{(1+r|z|^2)^2 (1-|z|^2)^2},
\quad z \in \D.
\label{eqn:density_A0t}
\end{equation}
\end{cor}
\noindent
As $r$ increases the first term in (\ref{eqn:GAF_A0t}), 
which gives the value of the GAF at the origin, 
decreases and hence the variance at the origin,
$\bE[|X_{\D}^r(0)|^2]=S_{\D}(0,0;r)
=(1+r)^{-1}$ decreases monotonically.
As a result the density of zeros 
in the vicinity of the origin 
increases as $r$ increases.
Actually we see that $\rho_{\D}^1(0; r)=1+r$.
\begin{rem}
\label{rem:q0limit}
The asymptotics \eqref{eqn:density_edges} show that 
the density of zeros of 
$\cZ_{X^r_{\A_q}}$ diverges at the inner boundary
$\gamma_q =\{z : |z|=q\}$ for each $q>0$ 
while the density of $\cZ_{X^r_{\D}}$ is finite at the origin 
as in \eqref{eqn:density_A0t}. Therefore infinitely many zeros near 
the inner boundary $\gamma_q$
seem to vanish in the limit as $q \to 0$. 
This is the reason why we regard the base space of
$\{\cZ_{X^r_{\A_q}} : q>0\}$ and the limit point process
$\cZ_{X^r_{\D}}$ as $\D^{\times}$ instead of $\D$ 
as mentioned before.
(See Section \ref{sec:Peres_Virag} for the general formulation of
 point processes.)
Indeed, in the vague topology, 
with which we equip a configuration space, 
we cannot see configurations outside each compact set, hence 
infinitely many zeros are not observed 
on each compact set in $\D^{\times}$ (not $\D$)
for any sufficiently small $q>0$ 
depending on the compact set that we take. 
\end{rem}

We note that if we take the further limit $r \to 0$ in 
(\ref{eqn:SA0t}), we obtain the Szeg\H{o} kernel of $\D$
given by (\ref{eqn:SD}).
Since the matrix
$
( S_{\D}(z_i, z_j)^{-1} )_{1 \leq i, j \leq n}
=(1-z_i \overline{z_j})_{1 \leq i, j \leq n}
$
has rank 2, the following equality called 
\textit{Borchardt's identity} holds
(see Theorem 3.2 in \cite{Min78}, Theorem 5.1.5 in \cite{HKPV09}),
\begin{equation}
\perdet_{1 \leq i, j \leq n} 
\Big[ (1-z_i \overline{z_j})^{-1} \Big]
= \det_{1 \leq i, j \leq n}
\Big[ (1-z_i \overline{z_j})^{-2} \Big].
\label{eqn:Borchardt}
\end{equation}
By the relation (\ref{eqn:SKC_D}), the $r \to 0$ limit
of $\cZ_{X_{\D}^r}$ is identified with
the DPP on $\D$, $\cZ_{X_{\D}}$,
studied by Peres and Vir\'ag \cite{PV05},
whose correlation functions are given by
\[
\rho_{\D, \mathrm{PV}}^n(z_1, \dots, z_n)
=\det_{1 \leq i, j \leq n} [K_{\D}(z_i, z_j)],
\quad n \in \N, \quad z_1, \dots, z_n \in \D,
\]
with respect to $m/\pi$ 
(see Section \ref{sec:Peres_Virag} below). 

\begin{rem}
\label{rem:Remark_r_infinity}
We see from (\ref{eqn:SA0t}) that
$\lim_{r \to \infty} S_{\D}(z,w; r)
= (1-z \wbar)^{-1} -1$, $z, w \in \D$,
which can be identified with 
the conditional kernel 
given a zero at the origin; 
$S_{\D}^0(z, w)
= S_{\D}(z, w)- S_{\D}(z,0) S_{\D}(0, w)/S_{\D}(0,0)$
for $S_{\D}(z,0) \equiv 1$. 
In this limit we can use Borchardt's identity again, since
the rank of the matrix
$(S_{\D}(z_i, z_j; \infty)^{-1})_{1 \leq i, j \leq n}
=(z_i^{-1} z_j^{-1}-1)_{1 \leq i, j \leq n}$ is two. 
Then, thanks to the proper limit of the prefactor 
of $\perdet$ in (\ref{eqn:rho_mainA2})
when $z_k \in \D^{\times}$ for all $k=1,2,\dots, n$; 
$\lim_{r \to \infty}(1+r)/(1+r \prod_{k=1}^n |z_k|^4)
=\prod_{k=1}^n |z_k|^{-4}$, 
we can verify that $\lim_{r \to \infty} \rho_{\D}^n(z_1, \dots, z_n; r)
= \rho_{\D, \mathrm{PV}}^n(z_1, \dots, z_n)$
for every $n \in \N$, and every 
$z_1, \dots, z_n \in \D^{\times}$. 
On the other hand, taking \eqref{eqn:GAF_A0t} into account, we have 
$X^{\infty}_{\D}(z)
= z \sum_{n=1}^{\infty} \zeta_n z^{n-1}
\dis= z X_{\D}(z)$, from which, 
we can see that as $r \to \infty$, $\cZ_{X_{\D}^r}$
converges to 
$\cZ_{X_{\D}^{\infty}} \dis= \cZ_{X_{\D}}+\delta_0$; that is,
the DPP of Peres and Vir\'ag with
a deterministic zero added at the origin.
This is consistent with the fact that $\rho_{\D}^1(0;r) =
1+r$ diverges as $r \to \infty$.
Since 
$\cZ_{X_{\D}^0} :=\lim_{r \to 0} \cZ_{X_{\D}^r} \dis= \cZ_{X_{\D}}$
as mentioned above, 
the one-parameter family of 
PDPPs $\{\cZ_{X_{\D}^r} : r \in (0, \infty)\}$ 
can be regarded as an interpolation between 
the DPP of Peres and Vir\'ag and 
that DPP with a deterministic zero added at the origin. 
\end{rem}

\subsection{Unfolded 2-correlation functions}
\label{sec:two_point}

By the determinantal factor in 
$\perdet$ (\ref{eqn:perdet})
the PDPP shall be negatively correlated 
when distances of points
are short in the domain $\A_q$.
The effect of the permanental part \cite{ST03a,MM06} 
in $\perdet$ will appear in long distances.
Contrary to such a general consideration for the PDPP,
if we take the double limit, $q \to 0$ and then $r \to 0$,
Borchardt's identity (\ref{eqn:Borchardt}) becomes applicable
and the zero point process is reduced to the DPP
studied by Peres and Vir\'ag \cite{PV05}.
In addition to this fact, the two-point correlation 
of the PDPP can be generally expressed using a single determinant
as explained in the sentence above (\ref{eqn:rho2_Aqt}).
We have to notice the point, however, that
the weight parameter $r |z|^2$ of $S_{\A_q}$
for the density (\ref{eqn:density_Aqt})
is replaced by $r|z_1|^2 |z_2|^2$ for the
two-point correlation (\ref{eqn:rho2_Aqt}),
and the prefactor $\theta(-r)/\theta(-r |z|^4)$
of $S_{\A_q}$
for $\rho_{\A_q}^1$ is changed to
$\theta(-r)/\theta(-r|z_1|^4 |z_2|^4)$ for $\rho_{\A_q}^2$.
Here we show that due to such alterations
our PDPP can not be reduced to any DPP
in general and it has indeed both of negative and positive correlations
depending on the distance of points and the values of parameters.
In order to clarify this fact, 
we study the two-point correlation
function normalized by the product of
one-point functions,
\begin{align}
g_{\A_q}(z, w; r)
:= \frac{\rho_{\A_q}^2(z,w; r)}
{\rho_{\A_q}^1(z;r) \rho_{\A_q}^1(w;r)},
\quad 
(z, w) \in \A_q^2, 
\label{eqn:unfolded}
\end{align}
where $\rho^1_{\A_q}$ and $\rho^2_{\A_q}$
are explicitly given by (\ref{eqn:density_Aqt})
and (\ref{eqn:rho2_Aqt}), respectively.
This function is simply called an intensity ratio in \cite{PV05},
but here we call it an \textit{unfolded 2-correlation function}
following a terminology used in 
random matrix theory \cite{For10}.
We will prove the following:
(i) When $0 < q < 1$, in the short distance,
the correlation is generally repulsive in common with DPPs
(Proposition \ref{thm:short_distance}).
(ii) There exists a critical value $r_0 =r_0(q) \in (q, 1)$
for each $q \in (0, 1)$ 
such that if $r \in (r_0, 1)$ positive correlation emerges
between zeros when the distance between them 
is large enough within $\A_q$ 
(Theorem \ref{thm:critical_line}). 
(iii) The limits 
$g_{\D}(z, w; r) :=\lim_{q \to 0} g_{\A_q}(z, w; r)$,
$z, w \in \D^{\times}$ and
$r_{\rm c}:=\lim_{q \to 0} r_0(q)$
are well-defined, and $r_{\rm c}$ is positive.
When $r \in [0, r_{\rm c})$
all positive correlations vanish in $g_{\D}(z, w; r)$
(Proposition \ref{thm:negative_corr_q0}),
while when $r \in (r_{\rm c}, \infty)$
positive correlations can survive
(Remark \ref{rem:positive_corr} 
given at the end of Section \ref{sec:proofs}).
In addition to these rigorous results, we will report
the numerical results for $q \in (0, 1)$: 
In intermediate distances between zeros, 
positive correlations are observed at any value of $r \in (q, 1)$,
but the distance-dependence of correlations 
shows two distinct patterns depending on
the value of $r$, whether $r \in (q, r_0)$ or $r \in (r_0, 1)$
(Figure \ref{fig:G_vee_plots}). 

It should be noted that 
the $(q,r)$-inversion symmetry (\ref{eqn:rho_inversion0}) 
implies the equality
(see the second assertion of Lemma 
\ref{thm:transformation_rho} given below),
\begin{equation}
g_{\A_q}
(q/z, q/w; q^2/r)
=g_{\A_q}(z, w; r),
\quad (z, w) \in \A_q^2.
\label{eqn:rho_symmetry}
\end{equation}
Provided that the moduli
of coordinates $|z|, |w|$ are fixed, 
we can verify that
the unfolded 2-correlation function
takes a minimum (resp. maximum) when
$\arg w =  \arg z$ (resp. $\arg w = - \arg z$) 
(Lemma \ref{thm:bounds} in
Section \ref{sec:upper_lower}).
We consider these two extreme cases. 
By putting
$w=x, z=q/x \in (\sqrt{q}, 1)$ we define the function
\begin{equation}
G_{\A_q}^{\wedge}(x; r) =
g_{\A_q}(q/x, x; r),
\quad x \in (\sqrt{q}, 1)
\label{eqn:g1_short}
\end{equation}
in order to characterize the short distance behavior 
of correlation,
and by putting 
$w=-z =x \in (q, 1)$ we define the function
\begin{equation}
G_{\A_q}^{\vee}(x; r) =
g_{\A_q}(-x, x; r),
\quad x \in (q, 1). 
\label{eqn:g1_long}
\end{equation}
in order to characterize the long distance behavior 
of correlation.

Since the PDPP is rotationally symmetric, 
$G_{\A_q}^{\wedge}(x; r)$ shows
correlation between two points 
on \textit{any} line passing through the origin
located in the \textit{same} side
with respect to
the inner circle $\gamma_q$ of $\A_q$.
The Euclidean distance between these two points
is $(x^2-q)/x$ and it becomes zero as $x \to \sqrt{q}$.
We can see the power law with index $\beta=2$ 
in the short distance correlation as follows,
which is the common feature with DPPs \cite{For10}.

\begin{prop}
\label{thm:short_distance}
As $x \to \sqrt{q}$, 
$G_{\A_q}^{\wedge}(x; r) \sim c(r) (x -\sqrt{q})^{\beta}$
with $\beta=2$,
where
$c(r) =
(8 q_0^4 r^3 \theta(-qr)^6)
/(q^2 \theta(q)^2 \theta(-r)^6)>0$.
\end{prop}
\noindent
Proof is given in 
Section \ref{sec:short_distance}.

The function $G_{\A_q}^{\vee}(x; r)$ shows
correlation between two points 
on a line passing through the origin, which are
located in the \textit{opposite} sides with respect to $\gamma_q$
and have the Euclidean distance $2x$. 
Long-distance behavior of the PDPP
will be characterized by this function 
in the limit $x \to 1$ in $\A_q$ 
(see Remark \ref{rem:slit_domain} given below).
In this limit the correlation decays as
$G_{\A_q}^{\vee}(x; r) \to 1$.
We find that the decay obeys 
the power law with a fixed index $\eta=4$, but 
the sign of the coefficient changes 
at a special value of $r$ for each $q \in (0, 1)$.
Given $(q, r)$, define $\tau_q$ and $\phi_{-r}$ by
\[
q=e^{\sqrt{-1} \pi \tau_q}, \quad
-r=e^{\sqrt{-1} \phi_{-r}},
\]
and consider the Weierstrass $\wp$-function
$\wp(\phi_{-r})=\wp(\phi_{-r}; \tau_q)$ 
given by (\ref{eqn:wp_expansion4}) in
Section \ref{sec:Weierstrass_elliptic} below.
The functions of $q$,
$e_i =e_i(q), i=1,2,3$ and $g_2=g_2(q)$ are defined by
(\ref{eqn:e123}) and (\ref{eqn:g2g3}).

\begin{thm}
\label{thm:critical_line}
\begin{description}
\item{\rm (i)} \, 
For $r > 0$, 
\begin{equation}
G_{\A_q}^{\vee}(x; r) \sim
1+\kappa(r) (1-x^2)^{\eta}
\quad \mbox{as $x \uparrow 1$},
\label{eqn:Gvee_up}
\end{equation}
and
\begin{equation}
G_{\A_q}^{\vee}(x; r) \sim
1+\frac{\kappa(r)}{q^8} (x^2-q^2)^{\eta}
\quad \mbox{as $x \downarrow q$},
\label{eqn:Gvee_down}
\end{equation}
with $\eta=4$, 
where
\begin{equation}
\kappa(r)=\kappa(r; q) :=
5 \wp(\phi_{-r})^2+2 e_1 \wp(\phi_{-r})
- (e_1^2 + g_2/2).
\label{eqn:kappa}
\end{equation}
The coefficient 
$\kappa$ has the reciprocity property, 
periodicity property, 
and their combination, 
\begin{equation}
\kappa(1/r)=\kappa(r), \quad
\kappa(q^2 r)=\kappa(r), \quad
\kappa(q^2/r)=\kappa(r).
\label{eqn:kappa_symmetry}
\end{equation}
Hence for the parameter space
$\{(q, r) : q \in [0, 1],  r>0\}$,
a fundamental cell is given by
$\Omega:=\{(q, r): q \in (0, 1), q \leq r \leq 1\}$.

\item{\rm (ii)} \,
It is enough to describe $\kappa(r)$ 
in $\Omega$. Let 
\begin{equation}
\wp_+=\wp_+(q) :=
-\frac{e_1}{5} + \frac{1}{10} \sqrt{24 e_1^2+10 g_2}.
\label{eqn:wp+}
\end{equation}
Then 
$e_1 > \wp_+ > e_2 > e_3$,
and
\begin{equation}
r_0=r_0(q) :=\exp \Big[
- \frac{1}{2} 
\int_{\wp_+}^{e_1}
\frac{ds}{\sqrt{(e_1-s)(s-e_2)(s-e_3)}} \Big]
\label{eqn:tcq}
\end{equation}
satisfies the inequalities,
\begin{equation}
q< r_0(q) < 1, \quad q \in (0, 1).
\label{eqn:ineq_a}
\end{equation}
The coefficient $\kappa(r)$ in (\ref{eqn:Gvee_up})
and (\ref{eqn:Gvee_down}) changes its sign at
$r =r_0$ as follows;
$\kappa(r) <0$ if 
$r \in (q, r_0)$,
and $\kappa(r) > 0$ if $r \in (r_0, 1)$.

\item{\rm (iii)} \,
The curve $\{r=r_0(q): q \in (0,1) \} \subset \Omega$ 
satisfies the following; 
\begin{align*}
&{\rm (a)} \quad  
r_{\rm c} := \lim_{q \to 0} r_0(q)
=\frac{1-\frac{\sqrt{4-\sqrt{6}}}{\sqrt{5}}}
{1+\frac{\sqrt{4-\sqrt{6}}}{\sqrt{5}}}
=2 \sqrt{6}-3-2 \sqrt{8-3 \sqrt{6}}
=0.2846303639 \cdots,
\nonumber\\
&
{\rm (b)} \quad 
r_0(q) \sim r_{\rm c}+c q^2
\quad \mbox{as $q \to 0$}
\nonumber\\
& 
\quad \mbox{with} \quad
c = \frac{8}{3}
\Big[ -72 + 22 \sqrt{6} 
+ 3(4 \sqrt{6}-1) \sqrt{8-3 \sqrt{6}} \Big]
=8.515307593 \cdots,
\nonumber\\
&{\rm (c)} \quad  
r_0(q) \sim 1-\frac{1}{2}(1-q)
\quad \mbox{as $q \to 1$}.
\end{align*}
\end{description}
\end{thm}
\noindent
The proof is
given in Sections \ref{sec:i}--\ref{sec:iii}.

\begin{rem}
\label{rem:slit_domain}
For $s >0$, 
define a horizontal slit 
$[-s + \sqrt{-1}, s+ \sqrt{-1}]$ in 
the upper half plane 
$\HH=\{z \in \C: \Im z > 0 \}$
and consider a doubly connected domain 
$D(s) := \HH \setminus [-s+\sqrt{-1}, s+ \sqrt{-1}]$.
Such a domain is called the {\it chordal standard domain}
with connectivity $n=2$ 
in \cite{BF08} (see also Chapter VII in \cite{Neh52}).
As briefly explained in
Appendix \ref{sec:D_Ds},
the conformal map from 
$\A_q$ to $D(s)$ is given by
\begin{equation}
H_q(z)
=- 2 \Big\{
\zeta(-\sqrt{-1} \log z)
+\sqrt{-1} (\eta_1/\pi) \log z \Big\},
\quad z \in \A_q,
\label{eqn:Hqz}
\end{equation}
where the Weierstrass $\zeta$-function and
its special value $\eta_1$ are defined in
Section \ref{sec:Weierstrass_elliptic} below.
This conformal map is chosen so that
the boundary points are mapped as
\[
H_q(-1)=0, \quad H_q(1)=\infty,
\quad
H_q(\pm \sqrt{-1} q) = \mp s+\sqrt{-1},
\quad H_q(\pm q)=\sqrt{-1}.
\]
The $x \to 1$ limit for a pair of points
$-x$ and $x$ on $\A_q \cap \R$ is hence
regarded as the pull-back of an
infinite-distance limit of two points on
$\HH \cap \sqrt{-1} \R$.
In the $q \to 0$ limit, $H_q$ is reduced to
the Cayley transformation
from $\D$ to $\HH$, 
$H_0(z)=-\sqrt{-1}(z+1)/(z-1)$, 
such that $H_0(-1)=0$,
$H_0(1)=\infty$ and $H_0(0)=\sqrt{-1}$.
\end{rem}

Theorem \ref{thm:critical_line} implies that
if $r \in (r_0, 1)$, $G_{\A_q}^{\vee}(x; r)>1$
when $x$ is closed to $q$ or $1$.
Appearance of such positive correlations
proves that the present PDPP $\cZ_{X_{\A_q}^r}$
is indeed unable to be identified with any DPP.

Let
$g_{\D}(z, w; r)
=\rho_{\D}^2(z,w;r)/
\rho_{\D}^1(z;r) \rho_{\D}^1(w;r)$,
$(z, w) \in {\D}^2$.
The asymptotic
(\ref{eqn:Gvee_up}) holds for
$G_{\D}^{\vee}(x; r):=g_{\D}(-x, x; r)$ with 
$\kappa_0(r):=\lim_{q \to 0} \kappa(r; q)$,
which has the reciprocity property
$\kappa_0(1/r)=\kappa_0(r)$ 
(see (\ref{eqn:kappa_q0}) in Section \ref{sec:iii}).
When $r \in (r_{\rm c}, 1/r_{\rm c})$, $\kappa_0(r)>0$
and hence $G_{\D}^{\vee}(x; r) > 1$ 
for $x \lesssim 1$, which
indicates appearance of attractive interaction
at large intervals in $\D^{\times}$.
When $r \in [0, r_{\rm c})$ or $r \in (1/r_{\rm c}, \infty)$, 
negative $\kappa_0(r)$ implies $G_{\D}^{\vee}(x; r) < 1$
even for $x \lesssim 1$. 
Moreover, we can prove the following.  
\begin{prop}
\label{thm:negative_corr_q0}
If $r \in [0, r_{\rm c})$, then
$g_{\D}(z, w; r)< 1$, $\forall (z, w) \in {\D}^2$.
\end{prop}
\noindent
The proof is given in Section \ref{sec:proof_negative}.
We should note that this statement does not hold for
$r \in (1/r_{\rm c}, \infty)$, since 
we can verify that 
$g_{\D}(z, w; r)$ can exceed 1 when $r > 1$
(see Remark \ref{rem:positive_corr} 
in Section \ref{sec:proof_negative}).
Therefore, we say that 
there are two phases
for the PDPP $\cZ_{X_{\D}^r}$
in the following sense: 
\begin{description}
\item{(i) \quad 
Repulsive phase:} when 
$r \in [0, r_{\rm c})$, 
all pairs of zeros are negatively correlated. 

\item{(ii) \quad
Partially attractive phase:} when 
$r \in (r_{\rm c}, \infty)$, 
positive correlations emerge between zeros.
\end{description}

\begin{figure}[htbp]
\begin{center}
\includegraphics[width=0.7\hsize]{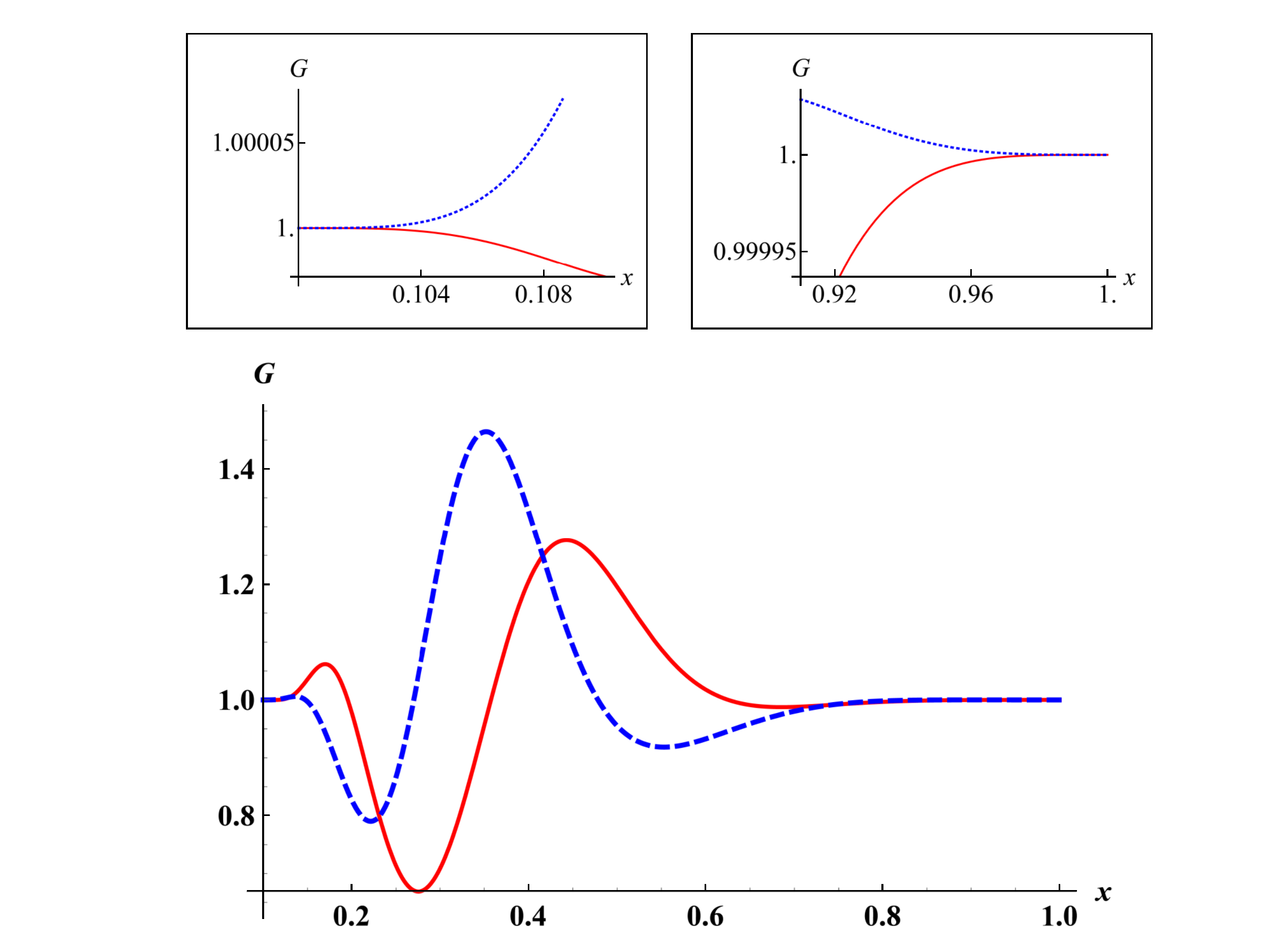}
\end{center}
\caption{Numerical plots of $G_{\A_q}^{\vee}(x; r)$
with $q=0.1$ 
are given in the interval $(q, 1)$ 
for $r=0.2$ (red solid curve) and $r=0.6$ (blue dashed curve).
Note that $0.2 < r_0(0.1) = 0.348 \cdots < 0.6$.
Then following Theorem 1.6 (i) and (ii), 
the red solid curve
(resp.~blue dashed curve) approaches to the unity
from below (resp. from above) 
as $x \to q=0.1$ (see the upper left inset)
and
as $x \to 1$ (see the upper right inset).
In the intermediate values of $x$, 
the red solid curve shows two local maxima
greater than unity and a unique 
minimum $<1$ 
at the point near $\sqrt{q}=0.316 \cdots$,
while 
the blue dashed curve has two local minima $<1$
and a unique maximum $>1$
at the point near $\sqrt{q}=0.316 \cdots$.
The global pattern of correlations
is converted when the sign of $\kappa(r)$
is changed. 
} 
\label{fig:G_vee_plots}
\end{figure}

When $q \in (0, 1)$, however, the repulsive phase
seems to disappear
and positive correlations can be observed at any value of
$r >0$.
Figure \ref{fig:G_vee_plots} shows 
numerical plots of $G_{\A_q}^{\vee}(x; r)$
for $q=0.1$ with $r_0(0.1)=0.348 \cdots$.
The red solid (resp.~blue dashed) curve shows the pair correlation 
for $r=0.2$
(resp.~$r=0.6)$.
Since $r=0.2 < r_0(0.1)$ the red solid curve tends to be
less than unity in the vicinity of edges
$x=q$ and $x=1$ 
as shown in the insets 
(following Theorem \ref{thm:critical_line} (i) and (ii)), 
but shows two local maxima
greater than unity and then has a unique 
minimum $<1$ 
at the point near $\sqrt{q}=0.316 \cdots$.
On the other hand, the blue dashed curve with 
$r=0.6 > r_0(0.1)$ tends to be 
greater than unity in the vicinity of edges
as shown in the insets
(following Theorem \ref{thm:critical_line} (i) and (ii)),
but shows two local minima $<1$
and then has a unique maximum $>1$
at the point near $\sqrt{q}=0.316 \cdots$.
As demonstrated by these plots, 
the change of sign of $\kappa(r)$
at $r=r_0 \in (q, 1)$ seems to convert
a global pattern of correlations.
Figure \ref{fig:phase_diagram} shows 
a numerical plot of the curve
$r=r_0(q)$, $q \in (0,1)$ in the fundamental 
cell $\Omega$ of the parameter space.
Detailed characterization of correlations
(not only pair correlations but also 
$\rho^n_{\A_q}, n \geq 3$) in PDPPs will be
a future problem.

\begin{figure}[htbp]
\begin{center}
\includegraphics[width=0.4\hsize]{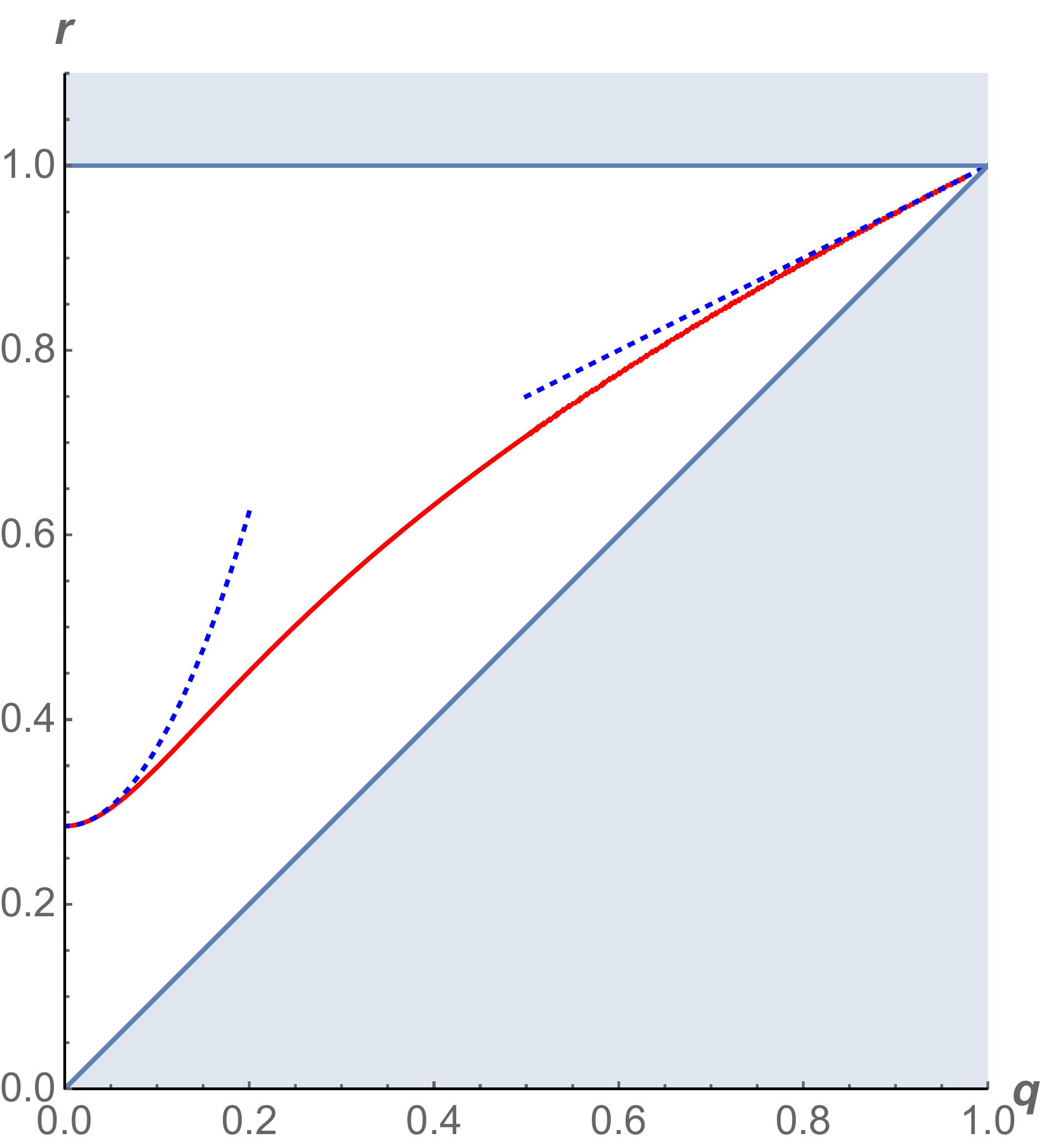}
\end{center}
\caption{
The curve $r=r_0(q)$ given by
(\ref{eqn:tcq}) in Theorem \ref{thm:critical_line} (ii)
is numerically plotted 
(in red) in the fundamental cell
$\Omega$ in the parameter space,
which is located between the diagonal line
$r=q$ (shown by a blue line) 
and the horizontal line $r=1$
satisfying (\ref{eqn:ineq_a}).
The parabolic curve $r_{\rm c}+ c q^2$
given by (iii) (b)
and the line $1-(1-q)/2$ 
by (iii) (c) are also dotted,
which approximate $r=r_0(q)$ well for $q \gtrsim 0$ and
$q \lesssim1$, respectively.}
\label{fig:phase_diagram}
\end{figure}

The paper is organized as follows.
In Section \ref{sec:preliminaries} we give
preliminaries, which include 
a brief review of reproducing kernels,
conditional Szeg\H{o} kernels, and
a general treatment of point processes
including DPPs.
There we also give 
definitions and basic properties of
special functions used to represent
and analyze GAFs and their zero point processes on an annulus.
Section \ref{sec:proofs} is devoted to
proofs of theorems.
Concluding remarks are given in Section \ref{sec:concluding}.
Appendices will provide additional information
related to the present study.

\SSC{Preliminaries}
\label{sec:preliminaries}
\subsection{Reproducing kernels}
\label{sec:RK}
A \textit{Hilbert function space}
is a Hilbert space $\cH$ of functions
on a domain $D$ in $\C^d$ 
equipped with inner product $\bra \cdot, \cdot \ket_{\cH}$
such that evaluation at each point 
of $D$ is a continuous functional on $\cH$.
Therefore, for each point $w \in D$, 
there is an element of $\cH$,
which is called the reproducing kernel at $w$
and denote by $k_{w}$, with the property
$\bra f, k_w \ket_{\cH}=f(w), \forall f \in \cH$.
Because $k_{w} \in \cH$, it is itself a function on $D$,
$k_w(z)=\bra k_w, k_z \ket_{\cH}$.
We write
\[
k_{\cH}(z, w):=k_w(z) = \bra k_w, k_z \ket_{\cH}
\]
and call it the \textit{reproducing kernel} for $\cH$.
By definition, it is hermitian; 
$\overline{k_{\cH}(z,w)}=k_{\cH}(w, z)$, $z, w \in D$.
If $\cH$ is a holomorphic Hilbert function space, then
$k_{\cH}$ is holomorphic in the first variable and
anti-holomorphic in the second.
We see that $k_{\cH}(z,w)$ is a positive semi-definite kernel: 
for any $n \in \N:=\{1,2,\dots\}$, 
for any points $z_i \in D$ and $\xi_i \in \C$, $i=1,2, \dots, n$,
\begin{equation}
\sum_{i=1}^n \sum_{j=1}^n k_{\cH}(z_i, z_j) \xi_i \overline{\xi_j} 
=  \Big\|\sum_{i=1}^n \xi_i k_{\cH}(z_i, \cdot) 
\Big\|_{\cH}^2 \ge 0. 
\label{eqn:positive_kernel}
\end{equation}
Let $\{e_n : n \in \cI \}$ be any CONS for $\cH$,
where $\cI$ is an index set. Then one can prove that
the reproducing kernel for $\cH$ is written in the form
\begin{equation}
k_{\cH}(z, w) = \sum_{n \in \cI} e_n(z) \overline{e_n(w)}.
\label{eqn:reproducing}
\end{equation}
We note that the positive definiteness of the kernel 
(\ref{eqn:positive_kernel})
is equivalent with the situation such that, 
for any points $z_i \in D$, $i \in \N$, the matrix 
$(k_{\cH}(z_i, z_j))_{1 \leq i,j \leq n}$ has a nonnegative determinant, 
$\det_{1\leq i, j \leq n}[k_{\cH}(z_i, z_j)] \geq 0$, 
for any $n \in \N$.

Here we show two examples
of holomorphic Hilbert function spaces,
the \textit{Bergman space} and the \textit{Hardy space}, 
for a unit disk $\D$ and the domains which are conformally 
transformed from $\D$ 
\cite{Neh52,Ber70,HKZ00,AM02,Bel16}. 

The Bergman space on $\D$, denoted by $L^2_{\rB}(\D)$, is
the Hilbert space of holomorphic functions on $\D$
which are square-integrable with respect to
the Lebesgue measure on $\C$ \cite{Ber70}.
The inner product for $L^2_{\rB}(\D)$ is given by
\[
\bra f, g \ket_{L^2_{\rB}(\D)}
:= \frac{1}{\pi} \int_{\D} f(z) \overline{g(z)} m(dz)
= \sum_{n=0}^{\infty} \frac{\widehat{f}(n) 
\overline{\widehat{g}(n)}}{n+1},
\]
where the $n$th Taylor coefficient of $f$ at 0 is denoted by
$\widehat{f}(n)$;
$f(z) = \sum_{n=0}^{\infty} \widehat{f}(n) z^n$.
Let $\widetilde{e_n}(z)
:=\sqrt{n+1} z^{n}, n \in \N_0$.
Then $\{\widetilde{e_n}(z) \}_{n \in \N_0}$ form a
CONS for $L^2_{\rB}(\D)$
and the reproducing kernel (\ref{eqn:reproducing}) is given by
\begin{align}
K_{\D}(z, w) 
&:=k_{L^2_{\rB}(\D)}(z, w)
\nonumber\\
&= \sum_{n \in \N_0}
(n+1) (z \overline{w})^n
=\frac{1}{(1-z \overline{w})^2},
\quad z, w \in \D.
\label{eqn:K_D}
\end{align}
This kernel is called the \textit{Bergman kernel} of $\D$.

The Hardy space on $\D$, $H^2(\D)$, consists of 
holomorphic functions on $\D$ such that the Taylor coefficients
form a square-summable series;
\[
\| f \|_{H^2(\D)}^2
:= \sum_{n \in \N_0} |\widehat{f}(n)|^2 < \infty,
\quad f \in H^2(\D).
\]
For every $f \in H^2(\D)$, 
the non-tangential limit
$\lim_{r \uparrow 1} f(r e^{\sqrt{-1} \phi})$ exists
a.e.~by Fatou's theorem and we write it as
$f(e^{\sqrt{-1} \phi})$.
It is known that $f(e^{\sqrt{-1} \phi}) \in L^2(\partial \D)$
\cite{AM02}. 
Then one can prove that the inner product of $H^2(\D)$ 
is given by the following
three different ways \cite{AM02},
\begin{equation}
\bra f, g \ket_{H^2(\D)}
= \begin{cases}
\displaystyle{
\sum_{n \in \N_0} \widehat{f}(n) \overline{\widehat{g}(n)}},
\cr
\displaystyle{
\lim_{r \uparrow 1} \frac{1}{2\pi}
\int_{0}^{2 \pi} f(r e^{\sqrt{-1} \phi})
\overline{g(r e^{\sqrt{-1} \phi})} d \phi},
\quad f, g \in H^2(\D), 
\cr
\displaystyle{
\frac{1}{2\pi}
\int_{0}^{2 \pi} f(e^{\sqrt{-1} \phi})
\overline{g(e^{\sqrt{-1} \phi})} d \phi},
\end{cases}
\label{eqn:IP_H2D}
\end{equation}
with $\|f\|_{H^2(\D)}^2=\bra f, f \ket_{H^2(\D)}$.
Let $\sigma$ be the measure on the boundary
of $\D$ which is the usual arc length measure.
Then the last expression of the inner product (\ref{eqn:IP_H2D})
is written as
$\bra f, g \ket_{H^2(\D)}
=(1/2 \pi) \int_{\gamma_1} f(z) \overline{g(z)} \sigma(dz)$,
where $\gamma_1$ is a unit circle 
$\{ e^{\sqrt{-1} \phi} : \phi \in [0, 2\pi) \}$ giving the boundary 
of $\D$.
If we set $e_n(z) :=e^{(0,0)}_n(z)=z^n, n \in \N_0$,
then $\{e_n(z)\}_{n \in \N_0}$ form
CONS for $H^2(\D)$. The reproducing kernel (\ref{eqn:reproducing})
is given by
\begin{align}
S_{\D}(z, w) 
&:= k_{H^2(\D)}(z, w)
\nonumber\\
&=
\sum_{n \in \N_0} (z \overline{w})^n
=\frac{1}{1-z \overline{w}},
\quad z, w \in \D, 
\label{eqn:S_D}
\end{align}
which is called the {\it Szeg\H{o} kernel} of $\D$.

Let $f: D \to \widetilde{D}$ be a conformal transformation
between two bounded domains $D, \widetilde{D} \subsetneq \C$
with $\cC^{\infty}$ smooth boundary.
We find an argument in Chapter 12 of \cite{Bel16} 
concluding that 
the derivative of the transformation $f$ denoted by
$f'$ has a single valued square root on $D$.
We let $\sqrt{f'(z)}$ denote one of the 
square roots of $f'$.
The Szeg\H{o} kernel and the Bergman kernel 
are then transformed by $f$ as 
\begin{align}
S_D(z, w) &= 
\sqrt{f'(z)} 
\overline{\sqrt{f'(w)}}
S_{\widetilde{D}}(f(z), f(w)),
\nonumber\\
K_D(z, w) &= |f'(z)|
|f'(w)| K_{\widetilde{D}}(f(z), f(w)), 
\quad z, w \in D.
\label{eqn:SD_KD_conformal}
\end{align}
See Chapters 12 and 16 of \cite{Bel16}.
Consider the special case in which
$D \subsetneq \C$ is a simply connected domain
with $\cC^{\infty}$ smooth boundary 
and $\widetilde{D}=\D$.
For each $\alpha \in D$, 
Riemann's mapping theorem  
gives a unique conformal transformation \cite{Ahl79};
\[
\mbox{
$h_{\alpha} : D \to \D$ \quad conformal 
such that \, 
$h_{\alpha}(\alpha)=0, \, \, h_{\alpha}'(\alpha) >0$.
}
\]
Such $h_{\alpha}$ is called the
\textit{Riemann mapping function}. 
By (\ref{eqn:S_D}), 
the first equation in 
(\ref{eqn:SD_KD_conformal}) gives the following formula
\cite{Bel95}, 
\begin{equation}
S_D(z, w)= \frac{S_D(z, \alpha) \overline{S_D(w, \alpha)}}
{S_D(\alpha, \alpha)} 
\frac{1}{1-h_{\alpha}(z) \overline{h_{\alpha}(w)}},
\quad z, w, \alpha \in D.
\label{eqn:SD_formula}
\end{equation}
Similarly, we have
\begin{equation}
K_D(z, w)= 
\frac{S_D(z, \alpha)^2 \overline{S_D(w, \alpha)^2}}
{S_D(\alpha, \alpha)^2}
\frac{1}{(1-h_{\alpha}(z) \overline{h_{\alpha}(w)})^2},
\quad z, w, \alpha \in D.
\label{eqn:KD_formula}
\end{equation}
Hence the following relationship is established, 
\begin{align}
S_{D}(z, w)^2 &= K_{D}(z, w),
\quad z, w \in D.
\label{eqn:S_K_D}
\end{align}
Although the Szeg\H{o} kernel could be eliminated from
he right-hand sides
of (\ref{eqn:SD_formula}) and (\ref{eqn:KD_formula})
by noting that
$h_{\alpha}'(z)=S_D(z,\alpha)^2/S_D(\alpha,\alpha)$,
the formula (\ref{eqn:SD_formula}) 
and the relation (\ref{eqn:S_K_D}) played
important roles in the study by Peres and Vir\'ag \cite{PV05}.
As a matter of fact, (\ref{eqn:SD_formula}) is equivalent with
(\ref{eqn:SaD}) and the combination of (\ref{eqn:S_K_D})
and (\ref{eqn:Borchardt}) gives (\ref{eqn:SKC_D}).

\subsection{Theta function $\theta$} 
\label{sec:S_qt_theta}

Assume that $p \in \C$ is a fixed number such that 
$0 < |p| < 1$.
We use the following standard notation \cite{GR04,Kra05,RS06},
\begin{align}
&(a; p)_{n} := \prod_{i=0}^{n-1} (1-a p^i), \qquad
(a; p)_{\infty} := \prod_{i=0}^{\infty} (1-a p^i),
\nonumber\\
&(a_1, \dots, a_k; p)_{\infty}
: =(a_1; p)_{\infty} \cdots (a_k; p)_{\infty}.
\label{eqn:Pochhammer}
\end{align}
The \textit{theta function} with argument $z$ and
nome $p$ is defined by
\begin{equation}
\theta(z; p)  :=(z, p/z; p)_{\infty}.
\label{eqn:theta}
\end{equation}
We often use the shorthand notation
$\theta(z_1, \dots, z_n; p)
: = \prod_{i=1}^n \theta(z_i; p)$.

As a function of $z$, the theta function $\theta(z; p)$ is
holomorphic in $\C^{\times}$ and has single zeros
precisely at $p^{i}$, $i \in \Z$, that is,
\begin{equation}
\{ z \in \C^{\times} : \theta(z; p)=0 \}
=\{p^i : i \in \Z \}.
\label{eqn:theta_zero}
\end{equation}
We will use the inversion formula 
\begin{equation}
\theta(1/z; p) = - \frac{1}{z} \theta(z; p)
\label{eqn:theta_inversion}
\end{equation}
and the quasi-periodicity property
\begin{equation}
\theta(pz; p) = -\frac{1}{z} \theta(z; p)
\label{eqn:theta_qp}
\end{equation}
of the theta function.
By comparing (\ref{eqn:theta_inversion}) and 
(\ref{eqn:theta_qp}) and performing 
the transformation $z \mapsto 1/z$, 
we immediately see the periodicity property, 
\begin{align}
\theta(p/z; p) 
= \theta(z; p).
\label{eqn:theta_period}
\end{align}
By Jacobi's triple product identity
(see, for instance, \cite[Section 1.6]{GR04}), 
we have 
the Laurent expansion
\[
\theta(z; p)=\frac{1}{(p; p)_{\infty}}
\sum_{n \in \Z} (-1)^n p^{\binom{n}{2}} z^n.
\]
One can show that \cite[Chapter 20]{NIST10}
\begin{align}
&\lim_{p \to 0} \theta(z; p)=1-z,
\label{eqn:theta_p0}
\\
&\theta'(1; p) 
:= \frac{\partial \theta(z; p)}{\partial z} \Big|_{z=1}
=- (p; p)_{\infty}^2.
\label{eqn:theta_prime}
\end{align}
The theta function satisfies the following 
\textit{Weierstrass' addition formula} \cite{Koo14},
\begin{equation}
\theta(xy, x/y, uv, u/v; p)
- \theta(xv, x/v, uy, u/y; p)
=\frac{u}{y} \theta(yv, y/v, xu, x/u; p).
\label{eqn:Weierstrass_add1}
\end{equation}

When $p$ is real and $p \in (0, 1)$, we see that
\begin{equation}
\overline{\theta(z; p)}= \theta( \zbar; p).
\label{eqn:theta_real}
\end{equation}
In this case the definition (\ref{eqn:theta})
with (\ref{eqn:Pochhammer}) implies that
\begin{align}
&\left.
\begin{array}{ll}
\theta(x; p) > 0, \, & x \in (p^{2i+1}, p^{2i})
\cr
\theta(x; p) =0, \, & x = p^{i}
\cr
\theta(x; p) < 0, \, & x \in (p^{2i}, p^{2i-1})
\end{array}
\right\}
\quad i \in \Z,
\nonumber\\
&\quad \theta(x; p) >0, \quad x \in (-\infty, 0).
\label{eqn:theta_signs}
\end{align}
Moreover, we can prove the following: 
In the interval $x \in (-\infty, 0)$,
$\theta(x) :=\theta(x; p)$ is strictly convex with
\begin{equation}
\min_{x \in (-\infty, 0)} \theta(x)
=\theta(-\sqrt{p}) =
\prod_{n=1}^{\infty}(1+p^{n-1/2})^2 > 0,
\label{eqn:min_theta}
\end{equation}
and 
$\lim_{x \downarrow -\infty} \theta(x)
=\lim_{x \uparrow 0} \theta(x)=+\infty$,
and in 
the interval $x \in (p, 1)$,
$\theta(x)$ is strictly concave with
\begin{equation}
\max_{x \in (p, 1)} \theta(x)
=\theta(\sqrt{p}) =
\prod_{n=1}^{\infty}(1-p^{n-1/2})^2,
\label{eqn:max_theta}
\end{equation}
$\theta(x) \sim (p; p)_{\infty}^2 (x-p)/p$
as $x \downarrow p$, and
$\theta(x) \sim (p; p)_{\infty}^2 (1-x)$
as $x \uparrow 1$,
where (\ref{eqn:theta_qp}) and 
(\ref{eqn:theta_prime}) were used.

\subsection{Ramanujan $\rho_1$-function, 
Jordan--Kronecker function and \\
weighted Szeg\H{o} kernel of $\A_q$}
\label{sec:S_qt_JK}
Assume that $q \in (0, 1)$. 
Consider the so-called 
\textit{Ramanujan $\rho_1$-function} 
\cite{Coo00,Ven12}
defined by
\begin{equation}
\rho_1(z)= \rho_1(z; q) = 
\frac{1}{2}+\sum_{n \in \Z \setminus \{0\}}
\frac{z^n}{1-q^{2n}}
\label{eqn:rho1}
\end{equation}
with $q^2 <|z| <1$.
As a generalization of $\rho_1$ 
the following function has been studied 
in \cite{MS94,Coo00,Ven12}, 
\begin{equation}
f^{\rm JK}(z, a) 
=f^{\rm JK}(z, a; q) 
:= \sum_{n \in \Z} \frac{z^n}{1-aq^{2n}}, 
\label{eqn:JK1}
\end{equation}
with $q^2 < |z| <1$, 
$a \notin \{ q^{2 i} : i \in \Z\}$,
which is called
the \textit{Jordan--Kronecker function}
(see \cite[p.59]{Ven12} and \cite[pp.70-71]{Wei76}).
\begin{prop}
\label{thm:S_qt_JK} 
Assume that $r >0$. 
Then the weighted Szeg\H{o} kernel of $\A_q$
(\ref{eqn:SAqr1}) 
is expressed by the Jordan--Kronecker function 
(\ref{eqn:JK1}) as
\begin{equation}
S_{\A_q}(z, w; r)=f^{\rm JK}(z \wbar, -r),
\quad z, w \in \A_q.
\label{eqn:S_qt_JK}
\end{equation}
In particular, the Szeg\H{o} kernel of $\A_q$ is
given by $S_{\A_q}(z, w)=f^{\rm JK}(z \wbar, -q)$,
$z, w \in \A_q$.
\end{prop}

The \textit{bilateral basic hypergeometric series}
in base $p$ with one numerator parameter $a$
and one denominator parameter $b$ is defined by 
\cite{GR04}
\[
{_1}\psi_1(a; b; p, z) 
= {_1}\psi_1 \Big[ \begin{array}{l} a \cr b \end{array}; p, z \Big]
:= \sum_{n \in \Z} \frac{(a;p)_{n}}{(b;p)_{n}} z^n,
\quad |b/a| < |z| < 1.
\]
The Jordan--Kronecker function (\ref{eqn:JK1}) 
is a special case
of the ${_1}\psi_1$ function \cite{Coo00,Ven12};
\[
f^{\rm JK}(z, a; q) = \frac{1}{1-a} {_1}\psi_1 (a; aq^2; q^2, z).
\]
The following equality is known as 
\textit{Ramanujan's ${_1}\psi_1$ summation formula} 
\cite{Coo00,GR04,Ven12}, 
\[
\sum_{n \in \Z} \frac{(a;p)_{n}}{(b;p)_{n}} z^n
=\frac{(az, p/(az), p, b/a;p)_{\infty}}
{(z, b/(az), b, p/a; p)_{\infty}},
\quad |b/a| < |z| < 1.
\]
Combining the above two equalities
with an appropriate change of variables, 
we obtain \cite{Coo00,Ven12}
\begin{equation}
f^{\rm JK}(z, a) = f^{\rm JK}(z, a; q) = 
\frac{(az, q^2/(az), q^2, q^2; q^2)_{\infty}}
{(z, q^2/z, a, q^2/a; q^2)_{\infty}}
= \frac{q_0^2 \theta(za; q^2)}
{\theta(z, a; q^2)},
\label{eqn:f2}
\end{equation}
where $q_0:=\prod_{n \in \N}(1-q^{2n})=(q^2; q^2)_{\infty}$.
Note that $\theta(z; q^2)$ is a holomorphic function
of $z$ in $\C^{\times}$.
Hence relying on (\ref{eqn:f2}), 
for every fixed $a$ in $\C^{\times} \setminus \{q^{2i} : i \in \Z\}$,
$f^{\rm JK}(\cdot, a)$ 
can be analytically continued to 
$\C^{\times} \setminus \{q^{2i} : i \in \Z\}$.
The poles are located exactly at the zeros of
$\theta(z; q^2)$ appearing in the denominator;
$\{ q^{2i} : i \in \Z\}$.
The following symmetries 
of $f^{\rm JK}$ are readily verified by (\ref{eqn:f2})
using (\ref{eqn:theta_inversion}) and (\ref{eqn:theta_qp})
\cite{Coo00,Ven12}.
\begin{align}
f^{\rm JK}(z, a) &= f^{\rm JK}(a, z),
\label{eqn:JK2a}
\\
f^{\rm JK}(z, a) &=-f^{\rm JK}(z^{-1}, a^{-1}),
\label{eqn:JK2b}
\\
f^{\rm JK}(z, a) &= z f^{\rm JK}(z, aq^2) = a f^{\rm JK}(zq^2, a).
\label{eqn:JK2c}
\end{align}
As shown in Chapter 3 in \cite{Ven12}, 
(\ref{eqn:JK1}) is rewritten as
\begin{align*}
f^{\rm JK}(z, a) =
\frac{1-z a}{(1-z)(1-a)}
&+ \sum_{n=1}^{\infty} q^{2 n^2} z^n a^n
\Big( 
1+ \frac{z q^{2n}}{1-z q^{2n}} 
+ \frac{a q^{2n}}{1-a q^{2n}} \Big)
\nonumber\\
&- \sum_{n=1}^{\infty} q^{2 n^2} z^{-n} a^{-n}
\Big( 
1+ \frac{z^{-1} q^{2n}}{1-z^{-1} q^{2n}} 
+ \frac{a^{-1} q^{2n}}{1-a^{-1} q^{2n}} \Big),
\end{align*}
which is completely symmetric in $z$ and $a$
and valid for 
$z, a \notin \{q^{2i} : i \in \Z \}$.
The equalities (\ref{eqn:JK2a})--(\ref{eqn:JK2c})
are proved also using this expression \cite{Ven12}.

From now on, we assume that $p=q^2$
and hence $\theta(\cdot)$ means
$\theta(\cdot; q^2)$ in the following.
We replace $z$ by $z \wbar$ and $a$ by $-r$ in (\ref{eqn:f2}).
Then Proposition \ref{thm:S_qt_JK} implies the following.
\begin{prop}[Mccullough and Shen \cite{MS94}]
\label{thm:S_Aq_theta}
For $r > 0$
\begin{equation}
S_{\A_q}(z, w; r)
= \frac{q_0^2 \theta(-r z \wbar)}
{\theta(-r, z \wbar)},
\quad z, w \in \A_q.
\label{eqn:S_qt_theta}
\end{equation}
In particular, 
\begin{equation}
S_{\A_q}(z, w) =S_{\A_q}(z, w; q)
= \frac{q_0^2 \theta(-q z \wbar)}
{\theta(-q, z \wbar)},
\quad z, w \in \A_q.
\label{eqn:S_q_theta}
\end{equation}
\end{prop}

Since $\theta(\cdot)$ is holomorphic in
the punctured complex plane 
$\C^{\times} :=\{z \in \C : |z| > 0\}$, 
by the expression (\ref{eqn:S_qt_theta}), 
$S_{\A_q}(z, w; r)$ can be analytically continued to $\C^{\times}$
as an analytic function of $z, r$ and 
an anti-analytic function of $w$. 
Actually the inversion formula (\ref{eqn:theta_inversion})
and the quasi-periodicity property (\ref{eqn:theta_qp})
of the theta function given in Section \ref{sec:S_qt_theta} 
imply the following functional equations, 
\begin{align}
&{\rm (i)} \quad S_{\A_q}(q^2 z, w; r) = - \frac{1}{r} S_{\A_q}(z, w; r),
\nonumber\\
&{\rm (ii)} \quad S_{\A_q}(1/z, w; r) = - S_{\A_q}(z, 1/w; 1/r),
\nonumber\\
&{\rm (iii)} \quad S_{\A_q}(z, w; q^2 r) 
= \frac{1}{z \wbar} S_{\A_q}(z, w; r). 
\label{eqn:symmetries1}
\end{align}
Then the following is easily verified.

\begin{lem}
\label{thm:zero_S}
Assume that $\alpha \in \A_q$.
Then $S_{\A_q}(z, \alpha; r)$ has zeros at
$z=-q^{2i}/(\alphabar r)$, $i \in \Z$ in $\C^{\times}$. 
In particular, 
$S_{\A_q}(z, \alpha)$ has a unique zero in $\A_q$ at
$z=\alphahat$ given by (\ref{eqn:alphabar}).
\end{lem}
\begin{proof}
Since $\theta$ is holomorphic in $\C^{\times}$, 
the expression (\ref{eqn:S_qt_theta}) implies 
that $S_{\A_q}(z, \alpha; r)$ is 
meromorphic in $\C^{\times}$. 
By (\ref{eqn:theta_zero}), $S_{\A_q}(z, \alpha; r)$ vanishes 
in $\C^{\times}$ only if $-z \alphabar r=q^{2i}$, $i \in \Z$.
By assumption $|\alpha| \in (q, 1)$.
Hence, when $r=q$, $|-q^{2i}/(\alphabar r)|=
q^{2i-1}/|\alpha| \in (q, 1)$,
if and only if $i=1$. 
\end{proof}

The second assertion of Lemma \ref{thm:zero_S} gives
the following probabilistic statement.

\begin{prop}
\label{thm:pair_structure}
For each $\alpha \in \A_q$, 
$X_{\A_q}(\alpha)$ and $X_{\A_q}(\alphahat)$
are mutually independent.
\end{prop}

\subsection{Weierstrass elliptic functions
and other functions}
\label{sec:Weierstrass_elliptic}

Here we show useful relations between
the theta function, 
Ramanujan $\rho_1$-function, 
Jordan--Kronecker function, 
and \textit{Weierstrass elliptic functions}.

Assume that $\omega_1$ and $\omega_3$ are 
complex numbers such that 
if we set $\tau=\omega_3/\omega_1$, then
$\Im \tau > 0$.
The lattice $\LL(\omega_1, \omega_3)$ on $\C$
with lattice generators $2 \omega_1$ and $2 \omega_3$ 
is given by
\[
\LL=\LL(\omega_1, \omega_3)
:= \{ 2m \omega_1 + 2n \omega_3 : (m, n) \in \Z^2 \}.
\]
The \textit{Weierstrass $\wp$-function} 
and \textit{$\zeta$-function} are defined by
\begin{align}
\wp(\phi) &= \wp(\phi | 2 \omega_1, 2 \omega_3)
:=\frac{1}{\phi^2} + \sum_{v \in \LL(\omega_1, \omega_3) 
\setminus \{0\}}
\left[ \frac{1}{(\phi-v)^2} - \frac{1}{v^2} \right], 
\nonumber\\
\zeta(\phi) &= \zeta(\phi | 2 \omega_1, 2 \omega_3)
:=\frac{1}{\phi} + \sum_{v \in \LL(\omega_1, \omega_3) 
\setminus \{0\}}
\left[ \frac{1}{\phi-v} +\frac{1}{v}+\frac{\phi}{v^2} \right].
\label{eqn:Weierstrass1}
\end{align}
(See, for instance, Section 23 in \cite{NIST10}.)
We put $\omega_2=-(\omega_1+\omega_3)$.
By the definition (\ref{eqn:Weierstrass1}) we see that
$\wp(\phi)$ is even and $\zeta(\phi)$ is odd 
with respect to $\phi$, and
$\wp(\phi)$ is an elliptic function
(i.e., a doubly periodic meromorphic function in $\C$);
$\wp(\phi+2 \omega_{\nu})=\wp(\phi)$, $\nu=1,2,3$.
We note that $\wp'(\omega_{\nu})=0$, $\nu=1,2,3$, 
$\wp(\phi)=-\zeta'(\phi)$, 
and 
$\zeta(\phi+2 \omega_{\nu})=\zeta(\phi)+ 2 \eta_{\nu}$
where $\eta_{\nu} :=\zeta(\omega_{\nu}), \nu=1,2,3$.
In the present paper we consider the following setting; 
\begin{align}
\omega_1=\pi, \quad &\frac{\omega_3}{\omega_1} = \tau_q, 
\quad \mbox{and} 
\nonumber\\
q=e^{\sqrt{-1} \pi \tau_q} \in (0, 1) 
\iff& \,
\tau_q =- \sqrt{-1} \frac{\log q}{\pi}
\in \sqrt{-1} \R_{> 0}. 
\label{eqn:omega1}
\end{align}
In the terminology of \cite[page 304]{GR04},
when we regard $p:=q^2$ as the nome of the theta function, 
$\tau_q$ shall be called the \textit{nome modular parameter},
and when we regard $q =p^{1/2} =: e^{2 \sqrt{-1} \pi \sigma_q}$
as the \textit{base} of $q$-special functions, 
$\tau_q$ will be the twice of the 
\textit{base modular parameter} $\sigma_q$.
In this setting, the $\wp$-function is considered as a function
of an argument $\phi$ and the modular parameter
$\tau_q$ though $q$. 
Then we have the following expansions, 
\begin{align}
\wp(\phi)
=\wp(\phi; \tau_q)
&= - \frac{1}{12} 
+ 2 \sum_{n=1}^{\infty} \frac{q^{2n}}{(1-q^{2n})^2}
+ \frac{1}{4} \frac{1}{\sin^2(\phi/2)}
-2 \sum_{n=1}^{\infty} \frac{n q^{2n}}{1-q^{2n}}
\cos (n \phi)
\nonumber\\
&= - \frac{1}{12} + 2 \sum_{n=1}^{\infty}
\frac{q^{2n}}{(1-q^{2n})^2} 
- \sum_{n=-\infty}^{\infty} \frac{e^{\sqrt{-1} \phi} q^{2n}}
{(1-e^{\sqrt{-1} \phi} q^{2n})^2}.
\label{eqn:wp_expansion2}
\end{align}
We use the notation
\begin{equation}
z=e^{\sqrt{-1} \phi_z} \iff
\phi_z= - \sqrt{-1} \log z.
\label{eqn:x_phi}
\end{equation}
Then $\phi_{zw}=\phi_z+\phi_w$, 
$\phi_{z^{-1}}=-\phi_z$, 
and $\phi_{q^2}=2 \omega_3$ modulo $2\pi \Z$.
Hence the evenness and the periodicity property 
of $\wp$ are written as
\begin{equation}
\wp(-\phi_z)=
\wp(\phi_{z^{-1}}) =\wp(\phi_z),
\quad
\wp(\phi_{q^2 z})
=\wp(\phi_z).
\label{eqn:wp_symmetry}
\end{equation}
The expansion (\ref{eqn:wp_expansion2}) is written as
\begin{align}
&\wp(\phi_{z})=\wp(\phi_{z}; \tau_q)
= - \frac{1}{12} - \frac{z}{(1-z)^2}
+ 2 \sum_{n=1}^{\infty}
\frac{q^{2n}}{(1-q^{2n})^2} 
-\sum_{n=1}^{\infty} 
\frac{n q^{2n}}{1-q^{2n}}
\left(z^n+\frac{1}{z^n} \right)
\nonumber\\
&\quad
= - \frac{1}{12} - \frac{z}{(1-z)^2}
+ 2 \sum_{n=1}^{\infty}
\frac{q^{2n}}{(1-q^{2n})^2} 
-\sum_{n=1}^{\infty} 
\frac{z q^{2n}}{(1-z q^{2n})^2}
-\sum_{n=1}^{\infty} 
\frac{z^{-1} q^{2n}}{(1-z^{-1} q^{2n})^2}.
\label{eqn:wp_expansion4}
\end{align}
The special values of $\wp$ are denoted by 
\begin{align}
e_1 = e_1(q) &:= \wp(\pi)
=\wp(\phi_{-1}; \tau_q)
\nonumber\\
&= \frac{1}{6} + 2 \sum_{n=1}^{\infty} \frac{q^{2n}}{(1-q^{2n})^2}
+ 2 \sum_{n=1}^{\infty} \frac{q^{2n}}{(1+q^{2n})^2},
\nonumber\\
e_2 = e_2(q) &:= \wp(\pi + \pi \tau_q)
=\wp(\phi_{-q}; \tau_q)
\nonumber\\
&= -\frac{1}{12} 
+ 2 \sum_{n=1}^{\infty} \frac{q^{2n}}{(1-q^{2n})^2}
+ 2 \sum_{n=1}^{\infty} \frac{q^{2n-1}}{(1+q^{2n-1})^2},
\nonumber\\
e_3 = e_3(q) &:= \wp(\pi \tau_q)
=\wp(\phi_{q}; \tau_q)
\nonumber\\
&= -\frac{1}{12} 
+ 2 \sum_{n=1}^{\infty} \frac{q^{2n}}{(1-q^{2n})^2}
- 2 \sum_{n=1}^{\infty} \frac{q^{2n-1}}{(1-q^{2n-1})^2}.
\label{eqn:e123}
\end{align}
We see that
\begin{equation}
e_1+e_2+e_3=0,
\label{eqn:sum0}
\end{equation}
and define
\begin{align}
g_2 &=g_2(q)
:=2 (e_1^2+e_2^2+e_3^2)
= -4 (e_2 e_3 +e_3 e_1 +e_1 e_2) >0,
\nonumber\\
g_3 &=g_3(q)
:=4 e_1 e_2 e_3 = \frac{4}{3}(e_1^3+e_2^3+e_3^3). 
\label{eqn:g2g3}
\end{align}

The \textit{imaginary transformation} of $\wp$ 
is given by \cite[p.31]{Ven12},
$\wp(\phi; \tau_q)
=\tau_q^{-2} \wp (\phi/\tau_q; - 1/\tau_q )$.
Hence (\ref{eqn:wp_expansion2}) is written as
\begin{align}
\wp(\phi)=\wp(\phi; \tau_q)
&= \frac{1}{|\tau_q|^2} \Big[
 \frac{1}{12} 
+ \frac{1}{4} \frac{1}{\sinh^2(\phi/(2 |\tau_q|))}
\nonumber\\
& \qquad
- 2 \sum_{n=1}^{\infty} 
\frac{e^{-2n \pi/|\tau_q|}}{(1-e^{-2n \pi/|\tau_q|})^2}
+2 \sum_{n=1}^{\infty} 
\frac{n e^{-2n \pi /|\tau_q|}}{1-e^{-2n \pi /|\tau_q|}}
\cosh(n \phi/|\tau_q|)
\Big],
\label{eqn:imaginary2}
\end{align}
where we used the relation $\tau_q=\sqrt{-1} |\tau_q|$
which is valid 
in the present setting (\ref{eqn:omega1}). 

It can be verified that 
$\wp$ satisfies the following differential equations
\cite[Section 23]{NIST10}),
\begin{align}
\wp'(\phi)^2
&= 4 \wp(\phi)^2- g_2 \wp(\phi) - g_3
\nonumber\\
&= 4 (\wp(\phi)-e_1)(\wp(\phi)-e_2)(\wp(\phi)-e_3),
\label{eqn:wp_diff0}
\\
\wp''(\phi) &= 6 \wp(\phi)^2 - \frac{g_2}{2}.
\label{eqn:wp_diff}
\end{align}

When $q \in (0, 1)$,
$e_1, e_2, e_3 \in \R$ 
and the following inequalities hold
(\cite[Section 2.8]{MM97}),
\begin{equation}
e_3 < e_2 < e_1. 
\label{eqn:inequality}
\end{equation}
From (\ref{eqn:wp_diff0}), we see that 
$\wp$ inverts the incomplete elliptic integral \cite{Law89,MM97}.
Under the setting (\ref{eqn:omega1}), 
we will use the following special result 
\cite[(23.6.31)]{NIST10} (see Section 6.12 of \cite{Law89});
if $e_2 \leq x \leq e_1$, then
$\wp^{-1}(x) \in [\omega_1, \omega_1+\omega_3]
:=\{\pi + \sqrt{-1} y : 0 \leq y \leq \pi |\tau_q|\}$
and
\begin{equation}
y=\frac{1}{2} \int_x^{e_1} 
\frac{ds}{\sqrt{(e_1-s)(s-e_2)(s-e_3)}}.
\label{eqn:wp_inverse}
\end{equation}

We introduce the Euler operator
\begin{equation}
\cD_z=z \frac{\partial}{\partial z}.
\label{eqn:Euler}
\end{equation}
If we use the notation (\ref{eqn:x_phi}),  then
$\cD_z = - \sqrt{-1} \partial/\partial \phi_z$.  

\begin{lem}
\label{thm:fundamental_eqs}
Under the notation (\ref{eqn:x_phi}), 
the following equalities hold,
\begin{align}
f^{\rm JK}(z, a) f^{\rm JK}(z, b) 
&= \cD_z f^{\rm JK}(z, ab)
+(\rho_1(a)+\rho_1(b)) f^{\rm JK}(z, ab),
\label{eqn:fundamental1}
\\
f^{\rm JK}(z, a) f^{\rm JK}(z, a^{-1})
&= \cD_z \rho_1(z) -\cD_a \rho_1(a), 
\label{eqn:fundamental2}
\\
\cD_z \rho_1(z)
&=
-\sqrt{-1} \frac{d}{d \phi_z} \rho_1(z)
=-\wp(\phi_z) + \frac{P}{12},
\label{eqn:fundamental3}
\\
f^{\rm JK}(z, a) f^{\rm JK}(z, a^{-1})
&
=\wp(\phi_a)-\wp(\phi_z),
\label{eqn:addition_f}
\\
f^{\rm JK}(z, -1)^2
&=e_1-\wp(\phi_z),
\label{eqn:fundamental4}
\end{align}
where
\[
P =P(q) = 1-24 \sum_{n=1}^{\infty} \frac{q^{2n}}{(1-q^{2n})^2}
=\frac{12}{\pi} \eta_1(q)
=1-24 \sum_{n=1}^{\infty} \frac{n q^{2n}}{1-q^{2n}}.
\]
\end{lem}
\noindent
The equality (\ref{eqn:fundamental1}) is
called the 
\textit{fundamental multiplicative identity
of the Jordan--Kronecker function} in \cite{Coo00,Ven12}.
The equality (\ref{eqn:fundamental2}) is obtained 
by taking the limit $b \to 1/a$ in
(\ref{eqn:fundamental1}) \cite{Coo00}.
The derivation of (\ref{eqn:fundamental3})
is also found in \cite{Coo00}.
Combination of (\ref{eqn:fundamental2})
and (\ref{eqn:fundamental3}) gives
(\ref{eqn:addition_f}). 
The equality (\ref{eqn:fundamental4}) is 
a special case of (\ref{eqn:addition_f}) with
$a=-1$ where
the definition of 
$e_1$ is used.

We set 
\begin{equation}
 a_n(z) := \cD_z^n \log \theta(z), \quad n \in \N.
\label{eqn:an_x}
\end{equation}
\begin{lem}
\label{thm:an} 
The following equalities hold,
\begin{align*}
a_1(z) &= \frac{1}{2} - \rho_1(z), \quad
a_2(z) = \wp(\phi_z) - \frac{P}{12}, 
\nonumber\\
a_3(z) &= -\sqrt{-1} \wp'(\phi_z), \quad
a_4(z) = - \wp''(\phi_z).
\end{align*}
\end{lem}
\begin{proof} 
For $a_1(z)$ we have
\begin{align}
a_1(z) 
&= z \frac{\theta'(z)}{\theta(z)} 
= -\frac{z}{1-z} - \sum_{n=1}^{\infty} \Big(
\frac{z q^{2n}}{1-z q^{2n}} - \frac{z^{-1} q^{2n}}{1-z^{-1} q^{2n}} 
\Big) 
\label{eqn:a1minus}
\\
&= -\frac{z}{1-z} - \Big\{ \rho_1(z) - \frac{1+z}{2(1-z)} \Big\}
= \frac{1}{2} - \rho_1(z). 
\nonumber
\end{align} 
For $a_2(z)$ use 
(\ref{eqn:fundamental3}) in Lemma \ref{thm:fundamental_eqs}.
Use $\cD_z=-\sqrt{-1} \partial/\partial \phi_z$
for $a_3(z)$ and $a_4(z)$.
\end{proof}
\begin{lem}
\label{thm:gammas}
The following equalities holds,
\begin{align*}
&{\rm (i)} \quad \lim_{z \to 1} \Big(a_1(z) + \frac{z}{1-z} \Big) 
= 0, \\
&{\rm (ii)} \quad \gamma_2 
:= \lim_{z \to 1} \Big\{ a_2(z) + \frac{z}{(1-z)^2} \Big\} 
= -2\sum_{n=1}^{\infty} 
\frac{q^{2n}}{(1-q^{2n})^2} = \frac{P-1}{12}, \\
&{\rm (iii)} \quad a_1(-1) = \frac{1}{2}, \\
&{\rm (iv)} \quad a_2(-1) 
= \frac{1}{4} + 2\sum_{n=1}^{\infty} \frac{q^{2n}}{(1+q^{2n})^2}. 
\end{align*}
\end{lem}
\begin{proof} 
We notice (\ref{eqn:a1minus}) and 
\[
a_2(z) = -\frac{z}{(1-z)^2} 
- \sum_{n=1}^{\infty} 
\Big\{
\frac{z q^{2n}}{(1-z q^{2n})^2}
+ \frac{z^{-1} q^{2n}}{(1-z^{-1} q^{2n})^2}
\Big\}. 
\]
The formulas (i)--(iv) are all obtained from
these equalities.
\end{proof}

We note that the following is the case,
\[
\theta'(-1)= -\theta(-1)/2 \iff
a_1(-1) = 1/2 \iff
\rho_1(-1) = 0.
\]

\subsection{$q \to 0$ limits
and asymptotics in $q \to 1$}
\label{sec:q_0}
By the definition (\ref{eqn:SAqr1}), 
the following are readily confirmed; 
\begin{align*}
\lim_{q \to 0} S_{\A_q}(z, w)
&= S_{\D}(z,w),
\nonumber\\
\lim_{q \to 0} S_{\A_q}(z, w; r)
&= \frac{1+r z \wbar}{(1+r)(1-z \wbar)}
=\frac{1}{1-z \wbar} - \frac{r}{1+r}
=: S_{\D}(z, w; r),
\nonumber\\
\lim_{r \to 0} S_{\D}(z, w; r)
&= S_{\D}(z, w).
\end{align*}
Notice that 
if we use the expressions 
(\ref{eqn:S_qt_theta}) and (\ref{eqn:S_q_theta})
in Proposition \ref{thm:S_Aq_theta}, 
their $q \to 0$ limits are immediately obtained by 
(\ref{eqn:theta_p0}) with $p=q^2$.

By (\ref{eqn:wp_expansion2}) and (\ref{eqn:wp_expansion4}), 
we see the following,
\begin{align}
\lim_{q \to 0} \wp(\phi; \tau_q)
&= - \frac{1}{12} + \frac{1}{4 \sin^2(\phi/2)},
\nonumber\\
\lim_{q \to 0} 
\wp(\phi_{z}; \tau_q)
&=-\frac{1}{12} - \frac{z}{(1-z)^2}
=- \frac{1+10z+z^2}{12(1-z)^2}.
\label{eqn:q0_limit1}
\end{align}
Similarly, 
(\ref{eqn:e123}) and (\ref{eqn:g2g3}) give
\begin{equation}
e_1(0) = 1/6,
\quad
e_2(0)= e_3(0)= - 1/12,
\quad
g_2(0) = 1/12, 
\quad
g_3(0) = 1/216.
\label{eqn:q0_limit2}
\end{equation}

In the present setting (\ref{eqn:omega1}),
$q \to 1 \iff 
|\tau_q| \to 0$.
For $\Re \, \phi \in (0, 2\pi)$,
(\ref{eqn:imaginary2}) gives the following
asymptotics in $|\tau_q| \to 0$,
\begin{align}
\wp(\phi, \tau_q)
&\sim 
( 1/12 + e^{-\phi/|\tau_q|} + e ^{-(2\pi-\phi)/|\tau_q|})/|\tau_q|^2, 
\nonumber\\
e_1
&\sim ( 1/12 + 2 e^{-\pi/|\tau_q|})/|\tau_q|^2,
\quad
e_2
\sim (1/12 - 2 e^{-\pi/|\tau_q|})/|\tau_q|^2.
\label{eqn:q1asym}
\end{align}
By (\ref{eqn:sum0}), the above implies
\begin{equation}
e_3 =-(e_1+e_2) 
\sim -1/(6 |\tau_q|^2),
\quad
g_2 
\sim 
( 1+ 4 e^{-\pi/|\tau_q|})/|\tau_q|^2.
\label{eqn:q1asym2}
\end{equation}

\subsection{Conditional weighted Szeg\H{o} kernels}
\label{sec:conditional_S}
For $r >0$, define
\begin{equation}
S_{\A_q}^{\alpha}(z, w; r)
:= S_{\A_q}(z, w; r) 
- \frac{S_{\A_q}(z, \alpha; r) S_{\A_q}(\alpha, w; r)}
{S_{\A_q}(\alpha, \alpha; r)},
\quad z, w, \alpha \in \A_q.
\label{eqn:conditionSat}
\end{equation}
We put
$S_{\A_q}^{\alpha}(z, w; r)
=f(z, w; r, \alpha)
h_{\alpha}^q(z) \overline{h_{\alpha}^q(w)}$
assuming $\overline{f(w, z; r, \alpha)}
=f(z, w; r, \alpha)$
and here we intend to determine $f$. 
By the definition of the 
conditional kernel (\ref{eqn:conditionS}),
we can verify that 
$S_{\A_q}^{\alpha}$ satisfies the same functional equations
with (\ref{eqn:symmetries1}) (i) and (iii);
$S_{\A_q}^{\alpha}(q^2 z, w; r)$ 
$= - (1/r) S_{\A_q}^{\alpha}(z, w; r)$,
$S_{\A_q}^{\alpha}(z, w; q^2 r)$ 
$= (1/z \wbar) S_{\A_q}^{\alpha}(z, w; r)$,
but in the equation corresponding to 
(\ref{eqn:symmetries1}) (ii) 
the conditioning parameter $\alpha$ should be also inverted as
$S_{\A_q}^{\alpha}(1/z, w; r)$ 
$= - S_{\A_q}^{1/\alpha}(z, 1/w; 1/r)$.
Moreover 
(\ref{eqn:symmetries1}) (i) implies
$S_{\A_q}^{q^2 \alpha}(z, w; r)=S_{\A_q}^{\alpha}(z, w; r)$.
On the other hand, (\ref{eqn:haq}) gives
$h^q_{\alpha}(q^2 z)=|\alpha|^2 h^q_{\alpha}(z)$, 
$h^q_{q^2 \alpha}(z)=z^2 (\overline{\alpha}/\alpha) h^q_{\alpha}(z)$, 
and
$h^q_{\alpha}(1/z)=(\alpha/\overline{\alpha}) h^q_{1/\alpha}(z)$. 
Hence $f$ should satisfy the functional equations
\begin{align*}
&{\rm (i)} \quad f(q^2 z, w; r, \alpha) =-\frac{1}{r |\alpha|^2} 
f(z, w; r, \alpha),
\nonumber\\
&{\rm (ii)} \quad  f(1/z, w; r, \alpha) = 
-f(z, 1/w; 1/r, 1/\alpha),
\nonumber\\
&{\rm (iii)} \quad  f(z, w; q^2 r, \alpha) = 
\frac{1}{z \wbar} f(z, w; r, \alpha),
\nonumber\\
&{\rm (iv)} \quad f(z, w; r, q^2 \alpha) = \frac{1}{(z \wbar)^2}
f(z, w; r, \alpha).
\end{align*}
Comparing them with (\ref{eqn:symmetries1}), 
it is easy to verify that if
$f(z,w; r, \alpha)=S_{\A_q}(z,w; r|\alpha|^2)$,
these functional equations are satisfied.
The above observation implies the 
equality (\ref{eqn:SaAq}).
Actually, Mccullough and Shen proved the following.
\begin{prop}[Mccullough and Shen \cite{MS94}]
\label{thm:MS}
The equality (\ref{eqn:SaAq}) holds
with (\ref{eqn:haq}). 
\end{prop}
Mccullough and Shen proved the above 
by preparing an auxiliary lemma. 
Here we give a direct proof from
Weierstrass' addition formula
(\ref{eqn:Weierstrass_add1}).
\begin{proof}
We put (\ref{eqn:conditionSat}) with (\ref{eqn:S_qt_theta})
and (\ref{eqn:haq}) to (\ref{eqn:SaAq}), then
the equality is expressed by
theta functions. After multiplying both sides by
the common denominator, we see that the
equality (\ref{eqn:SaAq}) is equivalent to the following,
\begin{align}
&
\theta(-r z \wbar, -r |\alpha|^2, \alphabar z, \alpha \wbar)
- \theta(-r \alphabar z, -r \alpha \wbar, z \wbar, |\alpha|^2)
\nonumber\\
& \qquad 
=z \wbar \theta(-r z \wbar |\alpha|^2, \alpha z^{-1}, 
\alphabar \, \wbar^{-1}, -r).
\label{eqn:Weierstrass_add2}
\end{align}
Now we change the variables from
$\{z, \wbar, \alpha, r\}$ to
$\{x, y, u, v \}$ as
$\alphabar z =x/y$, 
$\alpha \wbar=u/v$, 
$z \wbar = x/v$, 
$|\alpha|^2= u/y$, and
$r=-yv$.
Then the left-hand side 
of (\ref{eqn:Weierstrass_add2}) becomes 
$\theta(xy, x/y, uv, u/v)
-\theta(xv, x/v, uy, u/y)$, 
and the right-hand side becomes 
$(x/v) \theta(yv, (y/v)^{-1}, xu, (x/u)^{-1})$
which is equal to 
$(u/y) \theta(yv, y/v, xu, x/u)$
by (\ref{eqn:theta_inversion}).
Hence Weierstrass' addition formula
(\ref{eqn:Weierstrass_add1}) proves the equality
(\ref{eqn:Weierstrass_add2}). 
The proof is complete.
\end{proof}

We can prove the following.
\begin{lem}
\label{thm:Blaschke}
For $\alpha \in \A_q$,
\begin{description}
\item{\rm (i)} \quad
$h_{\alpha}^q(\alpha) =0$,

\item{\rm (ii)} \quad
$0< |h_{\alpha}^q(z)| < 1
\quad \forall z \in \A_q \setminus \{\alpha\}$,

\item{\rm (iii)} \quad
$|h_{\alpha}^q(z)| =
\begin{cases}
1, & \quad \mbox{if $z \in \gamma_1 :=\{z \in \C: |z|=1\}$},
\cr
|\alpha|, & \quad \mbox{if $z \in \gamma_q := \{z \in \C: |z|=q \}$},
\end{cases}
$

\item{\rm (iv)} \quad
$
\displaystyle{
{h_{\alpha}^q}'(\alpha)
= - \frac{\theta'(1)}{\theta(|\alpha|^2)}
=\frac{q_0^2}{\theta(|\alpha|^2)}>0
}$,

\item{\rm (v)} \quad
$
\displaystyle{
\lim_{q \to 0} h_{\alpha}^q(z)
= \frac{z-\alpha}{1-z \alphabar}
}$.
\end{description}
\end{lem}
\begin{proof}
When $w=z$, (\ref{eqn:conditionSat}) gives
$S_{\A_q}^{\alpha}(z,z; r)
=S_{\A_q}(z,z; r)-|S_{\A_q}(z, \alpha; r)|^2/S_{\A_q}(\alpha, \alpha; r)
\geq 0$, $z \in \A_q$,
which implies 
$0 \leq S_{\A_q}^{\alpha}(z, z; r)/S_{\A_q}(z, z; r) \leq 1$,
$z \in \A_q$.
As noted just after (\ref{eqn:SAqr1}), 
$S_{\A_q}(z,z;r)$ is monotonically decreasing in $r>0$.
Then, by (\ref{eqn:SaAq}), 
$S_{\A_q}^{\alpha}(z,z;r)=S_{\A_q}(z,z; r |\alpha|^2) |h^q_{\alpha}(z)|^2
> S_{\A_q}(z,z; r) |h^q_{\alpha}(z)|^2$,
if $|\alpha| <1$.
Hence it is proved that
$|h^q_{\alpha}(z)| < S_{\A_q}^{\alpha}(z, z; r)/S_{\A_q}(z, z; r)  
\leq 1$, $z \in \A_q$.
By the explicit expression (\ref{eqn:haq})
and by basic properties of the theta function
given in Section \ref{sec:S_qt_theta},
provided 
$z \in \overline{\A_q} :=\A_q \cup \gamma_1 \cup \gamma_q$, 
it is verified that 
$h^q_{\alpha}(z)=0$ if and only if 
$z=\alpha$, 
and $|h^q_{\alpha}(z)|=1$
if and only if
$z \in \gamma_1$ .
Using (\ref{eqn:theta_qp}) and (\ref{eqn:theta_real}),
we can show that
\[
h^q_{\alpha}(q e^{\sqrt{-1} \phi})
=q e^{\sqrt{-1} \phi}
\frac{\theta(\alpha q^{-1} e^{-\sqrt{-1} \phi})}
{\theta( q^2 \alphabar q^{-1} e^{\sqrt{-1} \phi})}
=-\alphabar e^{2 \sqrt{-1} \phi}
\frac{\overline{\theta(\alphabar q^{-1} e^{\sqrt{-1} \phi})}}
{\theta( \alphabar q^{-1} e^{\sqrt{-1} \phi})},
\quad \phi \in [0, 2 \pi).
\]
Then (i)--(iii) are proved.
If we apply (i) and (\ref{eqn:theta_prime}) to
the derivative of (\ref{eqn:haq}) with respect to $z$,
then (iv) is obtained.
Applying (\ref{eqn:theta_p0}) to (\ref{eqn:haq}) 
proves (v). 
The proof is complete.
\end{proof}

Since $h^q_{\alpha}(\cdot), \alpha \in \A_q$ is 
holomorphic in $\A_q$, (i)--(iii) of Lemma \ref{thm:Blaschke}
implies that $h^q_{\alpha}$ gives a conformal map
from $\A_{q}$ to $\D \setminus \{\mbox{a circular slit}\}$ 
as shown by Figure \ref{fig:h_alpha}.
In addition (v) of Lemma \ref{thm:Blaschke} means
$\lim_{q \to 0} h_{\alpha}^q(z)=h_{\alpha}(z)$,
where $h_{\alpha}(z)$ is the Riemann mapping function
associated to $\alpha$ in $\D$
given by a M\"{o}bius transformation (\ref{eqn:Mob1}).

\begin{rem}
\label{rem:Blaschke}
The present function $h^q_{\alpha}(z)$ is closely related with
the \textit{Blaschke factor} 
defined on page 17 in \cite{Sar65} 
for an annulus
$\A_{q^{1/2}, q^{-1/2}}
:=\{z \in \C: q^{1/2} < |z| < q^{-1/2} \}$, 
whose explicit expression
using the theta functions was given on pp.~386--388 in \cite{CH04}.
These two functions are, however, different from each other.
Let $\widehat{h}^q_{\alpha}(z)$
denote the Blaschke factor for the annulus $\A_q$, 
which is appropriately transformed from the function
given in \cite{CH04} for $\A_{q^{1/2}, q^{-1/2}}$. 
We found that
\[
\widehat{h}^q_{\alpha}(z)
=z^{-\log \alpha/\log q} h^q_{\alpha}(z).
\]
Also for the Blaschke factor $\widehat{h}^q_{\alpha}(z)$,
(i), (ii), and (v) in Lemma \ref{thm:Blaschke} are satisfied,
but instead of (iii), we have
$|\widehat{h}^q_{\alpha}(z)|=1$ 
if and only if
$z \in \gamma_1 \cup \gamma_q$ for $z \in \overline{\A_q}$.
Moreover, $\widehat{h}^q_{\alpha}$ is 
not univalent in $\A_q$ and is branched. 
\end{rem}

\subsection{Correlation functions of point processes
and the DPP of Peres and Vir\'ag on $\D$ }
\label{sec:Peres_Virag}

A point process is formulated as follows.
Let $S$ be a base space, which is locally compact Hausdorff space
with a countable base, and $\lambda$ be a Radon measure on $S$. 
The configuration space of a point process on $S$ is given by
the set of nonnegative-integer-valued Radon measures; 
\[
\Conf(S)
=\Big\{ \xi = \sum_i \delta_{x_i} : \mbox{$x_i \in S$,
$\xi(\Lambda) < \infty$ for all bounded set $\Lambda \subset S$} 
\Big\}.
\]
$\Conf(S)$ is equipped with the topological Borel $\sigma$-fields
with respect to the vague topology. 
A point process on $S$ is
a $\Conf(S)$-valued random variable $\Xi=\Xi(\cdot)$.
If $\Xi(\{ x \}) \in \{0, 1\}$ for any point $z \in S$,
then the point process is said to be {\it simple}.
Assume that $\Lambda_i, i=1, \dots, m$, 
$m \in \N$ are
disjoint bounded sets in $S$ and
$k_i \in \N_0, i=1, \dots, m$ satisfy
$\sum_{i=1}^m k_i = n \in \N_0$.
A symmetric measure $\lambda^n$ on $S^n$
is called the $n$-th {\it correlation measure},
if it satisfies
\[
\bE \Bigg[
\prod_{i=1}^m \frac{\Xi(\Lambda_i)!}
{(\Xi(\Lambda_i)-k_i)!} \Bigg]
=\lambda^n(\Lambda_1^{k_1} \times \cdots
\times \Lambda_m^{k_m}),
\]
where when $\Xi(\Lambda_i)-k_i < 0$,
we interpret $\Xi(\Lambda_i)!/(\Xi(\Lambda_i)-k_i)!=0$.
If $\lambda^n$ is absolutely continuous 
with respect to the $n$-product measure $\lambda^{\otimes n}$,
the Radon--Nikodym derivative
$\rho^n(x_1, \dots, x_n)$ is called
the {\it $n$-point correlation function}
with respect to the reference measure $\lambda$;
that is, 
$\lambda^n(dx_1 \cdots dx_n)
=\rho^n(x_1, \dots, x_n) 
\lambda^{\otimes n}(dx_1 \cdots dx_n)$.

Consider the case in which $S$ is given by a domain 
$\widetilde{D} \subset \C$ and 
$\Xi=\sum_{i} \delta_{Z_i}$ is a point process on $\widetilde{D}$
associated with the correlation functions 
$\{\rho_{\widetilde{D}}^n \}_{n \in \N}$. 
Here we assume that the reference measure $\lambda$
is given by the Lebesgue measure $m$ on $\C$
multiplied by a constant for simplicity
(e.g., $\lambda=m/\pi$).
For a one-to-one measurable transformation
$F:D \to  \widetilde{D}$, $D \subset \C$, 
we write the pull-back of the point process 
from $\widetilde{D}$ to $D$ as
$F^* \Xi :=\sum_i \delta_{F^{-1}(Z_i)}$.
We assume that $F$ is analytic and
$F'(z)=dF(z)/dz, z \in D$ is well-defined.
By definition the following is derived.
\begin{lem}
\label{thm:transformation_rho}
The point process $F^*\Xi$ on $D$ has 
correlation functions $\{\rho^n_D\}_{n \in \N}$ 
with respect to $\lambda$ given by
\[
\rho^n_D(z_1, \dots, z_n)
=\rho_{\widetilde{D}}^n(F(z_1), \dots, F(z_n)) 
\prod_{i=1}^n |F'(z_i)|^2,
\quad n \in \N, \quad
z_1, \dots, z_n \in D.
\]
The unfolded 2-correlation function
(\ref{eqn:unfolded}) is hence invariant under transformation,
\[
g_D(z_1, z_2)=
g_{\widetilde{D}}(F(z_1), F(z_2)),
\quad z_1, z_2 \in D.
\]
\end{lem}
\noindent
For a point process $\Xi=\sum_{i} \delta_{Z_i}$ on 
$D \subset \C$,
assume that there is a measurable function
$K^{\det}_D: D \times D \to \C$ such that 
the correlation functions are given by
the determinants of $K^{\det}_D$; that is,
\[
\rho_D^n(z_1, \dots, z_n)
=\det_{1 \leq i, j \leq n} 
[K^{\det}_D(z_i, z_j)]
\quad \mbox{
for every $n \in \N$ and
$z_1, \dots, z_n \in D$}
\]
with respect to $\lambda$. 
Then $\Xi$ is said to be a 
\textit{determinantal point process} (DPP) on $D$
with the \textit{correlation kernel} $K^{\det}_D$. 
For a one-to-one measurable transformation
$F:D \to \widetilde{D}$, $D \subset \C$
with a bounded derivative $F'$, 
$F^* \Xi$ is also a DPP on $D$
such that the correlation kernel 
with respect to $\lambda$ is given by
\begin{equation}
K^{\det}_{D}(z, w) :=|F'(z)| |F'(w)|
K^{\det}_{\widetilde{D}}(F(z), F(w)), \quad z, w \in D.
\label{eqn:transform_C}
\end{equation}
See \cite{ST00,Sos00,ST03a,ST03b,HKPV09,KS20,KS21}
for general construction and 
basic properties of determinantal point processes.

The zero point process $\cZ_{X_{\D}}$ of the GAF $X_{\D}$ 
defined by (\ref{eqn:GAF_D}) has the unit circle $\partial \D$ as a natural
boundary and $\cZ_{X_{\D}}(\D)=\infty$ a.s. 
For this zero process, Peres and Vir\'ag \cite{PV05} showed the following
remarkable result. 
\begin{thm}[Peres and Vir\'ag \cite{PV05}]
\label{thm:Peres_Virag}
$\cZ_{X_{\D}}$ is a DPP on $\D$
such that the correlation kernel 
with respect to $m/\pi$ 
is given by
the Bergman kernel $K_{\D}$ of $\D$ given by
(\ref{eqn:K_D}).
\end{thm}

The distribution of $\cZ_{X_{\D}}$ is invariant under
M\"{o}bius transformations that preserve $\D$ \cite{ST04,PV05}.
This invariance is a special case of the following,
which can be proved using
the conformal transformations 
of the Szeg\H{o} kernel and the Bergman kernel 
given by (\ref{eqn:SD_KD_conformal}) 
\cite{PV05,HKPV09}. 

\begin{prop}[Peres and Vir\'ag \cite{PV05}]
\label{thm:Peres_Virag2}
Let $\widetilde{D} \subsetneq \C$ be a 
simply connected domain with
$\cC^{\infty}$ boundary. 
Then there is a GAF $X_{\widetilde{D}}$ with 
covariance kernel
$\bE[X_{\widetilde{D}}(z) \overline{X_{\widetilde{D}}(w)}] 
= S_{\widetilde{D}}(z,w)$, $z, w \in \widetilde{D}$,
where $S_{\widetilde{D}}$ denotes 
the Szeg\H{o} kernel of $\widetilde{D}$.
The zero point process $\cZ_{X_{\widetilde{D}}}$ 
is the DPP such 
that the correlation kernel is given by
the Bergman kernel $K_{\widetilde{D}}$ of $\widetilde{D}$. 
This DPP is conformally invariant in the following sense.  
If $D \subsetneq \C$ is another 
simply connected domain with $\cC^{\infty}$ 
boundary, and $f: D \to \widetilde{D}$ is 
a conformal transformation, then 
$f^{*}\cZ_{X_{\widetilde{D}}}$ has the same distribution as 
$\cZ_{X_{D}}$.
In other words, $f^{*}\cZ_{X_{\widetilde{D}}}$
is a DPP such that the correlation kernel (\ref{eqn:transform_C}) 
with $K^{\det}_{\widetilde{D}}=K_{\widetilde{D}}$ 
is equal to the Bergman kernel $K_{D}$ of $D$.
\end{prop}

\SSC{Proofs}
\label{sec:proofs}
\subsection{Proof of Proposition \ref{thm:conformal_Aq}}
\label{sec:proof_conformal_Aq}
Use the expression (\ref{eqn:S_qt_JK}) 
of $S_{\A_q}(z, w; r)$ in Proposition \ref{thm:S_qt_JK}. 
Using (\ref{eqn:JK2b}) and (\ref{eqn:JK2c}), we can show that
\begin{align*}
&\sqrt{T_q'(z)} \overline{\sqrt{T_q'(w)}}
S_{\A_q}(T_q(z), T_q(w) ; r) 
= \sqrt{(-q)/z^2}
\overline{\sqrt{(-q)/w^2}}
f^{\rm JK} (q^2/z \wbar, -r) 
\nonumber\\
& \qquad = \frac{q}{z \wbar}
f^{\rm JK}(q^2/z \wbar, -r)
=-\frac{q}{z \wbar} f^{\rm JK}
(z \wbar/q^2, -1/r)
\nonumber\\
&\qquad = \frac{q}{r z \wbar} 
f^{\rm JK} (z \wbar, -1/r) 
= \frac{q}{r} f^{\rm JK} (z \wbar, -q^2/r) 
= \frac{q}{r} S_{\A_q} (z, w; q^2/r ). 
\end{align*}
In particular, when $r=q$, 
\[
\sqrt{T_q'(z)} \overline{\sqrt{T_q'(w)}} S_{\A_q}(T_q(z), T_q(w); q) 
=\sqrt{T_q'(z)} \overline{\sqrt{T_q'(w)}} S_{\A_q}(T_q(z), T_q(w)) 
=  S_{\A_q}(z, w), 
\]
which implies the invariance of the GAF $X_{\A_q}$ under conformal 
transformations preserving $\A_q$
by Schottkey's theorem \cite{AIMO08}. 

\subsection{Proof of Theorem \ref{thm:mainA1}}
\label{sec:proof_mainA1}

We recall a general formula for correlation functions
of zero point process of a GAF, which
is found in \cite{PV05}, but here we use 
a slightly different expression given 
by Proposition 6.1 of \cite{Shi12}. 
Let $\partial_z \partial_{\wbar} 
:= \frac{\partial^2}{\partial z \partial \wbar}$. 

\begin{prop}
\label{thm:correlation_function}
The correlation functions of $\cZ_{X_D}$ of the GAF $X_D$ on 
$D \subsetneq \C$ with
covariance kernel $S_D(z,w)$ are given by 
\[
 \rho_D^n(z_1,\dots, z_n) 
= \frac{ \per_{1 \leq i, j \leq n} \big[(\partial_z \partial_{\wbar}
S_D^{z_1,\dots,z_n})(z_i, z_j) \big]}
{\det_{1 \leq i, j \leq n} \Big[
 S_D(z_i, z_j) \Big]}, 
\quad n \in \N, \quad
z_1, \dots, z_n \in D, 
\]
with respect to a reference measure $\lambda$, 
whenever $\det_{1 \leq i, j \leq n}[S_D(z_i, z_j)]> 0$. 
Here the conditional kernels are defined by
(\ref{eqn:conditionS}) and (\ref{eqn:conditionS2}). 
\end{prop}
Here we abbreviate 
$\gamma^q_{\{z_{\ell}\}_{\ell=1}^n}$ given by (\ref{eqn:gamma})
to $\gamma^q_n$.
Then (\ref{eqn:MS_general}) gives
$S_{\A_q}^{z_1, \dots, z_n}(z, w; r)
=S_{\A_q}(
z, w; r \prod_{\ell=1}^n |z_{\ell}|^2)
\gamma^q_{n}(z)
\overline{ \gamma^q_{n}(w) }$
for $z, w, z_1, \dots, z_n \in \A_q$. 
By Lemma \ref{thm:Blaschke} (i), 
this formula gives
\[
(\partial_z \partial_{\wbar} 
S_{\A_q}^{z_1,\dots,z_n})(z_i, z_j; r) 
= S_{\A_q} \Big(z_i, z_j;  r \prod_{\ell=1}^n |z_{\ell}|^2 \Big) 
{\gamma^q_n}'(z_i)
\overline{{\gamma^q_n}'(z_j)}. 
\]
Therefore, 
Proposition \ref{thm:correlation_function} gives now
\begin{equation}
 \rho^n_{\A_q}(z_1, \dots, z_n; r) 
= \frac{\per_{1 \leq i, j \leq n}
 \left[S_{\A_q} \left(z_i, z_j;  r \prod_{\ell=1}^n |z_\ell|^2 \right) \right]
\prod_{k=1}^n |{\gamma^q_n}'(z_k)|^2 }
{\det_{1 \leq i, j \leq n} 
\Big[
S_{\A_q}(z_i, z_j; r) \Big]}. 
\label{eqn:rho_Z1}
\end{equation}
By (\ref{eqn:haq}) and Lemma \ref{thm:Blaschke} (i) and (iv), 
we see that 
\begin{align*}
\prod_{i=1}^n |{\gamma_n^q}'(z_i)|^2 
&= \prod_{i=1}^n \Big| \Big(
\prod_{1 \leq j \leq n, j \not=i}
h^q_{z_j}(z_i) \Big)
{h^q_{z_i}}'(z_i) \Big|^2
= \prod_{i=1}^n 
\Big| \Big( \prod_{1 \leq j \leq n,  j \not=i}
z_i \frac{\theta(z_j/z_i)}{\theta(\overline{z_j} z_i)} \Big)
\frac{ q_0^2 }{ \theta(|z_i|^2)} \Big|^2 
\nonumber\\
&= \Bigg|
\frac{ q_0^{2n} \prod_{1 \le i < j \le n} z_i \theta(z_j/z_i)
\cdot \prod_{1 \le i' < j' \le n} z_{j'} \theta(z_{i'}/z_{j'}) }
{\prod_{i=1}^n \prod_{j=1}^n \theta(z_i \overline{z_j}) } \Bigg|^2. 
\end{align*}
By (\ref{eqn:theta_inversion}),
$z_i \theta(z_j/z_i)
=z_i ( - z_j/z_i ) \theta(z_i/z_j)
=-z_j \theta(z_i/z_j)$.
Hence this is written as 
\begin{align}
\prod_{i=1}^n |{\gamma_n^q}'(z_{i})|^2 
&= q_0^{4n} \left|
\frac{(-1)^{n(n-1)/2} \big(\prod_{1 \leq i  < j \leq n} z_j \theta(z_i/z_j) \big)^2}
{\prod_{i=1}^n \prod_{j=1}^n \theta(z_i \overline{z_j})} \right|^2
\nonumber\\
&= q_0^{4n}
\Bigg(
\frac{\prod_{1 \leq i < j \leq n} |z_j|^2 
\theta(z_i/z_j, \overline{z_i}/\overline{z_j})}
{\prod_{i=1}^n \prod_{j=1}^n \theta(z_i \overline{z_j})}
\Bigg)^2.
\label{eqn:gamma2}
\end{align}

The following identity is known as an 
elliptic extension of Cauchy's evaluation of determinant 
due to Frobenius (see 
Theorem 1.1 in \cite{KN03}, 
Theorem 66 in \cite{Kra05}, 
Corollary 4.7 in \cite{RS06}, 
and references therein),
\begin{align*}
\det_{1 \leq i, j \leq n}
\left[ \frac{\theta(t x_i a_j)}{\theta(t, x_i a_j)} \right]
&=\frac{\theta(t \prod_{k=1}^n x_k a_k)}{\theta(t)}
\frac{\prod_{1 \leq i < j \leq n}
x_j a_j
\theta(x_i/x_j, a_i/a_j)}
{\prod_{i=1}^n \prod_{j=1}^n \theta(x_i a_j)}.
\end{align*}
By (\ref{eqn:S_qt_theta}) in Proposition \ref{thm:S_Aq_theta}, 
we have
\begin{equation}
q_0^{2n}
\frac{\prod_{1 \leq i < j \leq n} |z_j|^2
\theta(z_i/z_j, \overline{z_i}/\overline{z_j})}
{\prod_{i=1}^n \prod_{j=1}^n \theta(z_i \overline{z_j})}
=\frac{\theta(-s)}{\theta( -s \prod_{\ell=1}^n |z_{\ell}|^2)}
\det_{1 \leq i, j \leq n}
\left[S_{\A_q}(z_i, z_j; s) \right], 
\quad \forall s >0.
\label{eqn:Frobenius_formula}
\end{equation}
Then (\ref{eqn:gamma2}) is written as
\begin{align*}
\prod_{i=1}^n |{\gamma_n^q}'(z_i)|^2 
&= \frac{\theta(-r)}{\theta(-r \prod_{\ell=1}^n |z_{\ell}|^2)}
\det_{1 \leq i, j \leq n} [S_{\A_q}(z_i, z_j; r) ]
\nonumber\\
& \quad \times
\frac{\theta(-r \prod_{\ell=1}^n |z_{\ell}|^2)}
{\theta(-r \prod_{\ell=1}^n |z_{\ell}|^4)}
\det_{1 \leq i, j \leq n} 
\Big[S_{\A_q} \Big(z_{i}, z_{j}; 
r \prod_{\ell=1}^n |z_{\ell}|^2 \Big) \Big]
\nonumber\\
&= \frac{\theta(-r)}{\theta(-r\prod_{\ell=1}^n |z_{\ell}|^4)}
\det_{1 \leq i, j \leq n} [S_{\A_q}(z_i, z_j; r) ]
\det_{1 \leq i, j \leq n} 
\Big[S_{\A_q} \Big(z_{i}, z_{j}; r \prod_{\ell=1}^n |z_{\ell}|^2 
\Big) \Big]. 
\end{align*}
Applying the above to (\ref{eqn:rho_Z1}),
the correlation functions
in Theorem \ref{thm:mainA1} are obtained.

\subsection{Direct proof of
the $(q, r)$-inversion symmetry of correlation functions}
\label{sec:proof_inversion_symmetry}

The following is a corollary of Proposition
\ref{thm:conformal_Aq} (ii) and (iii). 
Here we give a direct proof using the explicit formulas
for correlation functions given in Theorem \ref{thm:mainA1}.

\begin{cor}
\label{thm:rho_inversion}
For every $n \in \N$ and
$z_1, \dots, z_n \in \A_q$, 
\begin{equation}
\rho^{n}_{\A_q}(T_q(z_1),\dots, T_q(z_n); r) 
\prod_{\ell=1}^n |T_q'(z_{\ell})|^2
= \rho^{n}_{\A_q}(z_1, \dots, z_n; q^2/r).
\label{eqn:rho_inversion1}
\end{equation}
In particular,
$\rho^{n}_{\A_q}(T_q(z_1),\dots, T_q(z_n); q) 
\prod_{\ell=1}^n |T_q'(z_{\ell})|^2
= \rho^{n}_{\A_q}(z_1, \dots, z_n; q)$, 
for $n \in \N$ and
$z_1, \dots, z_n \in \A_q$.
\end{cor}
\begin{proof}
We calculate
$\rho_{\A_q}^n(T_q(z_1), \dots, T_q(z_n); r)$
for $\rho_{\A_q}^n$ given by
(\ref{eqn:rho_mainA1}) 
in Theorem \ref{thm:mainA1}.
By (\ref{eqn:S_qt_JK}) in Proposition \ref{thm:S_qt_JK}, 
\begin{align*}
&S_{\A_q} \Big(T_q(z), T_q(w);  
r \prod_{\ell=1}^n |T_q(z_{\ell})|^2 \Big)
= S_{\A_q} \Big( q z^{-1}, q \wbar^{-1};  
q^{2n}r \prod_{\ell=1}^n |z_{\ell}|^{-2} \Big)
\nonumber\\
&\quad =f^{\rm JK} \Big(q^2 (z \wbar)^{-1}, 
- q^{2n}r \prod_{\ell=1}^n |z_{\ell}|^{-2}  \Big)
= - f^{\rm JK}
\Big(q^{-2} z \wbar, -q^{-2n} r^{-1} \prod_{\ell=1}^n|z_{\ell}|^2 \Big),
\end{align*}
where we used (\ref{eqn:JK2b}) at the last equation.
If we apply the 
equality between the leftmost side and the rightmost side
in (\ref{eqn:JK2c}), we see that
the above is equal to 
$q^{-2n}r^{-1} \prod_{\ell=1}^n |z_{\ell}|^2$
$f^{\rm JK} ( z \wbar, -q^{-2n} r^{-1} \prod_{\ell=1}^n |z_{\ell}|^2)$.
Then we apply the first equality in (\ref{eqn:JK2c}) 
$n+1$ times and obtain 
\begin{align}
S_{\A_q} \Big(T_q(z), T_q(w);  r \prod_{\ell=1}^n |T_q(z_{\ell})|^2 \Big)
&=q^{-2n} r^{-1} \prod_{\ell=1}^n |z_{\ell}|^2
(z \wbar)^{n+1} 
f^{\rm JK} \Big(z \wbar, - q^2 r^{-1} \prod_{\ell=1}^n
|z_{\ell}|^2 \Big)
\nonumber\\
&=q^{-2n} r^{-1} \prod_{\ell=1}^n |z_{\ell}|^2
(z \wbar)^{n+1}
S_{\A_q} \Big( z, w; q^2 r^{-1} \prod_{\ell=1}^n
|z_{\ell}|^2 \Big).
\label{eqn:factor3}
\end{align}
Here we note that by definition (\ref{eqn:perdet}) of $\perdet$,
the multilinearity of permanent and determinant
implies the equality
\[
\perdet_{1 \leq i, j \leq n} 
\Big[a b_i c_j m_{ij} \Big]
= a^{2n} \prod_{k=1}^n b_k^2 c_k^2 \cdot
\perdet_{1 \leq i, j \leq n} [m_{ij}].
\]
Then by (\ref{eqn:factor3}), we have
\begin{align}
&\perdet_{1 \leq i, j \leq n} \Big[
S_{\A_q} \Big(T_q(z_i), T_q(z_j);  
r \prod_{\ell=1}^n |T_q(z_{\ell})|^2 \Big) \Big]
\nonumber\\
& \quad
= \perdet_{1 \leq i, j \leq n}
\Big[
q^{-2n} r^{-1} \prod_{\ell=1}^n |z_{\ell}|^2
(z_i \zbar_j)^{n+1}
S_{\A_q} \Big( z_i, z_j; q^2 r^{-1} \prod_{\ell=1}^n
|z_{\ell}|^2 \Big) \Big]
\nonumber\\
& \quad
= q^{-4n^2} r^{-2n} 
\prod_{\ell=1}^n |z_{\ell}|^{4(2n+1)}
\perdet_{1 \leq i, j \leq n}
\Big[
S_{\A_q} \Big( z_i, z_j; q^2 r^{-1} \prod_{\ell=1}^n
|z_{\ell}|^2 \Big) \Big]. 
\label{eqn:factor4}
\end{align}

Now we consider the prefactor 
of $\perdet$ in
(\ref{eqn:rho_mainA1}).
By (\ref{eqn:theta_period}), 
$\theta(-r)=\theta(-q^2/r)$.
On the other hand,
\[
\theta \Big(-r \prod_{\ell=1}^n |T_q(z_{\ell})|^4 \Big)
= \theta \Big(-r q^{4n} \prod_{\ell=1}^n |z_{\ell}|^{-4} \Big)
= \theta \Big(
- q^{-2(2n-1)} r^{-1} \prod_{\ell=1}^n |z_{\ell}|^4 \Big). 
\]
If we apply (\ref{eqn:theta_qp}) once,
then we find that the above is equal to
$q^{-2(2n-1)} r^{-1}$$\prod_{\ell=1}^n |z_{\ell}|^4$
$\theta (
- q^{-2(2n-2)} r^{-1} \prod_{\ell=1}^n |z_{\ell}|^4 )$.
We apply (\ref{eqn:theta_qp}) $2n-1$ more times.
Then the above turns to be equal to
$q^{-2 \sum_{i=1}^{2n-1} i}$$r^{-2n}\prod_{\ell=1}^n |z_{\ell}|^{8n}$
$\theta (- q^2 r^{-1} \prod_{\ell=1}^n |z_{\ell}|^4)$. 
Then we have the equality
\begin{equation}
\frac{\theta(-r)}
{\theta \Big(-r \prod_{\ell=1}^n |T_q(z_{\ell})|^4 \Big)}
= q^{2n(2n-1)}  r^{2n}
\prod_{\ell=1}^n |z_{\ell}|^{-8n}
\frac{\theta(-q^2/r)}
{\theta \Big(
- q^2 r^{-1} \prod_{\ell=1}^n |z_{\ell}|^4 \Big)}.
\label{eqn:factor2}
\end{equation}

Combining the results (\ref{eqn:factor4})
and (\ref{eqn:factor2}), we have
\[
\rho^{n}_{\A_q}(T_q(z_1),\dots, T_q(z_n); r) 
= \rho^{n}_{\A_q}(z_1,\dots, z_n; q^2/r)
q^{-2n} \prod_{\ell=1}^n |z_{\ell}|^4.
\]
Since $|T_q'(z)|^2=q^2/|z|^4$,
(\ref{eqn:rho_inversion1}) is proved.
\end{proof}

\subsection{Proofs of Proposition
\ref{thm:short_distance} and 
Theorem \ref{thm:critical_line}}
\label{sec:proof_critical_line}
\subsubsection{Upper and lower bounds of 
unfolded 2-correlation function}
\label{sec:upper_lower}

By (\ref{eqn:density_Aqt}) and (\ref{eqn:rho2_Aqt}),
the unfolded 2-correlation function (\ref{eqn:unfolded})
is explicitly written as follows,
\begin{align}
g_{\A_q}(z,w; r) 
&= \frac{\theta(-r|z|^2, -r|w|^2, -r|z|^4 |w|^2, -r|z|^2|w|^4)^2}
{\theta(-r, -r|z|^4, -r|w|^4, -r|z|^4|w|^4) \theta(-r|z|^2 |w|^2)^4}
\nonumber\\
& \quad \times 
\left[ 1-
\left\{
\frac{ \theta(|z|^2, |w|^2)}
{\theta(-r|z|^4 |w|^2, -r|z|^2 |w|^4)}
\right\}^2
\frac{
| \theta(-r z \wbar |z|^2 |w|^2)|^4 
}{
|\theta(z \wbar)|^4}
\right],
\label{eqn:unfolded2}
\end{align}
with $z, w \in \A_q$. 
Using (\ref{eqn:theta_zero}) and (\ref{eqn:theta_qp}),
it is readily verified that
\[
g_{\A_q}(1,z;r)
=g_{\A_q}(z,1;r)
=g_{\A_q}(q,z;r)
=g_{\A_q}(z,q;r)=1,
\quad z \in \A_q.
\]
\begin{lem}
\label{thm:bounds}
If we set $a=|z|, b=|w|$, $a, b \in (q, 1)$, then
\[
g_{\A_q}(a, b; r)
\leq
g_{\A_q}(z, w; r) 
\leq g_{\A_q}(-a, b; r),
\quad z, w \in \A_q,
\]
where
\begin{align}
g_{\A_q}(\pm a, b; r) 
&= \frac{b^2 \theta( \pm a/b, -ra^2, -rb^2)^2}
{\theta(-r, -ra^4, -rb^4) \theta(\pm ab)^4 \theta(-r a^2 b^2)^3}
\nonumber\\
& \quad \times 
\left[ 
\theta(-r a^4 b^2, -r a^2 b^4)
\theta(\pm ab)^2
+
\theta(a^2, b^2)
\theta(\mp r a^3 b^3)^2
\right].
\label{eqn:upper_lower1}
\end{align}
\end{lem}

First we show the following lemma.
\begin{lem}\label{lem:theta_ratio} 
Let $\alpha, \beta >0$ with $\alpha \not\in \{q^{2i}: i \in \Z\}$. 
Then the function 
$|\theta(-\beta e^{i\varphi}) / \theta(\alpha e^{i\varphi})|^2$ 
on $\varphi \in [0,2\pi)$ 
attains its maximum at $\varphi=0$ and 
its minimum at $\varphi=\pi$. 
\end{lem}
\begin{proof} Set $f(x; \alpha, \beta) 
= (1+2\beta x + x^2)/(1-2\alpha x + x^2)$ for $x \in [-1,1]$. Then, 
\begin{align*}
\left|
\frac{\theta(-\beta e^{i\varphi})}{\theta(\alpha e^{i\varphi})}
\right|^2
&=\prod_{n=0}^{\infty}
f(\cos \varphi; \alpha q^{2n}, \beta q^{2n})
\prod_{m=0}^{\infty}
f(\cos \varphi; \alpha^{-1} q^{2(m+1)}, \beta^{-1} q^{2(m+1)})
\end{align*}
Since $\partial f(x;\alpha, \beta)/\partial x
=2(1+\alpha \beta) (\alpha+\beta)
/(1-2 \alpha x + \alpha^2)^2
\geq 0$, $f$ attains its maximum (resp. minimum) at 
$x=1$ (resp. $x=-1$). Hence the assertion follows. 
\end{proof}
Now we proceed to the proof of Lemma~\ref{thm:bounds}. 
\begin{proof}
We set $z=a e^{\sqrt{-1} \varphi_z}$, $w=b e^{\sqrt{-1} \varphi_w}$,
$a,b \in (q, 1)$, $\varphi_z, \varphi_w \in [0, 2 \pi)$.
Then we can see that (\ref{eqn:unfolded2}) depends on 
the angles $\varphi_z, \varphi_w$ only through the factor
$|\theta(-r z \wbar |z|^2 |w|^2) /
\theta(z \wbar)|^4$, and we have 
$|\theta(-r z \wbar |z|^2 |w|^2)/
\theta(z \wbar)|^2 
= |\theta(-r a^3 b^3 e^{\sqrt{-1}(\varphi_z-\varphi_w)})
/\theta(ab e^{\sqrt{-1}(\varphi_z-\varphi_w)})|^2$.
We can conclude that
$\theta(-r a^3 b^3)^2/\theta(ab)^2 \geq
\left| \theta(-r z \wbar |z|^2 |w|^2)/\theta(z \wbar)
\right|^2 \geq
\theta(r a^3 b^3)^2/
\theta(-ab)^2$
from Lemma~\ref{lem:theta_ratio}, 
and the inequalities are proved.
If we use Weierstrass' addition formula 
(\ref{eqn:Weierstrass_add1}) 
by setting
$x=r^{1/2} a^{5/2} b^{3/2}$, 
$y=-r^{1/2} a^{3/2} b^{1/2}$,
$u=-r^{1/2} a^{3/2} b^{5/2}$, and
$v=r^{1/2} a^{1/2} b^{3/2}$,
then we obtain 
$\theta(-r a^4 b^2, -r a^2 b^4)
\theta(-ab)^2-\theta(a^2, b^2)\theta(r a^3 b^3)^2
=b^2 \theta(-a/b)^2 \theta(-r a^2 b^2, -r a^4 b^4)$. 
Using this equality and the one obtained
by replacing $a$ by $-a$, it is easy to obtain 
(\ref{eqn:upper_lower1}).
The proof is complete. 
\end{proof}

\subsubsection{Proof of Proposition \ref{thm:short_distance}}
\label{sec:short_distance}

By the definition (\ref{eqn:g1_short}),
if we use (\ref{eqn:theta_inversion})--(\ref{eqn:theta_period}),
we can derive the following from (\ref{eqn:upper_lower1}),
\[
G_{\A_q}^{\wedge}(x; r) =
\frac{r^2 \theta(qx^2)^2 \theta(-r x^2, -r^{-1} x^2)^3}
{x^2 \theta(q)^2 \theta(-r)^4 \theta(-r x^4, -r^{-1} x^4)}
\left[ 1 +
\frac{\theta(-r q, x^2)^2}
{\theta(q)^2 \theta(-r x^2, -r^{-1} x^2)}
\right],
\quad x \in (\sqrt{q}, 1).
\]
Since 
$x^2 \sim q\{1+2q^{-1/2}(x-\sqrt{q})\}$
when $x \sim \sqrt{q}$, 
$\theta(qx^2)
\sim \theta(q^2\{1+2 q^{-1/2}(x-\sqrt{q})\})
\sim - \theta (1+2 q^{-1/2}(x-\sqrt{q}))$,
where (\ref{eqn:theta_qp}) was used.
Then
\[
\theta(qx^2)
\sim - \theta'(1) \cdot 2 q^{-1/2}(x-\sqrt{q})
= 2 q_0^2 q^{-1/2}(x-\sqrt{q})
\quad
\mbox{as $x \to \sqrt{q}$}.
\]
where (\ref{eqn:theta_prime}) was used.
Hence
$\theta(qx^2)^2 
\sim (4 q_0^4/q) (x-\sqrt{q})^2$
as $x \to \sqrt{q}$,
and $G_{\A_q}^{\wedge}(x; r) \asymp (x-\sqrt{q})^2$
as $x \to \sqrt{q}$.
Using (\ref{eqn:theta_inversion})--(\ref{eqn:theta_period}),
we can show that
$\theta(-r^{-1}q)=\theta(-r q)$,
$\theta(-r^{-1}q^2)=\theta(-r)$,
$\theta(-rq^2)=r^{-1} \theta(-r)$.
Then the coefficient is determined as 
given by $c(r)$.

\subsubsection{Proof of Theorem \ref{thm:critical_line} (i)}
\label{sec:i}

Replacing $x$ by $\sqrt{c}$ in (\ref{eqn:g1_long}),
here we consider 
$\widetilde{G}(c)=
\widetilde{G}(c; r, q):= G_{\A_q}^{\vee}(\sqrt{c}; r)$, 
$c \in (q^2, 1)$.
From (\ref{eqn:upper_lower1}) in Lemma \ref{thm:bounds},
we have 
\begin{equation}
\widetilde{G}(c)
= 
\frac{c \theta(-rc)^4 \theta(-1, -rc^3)^2}
{\theta(-r) \theta(-c)^2 \theta(-rc^2)^5}
\left[
1 + \frac{\theta(c, rc^3)^2}{\theta(-c, -rc^3)^2}
\right].
\label{eqn:gc1}
\end{equation}
It is easy to see that 
$\widetilde{G}(1)=1$.
Here we will prove the following.
\begin{prop}
\label{thm:differentials}
\begin{align}
\widetilde{G}'(1) &=\widetilde{G}''(1)=\widetilde{G}'''(1)=0,
\label{eqn:1_3rd0}
\\
\widetilde{G}^{(4)}(1)
&=\widetilde{G}^{(4)}(1; r, q)
=12(10 \wp(\phi_{-r})^2+4 e_1 \wp(\phi_{-r})-2 e_1^2-g_2).
\label{eqn:4th0}
\end{align}
\end{prop}
This proposition implies 
(\ref{eqn:Gvee_up}) 
with (\ref{eqn:kappa}), since
$\kappa(r) =\widetilde{G}^{(4)}(1)/4!$.
By (\ref{eqn:rho_symmetry}), we have the equality
$G_{\A_q}^{\vee}(x; r)
=G_{\A_q}^{\vee} (q/x; q^2/r)$.
Since $x \to q$ is equivalent with
$q/x \to 1$, (\ref{eqn:Gvee_up}) implies
\[
G_{\A_q}(q/x; q^2/r)
\sim 1+\kappa(q^2/r) (1-(q/x)^2)^4
\quad \mbox{as $x \downarrow q$}.
\]
Therefore, once Proposition \ref{thm:differentials} 
is proved and hence (\ref{eqn:kappa}) is verified, then
the equalities (\ref{eqn:kappa_symmetry})
are immediately concluded from (\ref{eqn:wp_symmetry}).
With the third equality in (\ref{eqn:kappa_symmetry}),
the above proves (\ref{eqn:Gvee_down}).
If $r>1$, then $1/r<1$, 
on the other hand if $0 < r < q$, given $q \in (0,1)$,
then $q^2/r >q$. Hence by the first and the third equalities
in (\ref{eqn:kappa_symmetry}) the values of $\kappa(r)$ 
in the parameter space outside of $\Omega$ can be determined
by those in $\Omega$. 
By the three equalities in (\ref{eqn:kappa_symmetry}),
the structure in $\Omega$ described by 
Proposition \ref{thm:critical_line} (ii) and (iii) is
repeatedly mapped into the parameter space
outside of $\Omega$.

Now we proceed to the proof
of Proposition \ref{thm:differentials}.
First we decompose $\widetilde{G}(c)$ given by (\ref{eqn:gc1})
as 
\[
\widetilde{G}(c) = I(c)+J(c)
= I(c)+\beta_r^2 (c-1)^2 I(c) K(c),
\]
with
\begin{align}
I(c) &=
\frac{c \theta(-rc)^4 \theta(-1, -rc^3)^2}
{\theta(-r) \theta(-c)^2 \theta(-rc^2)^5},
\qquad
\beta_r = 
\frac{\theta'(1) \theta(r)}{\theta(-1, -r)},
\nonumber\\
K(c) &= \Big( \frac{\theta(c)}{(c-1) \theta'(1)} \Big)^2
\frac{\theta(-1, -r, rc^3)^2}{\theta(-c, -rc^3, r)^2},
\label{eqn:I_beta_K}
\end{align}
The following is easily verified, 
where $\cD_z$ denotes the Euler operator (\ref{eqn:Euler}). 
\begin{lem}
\label{thm:euler}
Suppose that $f$ is a $\cC^{\infty}$-function
and $f(1)=1$, then
\begin{align*}
\cD_z \log f(z)|_{z=1} &=f'(1),
\nonumber\\
\cD_z^2 \log f(z)|_{z=1}
&=f''(1)+f'(1)-f'(1)^2.
\end{align*}
If, in addition, $f'(1)=0$, then
\begin{align*}
\cD_z^2 \log f(z)|_{z=1} &= f''(1), \\
\cD_z^3 \log f(z)|_{z=1} &= f'''(1) + 3 f''(1), \\
\cD_z^4 \log f(z)|_{z=1} &= f^{(4)}(1) + 6 f'''(1) + 7 f''(1) 
- 3 f''(1)^2. 
\end{align*}
\end{lem}

Recall that $a_n(z), n \in \N$ are defined by
(\ref{eqn:an_x}) in Section \ref{sec:Weierstrass_elliptic}.
\begin{prop}
\label{thm:Euler_calculation}
\upshape{(i)} $\cD_z^n \big(\log \theta(\alpha z^k) \big) 
= k^n a_n(\alpha z^k)$, 
\ \upshape{(ii)} $\cD_z \big(a_n(\alpha z^k) \big) 
= k a_{n+1}(\alpha z^k)$. 
\end{prop}
\noindent
This proposition is a corollary of the following lemma. 
\begin{lem}
\label{thm:der=substitute} 
Suppose that $f$ is a $\cC^{\infty}$-function. 
Let $F_n(w) := \cD_{w}^n \log f(w), n \in \N$.
Then for $k, n \in \N$ and a constant $\alpha$, 
$\cD_z^n \big( \log f(\alpha z^k) \big)= k^n F_n(\alpha z^k)$. 
\end{lem}
\begin{proof} 
It suffices to show 
the equality
$\cD_z \big( F_{n}(\alpha z^k) \big) = k F_{n+1}(\alpha z^k)$.
Indeed, 
\[
 \cD_z \big(F_{n}(\alpha z^k) \big)
= z \cdot \frac{d}{dw} F_n(w)\Big|_{w=\alpha z^k} \cdot \alpha kz^{k-1} 
= k \cdot \left(w \frac{d}{dw} F_n(w)\right)\Big|_{w=\alpha z^k} 
= k F_{n+1}(\alpha z^k). 
\]
Then the proof is complete.
\end{proof}

\begin{lem}
\label{thm:betasquare} \,
$\beta_r^2=a_2(-1)-a_2(-r)$.
\end{lem}
\begin{proof}
First we note that
by (\ref{eqn:S_qt_JK}) in Proposition \ref{thm:S_qt_JK},
(\ref{eqn:S_qt_theta}) in Proposition \ref{thm:S_Aq_theta}
and (\ref{eqn:theta_prime}),
$\beta_r=-S_{\A_q}(-1,1;r)=-f^{\rm JK}(-1, -r)$.
Then by (\ref{eqn:fundamental4}) in Lemma
\ref{thm:fundamental_eqs} in Section 
\ref{sec:Weierstrass_elliptic}
\begin{equation}
\beta_r^2=e_1-\wp(\phi_{-r})= \wp(\pi) - \wp(\phi_{-r}) 
= \wp(\phi_{-1}) - \wp(\phi_{-r}).  
\label{eqn:beta_t_2_wp}
\end{equation}
Hence the formula for $a_2(z)$ in Lemma \ref{thm:an}
in Section \ref{sec:Weierstrass_elliptic} 
proves the statement.
\end{proof}
By the definition (\ref{eqn:I_beta_K}), 
it is easy to see that 
\begin{equation}
I(1)=K(1)=1, \quad J(1) = J'(1) = 0, \quad J''(1) = 2\beta_r^2.  
\label{eqn:IJK} 
\end{equation}

In what follows, we will use Proposition \ref{thm:Euler_calculation}
(i) repeatedly. 
Using Lemma~\ref{thm:euler} with $I(1)=1$ 
and Proposition \ref{thm:Euler_calculation} (i),
\begin{align*}
I'(1) 
&= \cD_c \log I(c) \Big|_{c=1} \\
&= 1 + 4 a_1(-rc) +  2 \cdot 3 a_1(-rc^3) 
- 2 a_1(-c) - 5 \cdot 2 a_1(-rc^2) \Big|_{c=1} \\
&= 1 - 2 a_1(-1)=0,
\end{align*}
where Lemma \ref{thm:gammas} (iii)
in Section \ref{sec:Weierstrass_elliptic} was used
at the last equality.
Therefore, $\widetilde{G}'(1) = I'(1) + J'(1) = 0$. 

From now on, we use the notation 
$A_n(r) := a_n(r) - a_n(-r), n \in \N$.
Using Lemma~\ref{thm:euler} with $K(1)=1$, 
Proposition \ref{thm:Euler_calculation} (i), and
Lemma \ref{thm:gammas} (i), (iii) 
in Section \ref{sec:Weierstrass_elliptic}, 
we see that 
\begin{align}
K'(1) 
&= \cD_c \log K(c) \Big|_{c=1} 
\nonumber\\
 &= 2 \{ a_1(c) + c/(1-c) \}
+ 2 \cdot 3 a_1(rc^3) -2 a_1(-c) - 2 \cdot 3 a_1(-rc^3) \Big|_{c=1} 
\nonumber\\
&=6a_1(r) - 1 - 6a_1(-r)
= -1 + 6A_1(r).
\label{eqn:K'1}
\end{align}
Using Lemma~\ref{thm:euler} with $I(1)=1, I'(1)=0$ 
and Proposition \ref{thm:Euler_calculation} (i),
\begin{align}
I''(1) 
&= \cD_c^2 \log I(c) \Big|_{c=1} 
\nonumber\\
&= 4 a_2(-rc) +  2 \cdot 3^2 a_2(-rc^3) 
- 2 a_2(-c) - 5 \cdot 2^2 a_2(-rc^2) \Big|_{c=1} 
\nonumber\\
&= 2 (a_2(-r) - a_2(-1)) =-2\beta_r^2, 
\label{eqn:I''1}
\end{align}
where we used Lemma~\ref{thm:betasquare}
at the last equality. 
Therefore, by (\ref{eqn:IJK}), we obtain
$\widetilde{G}''(1) = I''(1) + J''(1) = 0$. 

\begin{lem}
\label{thm:dbeta} 
\quad
$\cD_r \beta_r = \beta_r A_1(r)$. 
Moreover, 
$\lim_{r \to 1} \beta_r A_1(r) 
= (\theta'(1)/\theta(-1))^2$. 
\end{lem}
\begin{proof} We observe that 
$\cD_r \log \beta_r 
= \cD_r \log \theta(r)  - \cD_r \log \theta(-r) 
= a_1(r) -a_1(-r) = A_1(r)$. 
On the other hand, $\cD_r \log \beta_r = \cD_r \beta_r/\beta_r$. 
Hence we obtain the first assertion. Note that 
\begin{equation}
\lim_{r \to 1} \frac{\beta_r}{r-1} 
= \lim_{r \to 1} 
\frac{\theta'(1)}{\theta(-1)\theta(-r)} \frac{\theta(r)}{r-1}
= \Big(\frac{\theta'(1)}{\theta(-1)} \Big)^2. 
\label{eqn:betaprime2}
\end{equation}
From Lemma \ref{thm:gammas} (i) and (iii) 
in Section \ref{sec:Weierstrass_elliptic}, 
we see that $(r-1) A_1(r) = 1 +{\rm O}(r-1)$ as $r \to 1$
and the second assertion is also proved.
\end{proof}
\begin{lem}
\label{thm:a3minust}
\,
$a_3(-r) = -2\beta_r^2 A_1(r)$. 
In particular, $a_3(-1) = 0$. 
\end{lem}
\begin{proof}
We apply $\cD_r$ to both sides of the identity
of Lemma \ref{thm:betasquare}.
From Lemma~\ref{thm:dbeta}, we have
the left-hand side
$\cD_r \beta_r^2 = 2\beta_r \cdot \cD_r \beta_r
=2\beta_r^2 A_1(r)$,
which is equal to the right-hand side
$-\cD_r a_2(-r) = - a_3(-r)$. 
The second assertion is obtained 
using the second assertion of Lemma \ref{thm:dbeta}
and the fact that $\beta_1=0$. 
\end{proof}
Using Lemma~\ref{thm:euler} with $I(1)=1, I'(1)=0$, 
Proposition \ref{thm:Euler_calculation} (i),
and Lemma~\ref{thm:a3minust},
we see that
\begin{align*}
I'''(1) + 3 I''(1)
&= \cD_c^3 \log I(c) \Big|_{c=1} \\
&= 4 a_3(-rc) + 2 \cdot 3^3 a_3(-rc^3) 
- 2 a_3(-c) - 5 \cdot 2^3 a_3(-rc^2) \Big|_{c=1} 
= 18 a_3(-r). 
\end{align*}
With (\ref{eqn:I''1}) we have 
$I'''(1) = 18 a_3(-r) + 6\beta_r^2$. 
By the Leibnitz rule, we see that 
\begin{align*}
J'''(1) 
&= 3 \frac{d}{dc} (\beta_r^2 I(c) K(c))\Big|_{c=1} \cdot 2 
= 6\beta_r^2 (I'(1) K(1) + I(1) K'(1)) = 6\beta_r^2 K'(1) \\
&= - 6\beta_r^2 + 36 \beta_r^2 A_1(r)= - 6\beta_r^2 -18 a_3(-r). 
\end{align*}
Here we used the fact $I'(1)=0$, 
(\ref{eqn:K'1}) and Lemma~\ref{thm:a3minust}. 
Therefore, we have  
$\widetilde{G}'''(1) = I'''(1) + J'''(1) = 0$. 
The proof of (\ref{eqn:1_3rd0}) is complete now.

Then we begin to prove (\ref{eqn:4th0}).
Using Lemma~\ref{thm:euler} with $I(1)=1, I'(1)=0$, 
Proposition \ref{thm:Euler_calculation} (i), 
and Lemma~\ref{thm:a3minust}, 
\begin{align*}
&I^{(4)}(1) + 6 I'''(1) + 7 I''(1) - 3 I''(1)^2
= \cD_c^4 \log I(c) \Big|_{c=1} \\
&\quad = 4 a_4(-rc) + 2 \cdot 3^4 a_4(-rc^3) 
- 2 a_4(-c) - 5 \cdot 2^4 a_4(-rc^2) \Big|_{c=1} \\
&\quad = 86 a_4(-r) - 2 a_4(-1). 
\end{align*}
Therefore, 
$ I^{(4)}(1) = 86a_4(-r) - 2a_4(-1) -108a_3(-r) -22\beta_r^2
+ 12\beta_r^4$. 

\begin{lem}\label{thm:a4minust}
\,
$a_4(-r) = -2 \beta_r^2 (2 A_1(r)^2 + A_2(r) )$.
In particular, $a_4(-1) = -2 (\theta'(1)/\theta(-1) )^4$. 
\end{lem}
\begin{proof}
Applying $\cD_r$ to both sides of 
the first assertion of Lemma \ref{thm:a3minust} 
together with Proposition \ref{thm:Euler_calculation} (ii) 
yields the first assertion. 
The second assertion follows from (\ref{eqn:betaprime2}) and 
the facts that
$(r-1) A_1(r)=1 +{\rm O}(r-1)$ and
$(r-1)^2A_2(r) = -1 + {\rm O}(r-1)^2$ as $r \to 1$,
which are verified by 
Lemma \ref{thm:gammas} (i)--(iv) 
in Section \ref{sec:Weierstrass_elliptic}.
\end{proof}
By the Leibnitz rule, 
\begin{align*}
 J^{(4)}(1)
&= \beta_r^2 \Big(\frac{4!}{2! 2! 0!} I''(1) K(1) 
+\frac{4!}{2! 1! 1!} I'(1) K'(1) 
+\frac{4!}{2! 0! 2!} I(1) K''(1) 
\Big) \cdot 2 \\
&= 2\beta_r^2 \left(-12 \beta_r^2 
+6 K''(1) \right),
\end{align*}
where we used the fact $I'(1)=0$, (\ref{eqn:K'1})
and (\ref{eqn:I''1}). 
From (\ref{eqn:K'1}), we have
\begin{align*}
\cD_c^2 \log K(c) \Big|_{c=1}
 &= 2 \big\{ a_2(c) + c/(c-1)^2 \big\}
+ 2 \cdot 3^2 a_2(rc^3) -2 a_2(-c) - 2 \cdot 3^2a_2(-rc^3) \Big|_{c=1} \\
&= 2 (\gamma_2 - a_2(-1)) + 18A_2(r),  
\end{align*}
where Lemma~\ref{thm:gammas} (ii)
in Section \ref{sec:Weierstrass_elliptic} was used.
Using Lemma~\ref{thm:euler} with $K(1)=1$
and nonzero $K'(1)$ given by (\ref{eqn:K'1}), we obtain
\begin{align*}
K''(1) 
&= K'(1)^2 - K'(1) + \cD_c^2 \log K(c) \Big|_{c=1} \\
&= (-1 + 6A_1(r)) (-2 + 6A_1(r) )
+ 2 (\gamma_2 -a_2(-1)) + 18A_2(r) \\ 
&= 
2 -18 A_1(r) + 36 A_1(r)^2
+ 2 (\gamma_2 -a_2(-1)) + 18A_2(r). 
\end{align*}
Hence we have
\begin{align*}
J^{(4)}(1) &= -24 \beta_r^4 + 24 \beta_r^2
+108 \cdot 2 \beta_r^2(2 A_1(r)^2+A_2(r))
-108 \cdot 2 \beta_r^2 A_1(r)
+24 \beta_r^2(\gamma_2-a_2(-1))
\nonumber\\
&= -24 \beta_r^4+24 \beta_r^2 -108 a_4(-r)
+108 a_3(-r) + 24 \beta_r^2(\gamma_2-a_2(-1)),
\end{align*}
where Lemmas \ref{thm:a3minust} and 
\ref{thm:a4minust} were used.
Therefore, 
\begin{align*}
\widetilde{G}^{(4)}(1) 
&= I^{(4)}(1) + J^{(4)}(1) 
\nonumber\\
&= -22 a_4(-r)-12 \beta_r^4 +
24 \beta_r^2 (\gamma_2-a_2(-1)+1/12)
-2 a_4(-1).
\end{align*}
Now we use the equality
$a_4(-r)=-\wp''(\phi_{-r})$ given by
Lemma \ref{thm:an} in Section \ref{sec:Weierstrass_elliptic}
and (\ref{eqn:beta_t_2_wp}).
We also note that
we can verify the equality
$\gamma_2-a_2(-1)+1/12=-e_1$
from Lemma \ref{thm:gammas} (ii), (iv) 
and (\ref{eqn:e123}) in Section \ref{sec:Weierstrass_elliptic}.
Then the above is written as
\begin{align*}
\widetilde{G}^{(4)}(1)
&= 22 \wp''(\phi_{-r})-12(\wp(\phi_{-r})-e_1)^2
+ 24 (\wp(\phi_{-r})-e_1)e_1+ 2 \wp''(\pi)
\nonumber\\
&= 22 \wp''(\phi_{-r})-12 \wp(\phi_{-r})^2
+48 e_1 \wp(\phi_{-r}) -36 e_1^2
+ 2 \wp''(\pi).
\end{align*}
Finally we use the differential equation 
(\ref{eqn:wp_diff}) of $\wp$.
Then (\ref{eqn:4th0}) is obtained. 
Proposition \ref{thm:differentials} is hence proved
and the proof of Theorem
\ref{thm:critical_line} (i) is complete. 

\subsubsection{Proof of Theorem \ref{thm:critical_line} (ii)}
\label{sec:ii}

By the definition and the properties of $\wp$
explained in Section \ref{sec:Weierstrass_elliptic},
the following is proved for $q \in (0, 1)$.

\begin{lem}
\label{thm:wp_increasing}
For $r \in (q, 1)$,
$\wp(\phi_{-r})$ is a monotonically increasing function
of $r$.
\end{lem}
By (\ref{eqn:sum0}) and (\ref{eqn:g2g3}), 
we see that $\kappa(r)$ given by (\ref{eqn:kappa})
is written as follows,
\[
\kappa(r) =2(\wp(\phi_{-r})-e_2)(\wp(\phi_{-r})-e_3)
+ 6(\wp(\phi_{-r})+e_1)(\wp(\phi_{-r})-e_1). 
\]
Hence
$\kappa(1) =2(e_1-e_2)(e_1-e_3)$ and
$\kappa(q)=6(e_3+e_1)(e_3-e_1)$.
Then by the inequalities (\ref{eqn:inequality}), 
we can conclude that
$\kappa(1)>0$ and $\kappa(q)<0$.
By (\ref{eqn:kappa}), 
we have
$\kappa(r)= 5 (\wp(\phi_{-r})-\wp_+)(\wp(\phi_{-r})-\wp_-)$
with the roots
$\wp_{\pm}=\wp_{\pm}(q)= (
-2 e_1 \pm \sqrt{24 e_1^2+10 g_2} )/10$
satisfying $\wp_{-} < 0 < \wp_+$.
Since monotonicity is guaranteed by 
Lemma \ref{thm:wp_increasing} for $r \in (q, 1)$,
$r_0$ is the unique zero
of $\kappa$ in the interval $(q, 1)$.
This is determined by
\begin{equation}
\wp(\phi_{-r_0}) = \wp_+,
\label{eqn:r_critical}
\end{equation}
which is equivalent to 
\begin{equation}
\phi_{-r_0}=\wp^{-1}(\wp_+)
\quad \iff \quad
r_0=- e^{\sqrt{-1} \wp^{-1}(\wp_+)}
=e^{\sqrt{-1}(-\pi+\wp^{-1}(\wp_+))}.
\label{eqn:r0}
\end{equation}
Using (\ref{eqn:sum0}), (\ref{eqn:g2g3}), and
(\ref{eqn:inequality}), we can verify by (\ref{eqn:wp+}) that
$e_3<e_2<\wp_+<e_1$. 
Hence (\ref{eqn:wp_inverse}) implies 
\[
-\pi+\wp^{-1}(\wp_+)
=\frac{\sqrt{-1}}{2} \int_{\wp_+}^{e_1}
\frac{ds}{\sqrt{(e_1-s)(s-e_2)(s-e_3)}}
\]
and (\ref{eqn:r0}) gives (\ref{eqn:tcq}).
The proofs of 
(\ref{eqn:ineq_a}) and
the assertion mentioned below it
are complete. 

\subsubsection{Proof of Theorem \ref{thm:critical_line} (iii)}
\label{sec:iii}
In the limit $q \to 0$, 
we have (\ref{eqn:q0_limit2}) and
(\ref{eqn:wp+}) gives
$\wp_+(0)=(-2+3 \sqrt{6})/60$.
The integral appearing in 
(\ref{eqn:tcq}) is then reduced to
\begin{align*}
\frac{1}{2} \int_{\wp_+(0)}^{1/6}
\frac{ds}{(s+1/12)\sqrt{1/6-s}}
&=-\log 
\frac{1-2\sqrt{1/6-\wp_+(0)}}
{1+2 \sqrt{1/6-\wp_+(0)}}
=-\log 
\frac{1-\frac{\sqrt{4-\sqrt{6}}}{\sqrt{5}}}
{1+\frac{\sqrt{4-\sqrt{6}}}{\sqrt{5}}}.
\end{align*}
Hence the first expression for 
$r_{\rm c}$ in (a) is obtained.
\begin{rem}
\label{rem:kappa_q0}
If we apply (\ref{eqn:q0_limit1}) and (\ref{eqn:q0_limit2})
in Section \ref{sec:q_0} to (\ref{eqn:kappa}),
then we have
\begin{align}
\kappa_0(r) := 
\lim_{q \to 0} 
\kappa(r; q)
&= - \frac{r^4+12 r^3-58r^2+12r+1}{16 (1+r)^4}
\nonumber\\
&= - \frac{(r+r^{-1})^2+12(r+r^{-1})-60}
{16(r^{1/2}+r^{-1/2})^4}.
\label{eqn:kappa_q0}
\end{align}
Since we have assumed $r_{\rm c} \in (0, 1)$,
$r_{\rm c}+r_{\rm c}^{-1} \in (2, \infty)$.
Then we see that $r=r_{\rm c}$ satisfies the 
equation
\[
r+r^{-1}
=2(2 \sqrt{6}-3) \iff
r^2-2(2\sqrt{6}-3) r+1 =0.
\]
The above quadratic equation has
two positive solutions which are reciprocal to each other.
The second expression for $r_{\rm c}$ 
in (a) is the smaller one of them. 
\end{rem}

From (\ref{eqn:wp_expansion2})
and (\ref{eqn:e123}), we have
\begin{align*}
\wp(\phi_{-r})
&=- 1/12+ r/(1+r)^2
+ \big\{ 2+ (r+r^{-1}) \big\} q^2
\nonumber\\
& \quad + \big\{
6+ (r+r^{-1}) 
-2 (r^2+ r^{-2} ) \big\} q^4 +{\rm O}(q^6),
\nonumber\\
e_1 &= 1/6+ 4 q^2 + 4 q^4 +{\rm O}(q^6),
\quad
e_2 = - 1/12+2q -2 q^2+8 q^3 -2 q^4
+{\rm O}(q^5),
\nonumber\\
e_3 &= -1/12 -2q -2 q^2-8 q^3 - 2 q^4
+{\rm O}(q^5),
\quad
g_2 = 1/12 
+20 q^2 + 180 q^4
+{\rm O}(q^6).
\end{align*}
Then the equation (\ref{eqn:r_critical}) 
is expanded in the variable $q$ as
\begin{align*}
& - 1/12 +(r_0^{1/2}+r_0^{-1/2})^{-2}
+ \big\{ 2+ (r_0+r_0^{-1} ) \big\} q^2
+ \big\{
6+ (r_0+ r_0^{-1}) 
-2 (r_0^2+ r_0^{-2}) \big\} q^4
\nonumber\\
& \quad
= -(2 - 3 \sqrt{6})/60 
-2(6 - 29 \sqrt{6}) q^2/15
-2(18+2533 \sqrt{6})q^4/45
+{\rm O}(q^6).
\end{align*}
Put
$r_0=r_{\rm c}+c_1 q + c_2 q^2 +{\rm O}(q^3)$
and use the value of $r_{\rm c}$ given by (a).
Then we have $c_1=0$ and the assertion (b) is proved. 

For (c) we consider the asymptotics of
the equation (\ref{eqn:r_critical}).
By (\ref{eqn:q1asym}) and (\ref{eqn:q1asym2}) we have
$( 1/12 + e^{-\phi_{-r_0(q)}
/|\tau_q|} + e ^{-(2\pi-\phi_{-r_0(q)})/|\tau_q|})/|\tau_q|^2 
\sim 1/(12 |\tau_q|^2)$
in $|\tau_q| \to 0$.
This is satisfied if and only if
$e^{-\phi_{-r_0(q)}/|\tau_q|} 
+ e ^{-(2\pi-\phi_{-r_0(q)})/|\tau_q|} = 0$,
that is, 
$\cos( (\pi-\phi_{-r_0}(q))/\tau_q ) =0$.
Under the setting (\ref{eqn:omega1})
with $r \in (0,1) $, this is realized by
\[
\pi- \phi_{-r_0(q)}= - \pi \tau_q/2
\iff
r_0(q)=-e^{\sqrt{-1} \phi_{-r_0}(q)}
=e^{\sqrt{-1} \pi \tau_q/2}= q^{1/2}.
\]
Since $q^{1/2}=(1-(1-q))^{1/2} 
\sim 1 - (1-q)/2$ as $q \to 1$, (c) is proved.

Hence the proof of Theorem \ref{thm:critical_line} (iii) 
is complete.

\subsection{Proof of Proposition \ref{thm:negative_corr_q0}}
\label{sec:proof_negative}

By taking the $q \to 0$ limit 
in Lemma \ref{thm:bounds}, the following is obtained.

\begin{lem}
\label{thm:upper_bound_q0}
Assume that $r >0$. 
If we set $a=|z|, b=|w|$, $a, b \in (0, 1]$, then
$g_{\D}(z, w; r) \leq g_{\D}(-a, b; r)$,
where
\begin{align*}
g_{\D}(-a, b; r)
&= \frac{(a+b)^2 (1+r a^2)^2 (1+r b^2)^2}
{(1+ab)^4 (1+r) (1+r a^4) (1+ r b^4) (1+r a^2 b^2)^3}
\nonumber\\
& \quad
\times
\Big\{ a^6 b^6 (2-a^2+2ab-b^2+2a^2b^2) r^2
\nonumber\\
& \qquad 
+ a^2 b^2 (a^2-2ab+4a^3b+b^2+a^4 b^2+4a b^3-2a^3b^3+a^2b^4)r
\nonumber\\
& \qquad 
+ (2-a^2+2ab-b^2+2a^2b^2) \Big\}.
\end{align*}
\end{lem}
From now on we will assume $r \in (0, 1]$. 
It is easy to see that
$g_{\D}(-a, b; r)=g_{\D}(a, -b; r)$,
and 
$g_{\D}(-a, 1; r)=g_{\D}(-1, b; r) =1$,
 $a, b \in (0, 1]$.
We define a function $D(a,b;r)$ by 
\[
\frac{\partial g_{\D}(-a, b; r)}{\partial a}
=
\frac{4 a^7 b^4 r^{5/2} D(a,b;r)
(1-a)(1+a)(1-b)^2 (1+b)^2 (a+b) (1+ra^2) (1+rb^2)^2}
{(1+ab)^5 (1+r) (1+ra^4)^2 (1+ra^2 b^2)^4 (1+rb^4)}.
\]
The above implies that 
if $D(a,b;r) \geq 0$ for $r \in (0, r_{\rm c})$,
$\forall (a, b) \in (0, 1]^2$, then
Proposition \ref{thm:negative_corr_q0} is proved.

We can prove the following.
\begin{lem}
\label{thm:bound_v}
Let $p(x) := x+1/x$ and
$\widetilde{D}(a, b; s) = p(a^7 b^4 s^5)
+ 13 p(a^3 b^2 s^3) - 46 p(a^4 b^2 s)$.
Then
$D(a, b; r) \geq \widetilde{D}(a, b; r^{1/2})$, 
$\forall (a, b, r) \in (0, 1]^3$.
\end{lem}
\begin{proof} 
A tedious but direct computation shows that 
\begin{align}
&D(a,b; r) = p(a^7 b^4 r^{5/2}) 
\nonumber\\
& \quad + \big\{p(a^3 b^2 r^{3/2}) + 5 p(a^5 b^2 r^{3/2}) 
-  p(a^2 b^3 r^{3/2})+ 3 p(a^4 b^3 r^{3/2}) 
+ 2 p(a^6 b^3 r^{3/2}) + 3 p(a^5 b^4 r^{3/2}) \big\} 
\nonumber\\
&\quad - \big\{10 p(a r^{1/2}) + 5 p(a^3 r^{1/2}) 
+ 2 p(a b^{-2} r^{1/2}) + 3 p( b^{-1} r^{1/2}) + 9 p(a^2 b^{-1} r^{1/2}) 
\nonumber\\
&\qquad + 5 p( b r^{1/2}) + 
10 p(a^2 b r^{1/2}) + p(a^4 b r^{1/2}) + 2 p(a b^2 r^{1/2}) - 
   p(a b^4 r^{1/2})
\big\}.  
\label{eqn:D_abr}
\end{align}
We note that $p(x)$ is decreasing on $(0,1]$ and $p(x) = p(x^{-1})$. 
By the monotonicity of $p(x)$, the following inequalities 
are guaranteed,
\begin{align}
&3p(a^4 b^3 r^{3/2}) \ge p(a^2 b^3 r^{3/2}) + 2 p(a^3 b^2 r^{3/2}),
\label{eqn:ineq1}
\\
& 2p(a b^2 r^{1/2}) \le  p(a b^4 r^{1/2}) + p(a^4 b^2 r^{1/2}),
\label{eqn:ineq2}
\\
& \max\{p(a b^{-2} r^{1/2}), \ p( b^{-1} r^{1/2}), \ p(a^2 b^{-1} r^{1/2})
\} \le p(a^4 b^2 r^{1/2}).
\label{eqn:ineq3}
\end{align}
For (\ref{eqn:D_abr}) 
we apply (\ref{eqn:ineq1}) in the first braces
and do (\ref{eqn:ineq2}) and (\ref{eqn:ineq3})
in the second braces. 
Then the desired inequality readily follows.
\end{proof}

Now we prove the following.
\begin{lem}
\label{thm:ms}
Let 
\[
m(s) := \inf_{(a,b,u) \in (0,1] \times (0,1] \times (0,s]} 
\widetilde{D}(a,b; u),
\]
and $s_{\rm c} := r_{\rm c}^{1/2}$.
Then, $m(s)$ attains its minimum at $(1,1,s)$ and 
$m(s) \ge 0$ if and only if $0 < s \le s_{\rm c}$. 
\end{lem}
\begin{proof}
We fix $s \in (0, 1]$. For $x \in (0,1]$, we consider the curve 
$C_x$ defined by $a^2 b = x$, or equivalently 
by $b = x/a^2$. We note that 
\[
 (0,1]^2 = \bigcup_{x \in (0,1]} \{(a,b) \in (0,1]^2 : a^2b = x, \ 
x^{1/2} \le a \le 1\}. 
\]
On the curve $C_x$, we can write 
$\widetilde{D}(a, a^{-2}x; s) = p(a^{-1} x^4 s^5) 
+ 13 p(a^{-1} x^2 s^3) - 46p(x^2 s)$, $x^{1/2} \le a \le 1$.
Since $p'(x) =1 - x^{-2} \le 0$ for $x \in (0,1]$, we have 
\[
\frac{\partial}{\partial a} 
\widetilde{D}(a, a^{-2}x; s) = p'(a^{-1} x^4 s^5) (-a^{-2} x^4 s^5) 
+ 13 p'(a^{-1} x^2 s^3) (-a^{-2} x^2 s^3) \ge 0, 
\]
and hence $\widetilde{D}(a, a^{-2}x; s)$ 
attains its minimum at $a = x^{1/2}$ and $b=1$. 
Therefore, for $s \in (0, 1]$, 
\begin{equation}
\inf_{(a,b) \in (0,1] \times (0,1]} \widetilde{D}(a, b;s) 
= \inf_{x \in (0,1]} \widetilde{D}(x^{1/2}, 1; s)
= \inf_{a \in (0,1]} \widetilde{D}(a, 1; s). 
\label{eqn:inf1} 
\end{equation}
For $(a, u) \in (0,1] \times (0,s]$, we consider 
$\widetilde{D}(a,1; u) = p(a^7 u^5) + 13 p(a^3 u^3) - 46 p(a^4 u)$. 
For $y \in (0,s]$, we consider the curve $C'_y$ defined by $a^4 u = y$ 
or equivalently, by $a = (y/u)^{1/4}$. Note that 
\[
 (0,1] \times (0,s] 
= \bigcup_{y \in (0,s]} \{(a,u) \in (0,1] \times (0,s] : a^4 u = y, \ 
y \le u \le s\}. 
\]
Then, on the curve $C'_y$, we can write 
$\widetilde{D}((y/u)^{1/4}, 1; u) 
 =  p(y^{7/4} u^{13/4}) + 13 p(y^{3/4} u^{9/4}) 
- 46 p(y)$, $y \le u \le s$.  
Since $(\partial/\partial u) \widetilde{D}((y/u)^{1/4}, 1; u) \le 0$, 
we conclude that 
$\widetilde{D}((y/u)^{1/4}, 1; u)$ attains its minimum at $u = s$, 
and hence,  
from (\ref{eqn:inf1}), we have 
\begin{equation}
m(s)
= \inf_{(a, u) \in (0,1] \times (0,s]} \widetilde{D}(a, 1; u)
= \inf_{y \in (0,s]} \widetilde{D}((y/s)^{1/4}, 1; s)
= \inf_{a \in (0,1]} \widetilde{D}(a, 1; s). 
\label{eqn:inf2} 
\end{equation}
It suffices to show $m(s) \ge 0$ if $s \le s_{\rm c}$.
Since $x p'(x) = q(x) := x-1/x$, we can verify easily that 
\begin{align*}
a \frac{\partial}{\partial a} \widetilde{D}(a,1; s) 
&= 7 q(a^7 s^5) + 13 \cdot 3q(a^3 s^3) - 46 \cdot 4q(a^4 s) \\
&=\frac{1}{a^7 s^5} (7 a^{14} s^{10} - 7 + 39 a^{10} s^8-39 a^4 s^2 
-184 a^{11} s^6 + 184 a^3 s^4) 
=: \frac{1}{a^7 s^5} \delta (a,s). 
\end{align*}
For $a, s \in (0,1]$, we see that 
\begin{align*}
\delta(a,s) 
&\le 7 s^{10} + 39 s^8+184 a^2 s^4-39 a^4 s^2-7 \\
&= 7s^{10} + 39 s^8 - 39s^2 \Big(a^2 -\frac{92}{39} s^2 \Big)^2 
+ \frac{92^2}{39}s^6 -7 
\leq 7s^{10} + 39 s^8 + \frac{92^2}{39}s^6 -7. 
\end{align*}
Since the last function of $s$ is increasing in $(0,1]$ 
and it takes a negative value at $s=11/20$, 
we have $\delta(a,s) < 0$ for $(a,s) \in (0,1] \times (0, 11/20]$. 
Therefore, 
$\widetilde{D}(a,1,s)$ is decreasing in $a$ for $s \in (0, 11/20]$, 
which together with (\ref{eqn:inf2}) implies 
\[
 m(s) = \widetilde{D}(1,1,s) 
 = \frac{1+s^2}{s^5} (s^8 + 12s^6 - 58s^4 + 12s^2 +1). 
\]
Here we note Remark \ref{rem:kappa_q0}
given in Section \ref{sec:iii}. 
Consequently, $m(s) \ge m(s_{\rm c}) = 0$ 
for $s \in (0,s_{\rm c}]$ as $s_{\rm c} 
=r_{\rm c}^{1/2} = 0.533 \dots \le 11/20$. 
\end{proof}
\begin{rem}
\label{rem:positive_corr}
We see that 
\[
g_{\D}(-r^{-1/4}, r^{-1/4}; r)
= \frac{6+r+r^{-1}}{4 (r^{1/2}+r^{-1/2})} =: \widetilde{g}(r).
\]
It is readily verified that 
$\widetilde{g}(1) = 1$
and $d \widetilde{g}(r)/dr=
(r-1)^3/\{8r^{3/2}(r+1)^2\} \geq 0$, $r \ge 1$. 
Then, $\widetilde{g}(r) > 1$ for any $r > 1$. 
Since $1/r_{\rm c}=3.51 \cdots > 1$, 
the PDPP $\cZ_{X_{\D}^r}$ is still in the partially attractive phase 
although $\kappa_0(r)$ becomes negative when $r \in (1/r_{\rm c}, \infty)$ 
due to the symmetry $r \leftrightarrow 1/r$ built
in \eqref{eqn:kappa_q0}.
\end{rem}

\SSC{Concluding Remarks}
\label{sec:concluding}

Peres and Vir\'ag proved a relationship
between the Szeg\H{o} Kernel $S_{\D}$ and 
the Bergman kernel $K_{\D}$ 
in the context of probability theory: 
A GAF 
is defined so that its covariance kernel is given
by $S_{\D}$. Then the zero point process $\cZ_{X_{\D}}$
is proved 
to be a DPP for which the correlation kernel is
given by $K_{\D}$.
The background of their work is explained in
the monograph \cite{HKPV09}, in which 
we find that the \textit{Edelman--Kostlan formula}
\cite{EK95} gives the density of $\cZ_{X_{\D}}$ 
with respect to $m/\pi$ as
\[
\rho^1_{\D, \mathrm{PV}}(z)
=\frac{1}{4} \Delta \log S_{\D}(z,z),
\quad z \in \D,
\]
where $\Delta := 4 \partial_z \partial_{\zbar}$.
Moreover, we have the equality
\begin{equation}
K_{\D}(z, w)= 
\partial_z \partial_{\wbar} 
\log S_{\D}(z, w)=S_{\D}(z, w)^2,
\quad z, w \in \D.
\label{eqn:EK_like}
\end{equation}
On the other hand, 
as explained above (\ref{eqn:S_K_D}), 
for the kernels on simply connected domain
$D \subsetneq \C$, the equality 
\begin{equation}
S_{D}(z, w)^2= K_{D}(z, w), \quad
z, w \in D,
\label{eqn:S_K_D_special}
\end{equation}
is established.

In the present paper, we have reported our work to
generalize the above to 
a family of GAFs and their zero point processes
on the annulus $\A_q$.
By comparing the expression (\ref{eqn:density_Aqt}) 
for the density obtained from Theorem \ref{thm:mainA1} 
with (\ref{eqn:log_deriv_Aq2})
in Proposition \ref{thm:log_derivative} in
Appendix \ref{sec:log_derivative} given below, 
we can recover the Edelman--Kostlan formula as follows,
\[
\rho^1_{\A_q}(z; r)
=\frac{\theta(-r)}{\theta(-r|z|^4)}
S_{\A_q}(z, z; r|z|^2)^2
=\frac{1}{4} \Delta \log S_{\A_q}(z, z; r),
\quad z \in \A_q.
\]
However, (\ref{eqn:EK_like}) does not
hold for the weighted Szeg\H{o} kernel for $H^2_r(\A_q)$.
As shown by (\ref{eqn:log_deriv_Aq1}), 
the second log-derivative of $S_{\A_q}(z, w; r)$
cannot be expressed by $S_{\A_q}(z, w; r)$ itself
but a new function $S_{\A_q}(z, w; r z \wbar)$
should be introduced. 
In addition the proportionality between the square of 
the Szeg\H{o} kernel
and the Bergman kernel (\ref{eqn:S_K_D_special})
is no longer valid for the point processes on $\A_q$
as shown in Proposition \ref{thm:relation} in 
Appendix \ref{sec:relation}.

The Borchardt identity
plays an essential role in the proof of Peres and Vir\'ag,
which is written as
\[
\perdet_{1 \leq i, j \leq n}
\Big[S_{\D}(z_i, z_j) \Big]
= \det_{1 \leq i, j \leq n} 
\Big[ S_{\D}(z_i, z_j)^2 \Big],
\quad n \in \N, \quad
z_1, \dots, z_n \in \D.
\]
Since the $n$-point correlation function 
$\rho^n_{\D, \mathrm{PV}}(z_1, \dots, z_n)$ of $\cZ_{X_{\D}}$
is given by the left-hand side, $\forall n \in \N$, 
this equality proves that $\cZ_{X_{\D}}$ is a DPP.
For $S_{\A_q}$ the corresponding equality does not hold.
We have proved, however, that all correlation functions
of our two-parameter family of zero point processes 
$\{\cZ_{X_{\A_q}^r} : q \in (0,1), r > 0 \}$ 
on $\A_q$ can be expressed using
$\perdet$ defined by (\ref{eqn:perdet})
and we stated that
they are permanental-determinantal point processes
(PDPPs). 

We would like to place an emphasis on the fact that
the present paper is not an incomplete work nor
just replacing determinants by $\perdet$'s. 
The essentially new points,
which are not found in the previous works \cite{PV05,HKPV09}, 
are the following:
\begin{description}
\item{(i)} \,
Even if we start from the GAF whose covariance kernel
is given by the original Szeg\H{o} kernel 
$S_{\A_q}(\cdot, \cdot)=S_{\A_q}(\cdot, \cdot; q)$ 
on $\A_q$, 
the full description of conditioning with zeros
needs a series of new covariance kernels.

\item{(ii)} \,
The covariance kernels of the induced GAFs
generated by conditioning of zeros are identified with
the \textit{weighted} Szeg\H{o} kernel
$S_{\A_q}(\cdot, \cdot; \alpha)$ studied by
Mccullough and Shen \cite{MS94}. 
In the present study, the weight parameter $\alpha$
plays an essential role, since it is determined by
$\alpha= r \prod_{\ell=1}^n |z_{\ell}|^2$ 
and represents a geometrical information of the zeros in $\A_q$ 
$\{z_1, \dots, z_n\}, n \in \N$ 
put in the conditioning.

\item{(iii)} \,
Corresponding to such an inductive structure of
conditional GAFs, the correlation kernel 
of our PDPP of $\cZ_{\A_q^r}, r > 0$
is given by 
$S_{\A_q}(\cdot, \cdot; \alpha)$
with $\alpha=r \prod_{\ell=1}^n |z_{\ell}|^2$
in order to give the correlation function
for the points $\{z_1, \dots, z_n\}$;
$\rho^n_{\A_q}(z_1, \dots, z_n; r)$.
In addition, the $n$-product measure 
of the Lebesgue measure on $\C$ divided by $\pi$,
$(m/\pi)^{\otimes n}$, 
should be weighted by
$\theta(-r)/\theta(-r \prod_{k=1}^n |z_k|^4)$
to properly provide $\rho^n_{\A_q}(\cdot; r)$.

\item{(iv)} \,
The parameter $r$ also 
plays an important role to describe the
symmetry of the GAF and its zero point process
under the transformation
which we call the $(q, r)$-inversion,
\begin{equation}
(z, r) \, \longleftrightarrow \,
\left( \frac{q}{z}, \frac{q^2}{r} \right) 
\, \in \A_q \times (0, \infty).
\label{eqn:inversion_relation}
\end{equation}
And if we adjust $r=q$ the 
GAF and its zero point process become
invariant under conformal transformations 
which preserve $\A_q$. 
\end{description}

Inapplicability of the Borchardt identity 
to our zero poin processes
$\cZ_{X_{\A_q}^r}$ causes interesting behaviors of them as 
\textit{interacting particle systems}.
We have proved that the short-range interaction
between points is
repulsive with index $\beta=2$ 
in a similar way to the usual DPP,
but attractive interaction is also observed in $\cZ_{X_{\A_q}^r}$.
The index for decay of correlations is given by
$\eta=4$. 
We found that there is a special value $r=r_0(q) \in (q, 1)$
for each $q \in (0, 1)$ at which
the coefficient of the power-law decay of 
correlations changes its sign.
We have studied the zero point process obtained
in the limit $q \to 0$, which has
a parameter $r \in [0, \infty)$.
In this PDPP $\cZ_{X_{\D}^r}$,
$r_{\rm c}:=r_0(0)$ can be regarded as the critical value
separating two phases in the sense that if $r \in [0, r_{\rm c})$
the zero point process is completely repulsive,
while if $r \in (r_{\rm c}, \infty)$
attractive interaction emerges depending on
distances between points. 

There are many open problems, since such 
PDPPs have not been studied so far.
Here we list out some of them.
\begin{description}
\item{(1)} \,
We prove that the GAF $X_{\A_q}^r$ and its 
zero point process $\cZ_{X_{\A_q}^r}$ have 
the rotational invariance and
the $(q, r)$-inversion symmetry,
and when $r=q$, they are invariant under conformal transformations
which preserve $\A_q$ (Proposition \ref{thm:conformal_Aq}).
We claimed in Remark \ref{rem:parameter_L} that
$X_{\A_q}^r$ can be extended to a one-parameter family of GAFs
$\{X_{\A_q}^{r, (L)} : L \in \N \}$ having the rotational invariance and
the $(q, r)$-inversion symmetry
and this family is an extension of $\{X_{\D}^{(L)}: L \in \N\}$
studied in 
\cite[Sections 2.3 and 5.4]{HKPV09}
in the sense that
$\lim_{q \to 0} X_{\A_q}^{r, (L)}|_{r=q} \dis= X_{\D}^{(L)}$, $L \in \N$.
In \cite[Section 2.5]{HKPV09}, it is argued that 
$\{X_{\D}^{(L)}:  L > 0\}$ is the only GAFs,
up to multiplication by deterministic non-vanishing analytic functions,
whose zeros are isometry-invariant under the conformal transformations
preserving $\D$. 
This assertion is proved by the fact that the zero point process 
of the GAF is completely determined
by its first correlation function.
Therefore, the ``canonicality'' of the GAFs
$\{X_{\D}^{(L)} : L > 0\}$ is guaranteed by the
uniqueness, up to multiplicative constant, 
of the density function with respect to $m(dz)/\pi$,
$\rho_{\D, \mathrm{PV}}^1(z)=1/(1-|z|^2)^2$, 
which is invariant 
under the M\"{o}bius transformations preserving $\D$.
We have found, however, that the density function
with parameter $r>0$, 
$\varrho(z; r), z \in \A_q$ is not 
uniquely determined to be
$\rho_{\A_q}^1(z; r)$ as (\ref{eqn:density_Aqt}) 
by the requirement that it is rotationally invariant
and having the $(q,r)$-inversion symmetry.
For example, we have the three-parameter 
($\alpha_1 > 1-\alpha_2$, $\alpha_2>0$, $\alpha_3 \in \R$)
family of density functions,
\[
\varrho(z; r; \alpha_1, \alpha_2, \alpha_3)
=\frac{\theta(-r)^{\alpha_3}}
{\theta(-|z|^{2 \alpha_1} r^{\alpha_2})}
f^{\rm JK}(|z|^2, -|z|^{2 \beta_1} r^{\beta_2})^2
\]
with $\beta_1=(\alpha_1-\alpha_2)(\alpha_1+\alpha_2-1)/4+1/2$,
$\beta_2=\alpha_2(\alpha_1+\alpha_2-1)/2$, 
which satisfy the above requirement of symmetry.
We see that
$\varrho(z; r; 2, 1, 1)=\rho_{\A_q}^1(z; r)$ and
$\lim_{q \to 0} \varrho(z; q; \alpha_1, \alpha_2, \alpha_3)
=\rho_{\D, \mathrm{PV}}^1(z)$.
The present study of the GAFs on $\A_q$
and their zero point processes will be 
generalized in the future.

\item{(2)} \,
As shown by (\ref{eqn:density_edges}), 
the asymptotics of the density of zeros 
$\rho_{\A_q}^1(z) \sim (1-|z|^2)^{-2}$
with respect to $m(dz)/\pi$ 
in the vicinity of the outer boundary of $\A_q$ 
can be identified with the metric 
in the hyperbolic plane called 
the \textit{Poincar\'e disk model}
(see, for instance, \cite{Hel84,Cha06}). 
The zero point process 
$\cZ_{X_{\D}}$ of Peres and Vir\'ag
can be regarded as a uniform DPP on 
the Poincar\'e disk model 
\cite{PV05,DL19,BQ18+}.
Is there any meaningful geometrical space 
in which the present zero point process
$\cZ_{X_{\A_q}^r}$ seems to be uniform?
As mentioned in Remark \ref{rem:Remark_condition1},
conditioning with zeros does 
not induce any new GAF on $\D$
\cite{PV05}, but it does on $\A_q$.
Is it possible to give some geometrical explanation
for such a new phenomenon appearing 
in replacing $\D$ by $\A_q$ 
reported in the present paper?

\item{(3)} \,
As mentioned above and in 
Theorem \ref{thm:critical_line} (i), 
we have found power-law decays of 
unfolded 2-correlation functions to the unity 
with an index $\eta=4$. 
Although the coefficient of this power-law 
changes depending on $q$ and $r$, 
the index $\eta=4$ seems to be
universal in the PDPPs
$\cZ_{X_{\A_q}^r}$,
$\cZ_{X_{\D}^r}$ and 
the DPP of Peres and Vir\'ag $\cZ_{X_{\D}}$ 
(except the PDPPs at $r=r_0(q) \in (q, 1), q \in [0, 1)$).
The present proof of Theorem \ref{thm:critical_line} (i)
relied on brute force calculations
showing vanishing of derivatives up to the third order.
Simpler proof is required. 
In the metric of a proper hyperbolic space,
the decay of correlation will be exponential.
In such a representation, what is the meaning of
the `universal value' of $\eta$?

\item{(4)} \,
As mentioned in Remark \ref{rem:Remark_r_infinity}, 
the simplified 
PDPPs $\{\cZ_{X_{\D}^r} : r \in (0, \infty)\}$ 
can be regarded as an interpolation between 
the DPP of Peres and Vir\'ag 
$\cZ_{X_{\D}}$ and its deterministic perturbation at the origin
$\cZ_{X_{\D}} +\delta_0$. 
The first approximation of the perturbation of the deterministic zero near $r=\infty$ 
is given by $\frac{-1}{\sqrt{1+r}}\zeta_0/\zeta_1$ by solving 
the approximated linear equation $\frac{\zeta_0}{\sqrt{1+r}} + \zeta_1 z=0$. 
Here the ratio $\zeta_0/\zeta_1$ is distributed according to 
the push-forward of the uniform distribution on the unit sphere 
by the stereographic projection
(see Krishnapur \cite{Kri09} for the matrix generalization).
Can we trace such a flow of zeros in $\{\cZ_{X_{\D}^r} : r \in (0,\infty)\}$ 
more precisely?

\item{(5)} \,
The simplified 
PDPP $\cZ_{X_{\D}^r}$ 
was introduced as a $q \to 0$
limit of the PDPP $\cZ_{\X_{\A_q}^r}$ in this paper. 
On the other hand, as shown by (\ref{eqn:GAF_A0t})
the GAF $X_{\D}^r$
can be obtained from the 
GAF $X_{\D}$ of Peres and Vir\'ag 
by adding a one-parameter ($r >0 $)
perturbation on a single term.
Can we explain the hierarchical structures of
these GAF and PDPP on $\D$ and the 
existence of the critical value $r_{\rm c}$
for correlations of $\cZ_{X_{\D}^r}$ 
apart from all gadgets related to elliptic functions?
Can we expect any interesting phenomenon 
at $r=r_{\rm c}$?

\item{(6)} \,
As mentioned at the end of Section \ref{sec:correlations},
the zero point process 
of the GAF $X_{\D}$ studied by
Peres and Vir\'ag \cite{PV05} is the DPP $\cZ_{X_{\D}}$, 
whose correlation kernel is given by
$K_{\D}(z, w)=S_{\D}(z,w)^2=1/(1-z \wbar)^2$,
$z, w \in \D$ with respect to the
reference measure $m/\pi$.
Krishnapur \cite{Kri09} introduced a one-parameter ($\ell \in \N$)
extension of DPPs $\{\cZ_{X_{\D}}^{(\ell)} : \ell \in \N\}$,
whose correlation kernels are given by
$K_{\D}^{(\ell)}(z, w)=1/(1-z \wbar)^{\ell+1}$ with respect to
the reference measure $\ell (1-|z|^2)^{\ell-1} m/\pi$ on $\D$.
He proved that $\cZ_{\D}^{(\ell)}$ is realized as the zeros of
$\det [ \sum_{n \in \N} \ssG_n z^n]$,
where $\{\ssG_n \}_{n \in \N}$ are i.i.d.~complex Ginibre random matrices
of size $\ell \in \N$.
A similar extension of the present PDPP
$\cZ_{\A_q^r}$ on $\A_q$ will be challenging.

\item{(7)} \,
For $0 < t \leq 1$, let
$\D_t:=\{z \in \C: |z| < t\}$. 
The CLT for the number of points $\cZ_{X_{\D}}(\D_t)$ 
as $t \to 1$ can be easily shown, 
since $\cZ_{X_{\D}}$ is a DPP and then $\cZ_{X_{\D}}(\D_t)$ 
can be expressed as a sum of independent Bernoulli random variables
\cite[Corollary 3 (iii)]{PV05} \cite{Shi06}. 
For $0 < q \leq s < t \leq 1$, let
$\A_{s, t}:=\{z \in \C:  s < |z| < t\}$. 
It would also be expected that the CLT holds for
$\big( \cZ_{\A_q}(\A_{s, \sqrt{q}}), 
\cZ_{\A_q}(\A_{\sqrt{q}, t})\big)$ as $s \to q$ and 
$t \to 1$ simultaneously in some sense. 
Is there a useful expression for those random variables as above and 
can we prove the CLTs for them? 

\item{(8)} \,
In the present paper, we have tried to
characterize the density functions
and the unfolded 2-correlation functions
of the PDPPs.
As demonstrated by Fig.\ref{fig:G_vee_plots},
change of global structure 
is observed at $r=r_{\rm c}$ for
the unfolded 2-correlation function.
Precise description of such a topological change is required. 
More detailed quantitative study 
would be also interesting.
For example, we can show that 
$G_{\A_q}^{\vee}(x; r)$ plotted in 
Fig.\ref{fig:G_vee_plots} attains its maximum
at $x=\sqrt{q}$ when $r=1$ and the value is
given by 
$G_{\A_q}^{\vee}(\sqrt{q}; 1)
=(q_2(q)^8+q_3(q)^8)/(16q q_1(q)^8) >1$, where
$q_1(q):=\prod_{n=1}^{\infty}(1+q^{2n})$,
$q_2(q):=\prod_{n=1}^{\infty}(1+q^{2n-1})$, and
$q_3(q):=\prod_{n=1}^{\infty}(1-q^{2n-1})$.
How about the local minima?
Moreover, 
systematic study on three-point and
higher-order correlations will be needed
to obtain a better understanding of differences 
between PDPPs and DPPs.

\item{(9)} \,
Matsumoto and one of the present authors \cite{MS13} 
studied the \textit{real} GAF on a plane
and proved that the zeros of the real GAF
provide a Pfaffian point process (PfPP).
There a
Pfaffian--Hafnian analogue of Borchardt's identity 
was used \cite{IKO05}. 
Is it meaningful to consider 
the Pfaffian--Hafnian analogues of
PDPPs?
Systematic study on the comparison
among DPPs, PfPPs, permanental PPs,
Hafnian PPs, 
PDPPs, and Hafnian--Pfaffian PPs 
will be a challenging future problem.

\item{(10)} \,
The symmetry of the present 
GAF and its zero point process under the
$(q, r)$-inversion (\ref{eqn:inversion_relation}) 
mentioned above and
the pairing of uncorrelated points
in the GAF $X_{\A_q}$ shown by 
Proposition \ref{thm:pair_structure} suggest that
the inner boundary $\gamma_q$ plays essentially the
same role as the outer boundary $\gamma_1$.
As an extension of the Riemann mapping function for
a simply connected domain $D \subsetneq \C$,
a function mapping
a multiply connected domain to
the unit disk is called the 
\textit{Ahlfors map} \cite{TT99,Bel16}.
(See also Remark \ref{rem:pair_zero} again.)
Could we use such Ahlfors maps to construct
and analyze GAFs and their zero point processes on 
general multiply connected domains?

\end{description}

\vskip 1cm
\noindent{\bf Acknowledgements} \,
The present study was stimulated by
the useful discussion with Nizar Demni
when the authors attended the workshop 
`Interactions between commutative and non-commutative probability'
held in Kyoto University, Aug 19--23, 2019.
The present authors would like to thank
Michael Schlosser for useful discussion on
the elliptic extensions of determinantal and permanental formulas.
The present work has been supported by
the Grant-in-Aid for Scientific Research (C) (No.19K03674),
(B) (No.18H01124), and (S) (No.16H06338) 
of Japan Society for the Promotion of Science (JSPS).

\appendix

\SSC{Hyperdeterminantal Point Processes}
\label{sec:hDPP}
Recall that determinant and permanent are
defined for an $n \times n$ matrix 
(a 2nd order tensor on an $n$-dimensional space) 
$M=(m_{i_1 i_2})_{1 \leq i_1, i_2 \leq n}$ as
\begin{align}
\det M &= \det_{1 \leq i_1, i_2 \leq n} [m_{i_1 i_2}]
:=\sum_{\sigma \in \mS_n} \sgn(\sigma) 
\prod_{\ell=1}^n m_{\ell \sigma(\ell)}
= \frac{1}{n !} \sum_{(\sigma_1, \sigma_2) \in \mS_n^2} 
\sgn(\sigma_1) \sgn(\sigma_2)
\prod_{\ell=1}^n m_{\sigma_1(\ell) \sigma_2(\ell)}, 
\nonumber\\
\per M &= \per_{1 \leq i_1, i_2 \leq n} [m_{i_1 i_2}]
:=\sum_{\sigma \in \mS_n} 
\prod_{\ell=1}^n m_{\ell \sigma(\ell)}
= \frac{1}{n !} \sum_{(\sigma_1, \sigma_2) \in \mS_n^2} 
\prod_{\ell=1}^n m_{\sigma_1(\ell) \sigma_2(\ell)}, 
\label{eqn:det_and_per}
\end{align}
where
$\mS_n$ denotes the symmetric group of order $n$.
The notion of determinant has been extended as follows.
Cayley's first hyperdeterminant is defined for a $k$-th order
tensor (hypermatrix) on an $n$-dimensional space
$M=( m_{i_1 \dots i_k})_{1 \leq i_1, \dots, i_k \leq n}$ as
\begin{equation}
\Det M =\Det_{1 \leq i_1, \dots, i_k \leq n} [m_{i_1 \dots i_k}]
:=\frac{1}{n!} \sum_{(\sigma_1, \dots, \sigma_k) \in \mS_n^k}
\prod_{i=1}^k \sgn(\sigma_i) \prod_{\ell=1}^n 
m_{\sigma_1(\ell) \dots \sigma_k(\ell)}.
\label{eqn:Cayley}
\end{equation}
It is straightforward to see that $\Det M=0$ if $k$ is odd.
Gegenbauer generalized (\ref{eqn:Cayley}) to the case where
some of the indices are non-alternated.
If $\cI$ denotes a subset of $\{1, \dots, k\}$, one has
\begin{equation}
{\DetI} M ={\DetI_{1 \leq i_1, \dots, i_k \leq n}} [m_{i_1 \dots i_k}]
:=\frac{1}{n!} \sum_{(\sigma_1, \dots, \sigma_k) \in \mS_n^k}
\prod_{i \in \cI} \sgn(\sigma_i) \prod_{\ell=1}^n 
m_{\sigma_1(\ell) \dots \sigma_k(\ell)}.
\label{eqn:Gegenbauer}
\end{equation}
These extensions of the determinant are called
\textit{hyperdeterminants}. 
See \cite{Mat08,EG09,LV10} and references therein.

\begin{lem}
\label{thm:hyperdet1}
Let $A=(a_{i_1 i_2})$ and $B=(b_{i_1 i_2})$ be $n \times n$ matrices.
Then $\per A \det B = \DetJ C$, 
where $C=(c_{i_1 i_2 i_3})$ is 
the $n \times n \times n$ hypermatrix with the entries
\begin{equation}
c_{i_1 i_2 i_3}=a_{i_2 i_1} b_{i_2 i_3},
\quad i_1, i_2, i_3 \in \{1, \dots, n\}. 
\label{eqn:hyperdet2}
\end{equation}
In particular, 
$\perdet M =\DetJ[ m_{i_2 i_1} m_{i_2 i_3}]$,
where $\perdet M$ is defined by (\ref{eqn:perdet}). 
\end{lem}
\begin{proof}
By the definition (\ref{eqn:det_and_per}), 
\begin{align*}
\per A \det B
&= \sum_{\tau_1 \in \mS_n} \prod_{i=1}^n a_{i \tau_1(i)}
\sum_{\tau_2 \in \mS_n} \sgn(\tau_2) \prod_{j=1}^n b_{j \tau_2(j)}
= \sum_{\tau_1 \in \mS_n} \sum_{\tau_2 \in \mS_n} \sgn(\tau_2)
\prod_{i=1}^n a_{i \tau_1(i)} b_{i \tau_2(i)}
\nonumber\\
&= \frac{1}{n!}
\sum_{\sigma_1 \in \mS_n} \sum_{\sigma_2 \in \mS_n} 
\sum_{\sigma_3 \in \mS_n} 
\sgn(\sigma_1^{-1} \circ \sigma_3)
\prod_{i=1}^n a_{i \, \sigma_1^{-1} \circ \sigma_2(i)} 
b_{i \, \sigma_1^{-1} \circ \sigma_3 (i)} \\
&= \frac{1}{n!}
\sum_{\sigma_1 \in \mS_n} \sum_{\sigma_2 \in \mS_n} 
\sum_{\sigma_3 \in \mS_n} 
\sgn(\sigma_1) \sgn(\sigma_3)
\prod_{i=1}^n a_{\sigma_1(i) \sigma_2(i)} 
b_{\sigma_1(i) \sigma_3 (i)}.
\end{align*}
We change the symbols of permutations
as $\sigma_1 \to \rho_2, \sigma_2 \to \rho_1, \sigma_3 \to \rho_3$.
Then the above is written as
$(1/n!) \sum_{\rho_1 \in \mS_n} \sum_{\rho_2 \in \mS_n} 
\sum_{\rho_3 \in \mS_n} \sgn(\rho_2) \sgn(\rho_3)
\prod_{i=1}^n a_{\rho_2(i) \rho_1(i)} b_{\rho_2(i) \rho_3(i)}$.
Hence if we assume (\ref{eqn:hyperdet2}), then this is written as
$(1/n!) \sum_{(\sigma_1, \sigma_2, \sigma_3) \in \mS_n^3}
\prod_{i \in \{2, 3\}} \sgn(\sigma_i) \prod_{j=1}^n 
c_{\sigma_1(j) \sigma_2(j) \sigma_3(j)}.$
By the definition (\ref{eqn:Gegenbauer}), 
the proof is complete.
\end{proof}

Theorem \ref{thm:mainA1} of the present paper 
can be written in the following way.
\begin{thm}
\label{thm:main_hDPP}
$\cZ_{X_{\A_q}^r}$ is a hyperdeterminantal point process (hDPP)
in the sense that 
it has correlation functions expressed by
hyperdeterminants as 
\begin{align*}
&\rho^n_{\A_q}(z_1,\dots,z_n; r) 
= \frac{\theta(-r)}{\theta( -r \prod_{k=1}^n |z_{k}|^4)}
\nonumber\\
& \qquad \qquad
 \times
\DetJ_{1 \leq i_1, i_2, i_3 \leq n}
\Big[
S_{\A_q} \Big(z_{i_2}, z_{i_1}; r \prod_{\ell=1}^n |z_{\ell}|^2 \Big)
S_{\A_q} \Big(z_{i_2}, z_{i_3}; r \prod_{\ell=1}^n |z_{\ell}|^2 \Big)
\Big],
\end{align*}
for every $n \in \N$ and
$z_1, \dots, z_n \in \A_q$
with respect to $m/\pi$.
\end{thm}

\SSC{Conformal Map from $\A_q$ to $D(s)$}
\label{sec:D_Ds}
A general Schwarz--Christoffel formula 
for conformal maps from $\A_q$ to
a doubly connected domain
is given as Eq.(1) in \cite{DEP01} and
on page 68 in \cite{DT02}. 
We can read that a conformal map from $\A_q$
to a chordal standard domain
$D(s)$, $s >0$ is given in the form
\[
f(z)=C \int_{-1}^z 
\frac{\theta(-\sqrt{-1} q u, \sqrt{-1} q u)}
{\theta(u)} d u,
\]
where $C$ is a parameter. 
We can show that the integral is transformed into
an integral of the Weierstrass $\wp$-function
and hence the map is expressed by
the $\zeta$-function. 
A result is given by (\ref{eqn:Hqz}) in
Remark \ref{rem:slit_domain}.
We note that the obtained function $H_q$
is related to the Villat kernel $\cK$
(see, for instance, \cite{FK14}),
\[
\cK(z)=\cK(z; q) := 
\sum_{n \in \Z} \frac{1+q^{2n} z}{1-q^{2n}z}
= \frac{1+z}{1-z}
+2 \sum_{n=1}^{\infty} \left(
\frac{q^{2n}}{q^{2n}-z} + \frac{q^{2n}z}{1-q^{2n}z} \right),
\quad z \in \A_q,
\]
by a simple relation 
$H_q(z)=\sqrt{-1} \cK(z), z \in \A_q$.
Moreover, we can verify the equality
$\cK(z)=2 \rho_1(z)$, $z \in \A_q$. 

\SSC{Bergman Kernel and Szeg\H{o} Kernel of an Annulus}
\label{sec:KAq_SAq}
\subsection
{$K_{\A_q}$ expressed by Weierstrass $\wp$-function}
\label{sec:Weierstrass}

A CONS for the Bergman space on $\A_q$
is given by 
$\{ \widetilde{e}^{(q)}_n(z)\}_{n \in \Z}$ where we set
\[
\widetilde{e}^{(q)}_n(z)
=\begin{cases}
\displaystyle{
\sqrt{\frac{n+1}{1-q^{2(n+1)}} } z^n,
}
& n \in \Z \setminus \{ -1 \},
\cr
\displaystyle{
\sqrt{\frac{1}{-2 \log q}} z^{-1}},
& n=-1.
\end{cases}
\]
The Bergman kernel of $\A_q$ is then given by
\begin{align}
K_{\A_q}(z, w)
&:= k_{L^2_{\rB}(\A_q)}(z, w)
= \sum_{n \in \Z} \widetilde{e}^{(q)}_n(z)
\overline{ \widetilde{e}^{(q)}_n(w) }
\nonumber\\
&= - \frac{1}{2 \log q} \frac{1}{z \overline{w}}
+\frac{1}{z \overline{w}} 
\sum_{n \in \Z \setminus \{0\}} 
\frac{n}{1-q^{2n}}(z \overline{w})^n,
\quad z, w \in \A_q.
\label{eqn:K_Aq}
\end{align}
Using (\ref{eqn:wp_expansion2}) and the
notation (\ref{eqn:x_phi}), 
we can verify that 
this kernel is expressed
using the Weierstrass $\wp$-function 
(\ref{eqn:wp_expansion2}) as \cite{Ber70}
\begin{equation}
K_{\A_q}(z, w)
= - \frac{1}{2 \log q} \frac{1}{z \wbar}
- \frac{1}{z \wbar}
\Big(
\wp(\phi_{z \wbar})
+ \frac{P}{12}
\Big), \quad z, w \in \A_q.
\label{eqn:BergmanK3}
\end{equation}

\subsection
{Second log-derivatives of $S_{\A_q}$}
\label{sec:log_derivative}
Here we prove the following.
\begin{prop}
\label{thm:log_derivative}
For $r > 0$, 
the following equality holds,
\begin{equation}
\partial_z \partial_{\wbar} \log S_{\A_q}(z, w; r)
=\frac{\theta(-r)}{\theta(-r (z \wbar)^2)}
S_{\A_q}(z, w; r z \wbar)^2,
\quad z, w \in \A_q.
\label{eqn:log_deriv_Aq1}
\end{equation}
In particular, 
\begin{equation}
\Delta \log S_{\A_q}(z, z; r)
=4 \frac{\theta(-r)}{\theta(-r|z|^4)}
S_{\A_q}(z, z; r|z|^2)^2,
\quad z \in \A_q.
\label{eqn:log_deriv_Aq2}
\end{equation}
\end{prop}
\begin{proof}
Let
$\vartheta_1(\xi) 
:= \sqrt{-1} q^{1/4} q_0 e^{-\sqrt{-1} \xi} 
\theta(e^{2 \sqrt{-1} \xi})$ \cite[(11.2.2)]{GR04}.
This is one of the well-known four kinds
of \textit{Jacobi theta functions}
$\vartheta_i(\xi), i=0,1,2,3$.
(See \cite[Section 1.6]{GR04} and \cite[Section 20.5]{NIST10}.)
Using $\vartheta_1$, (\ref{eqn:S_qt_theta}) 
in Proposition \ref{thm:S_Aq_theta} is written as
\[
S_{\A_q}(z, w; r)
= \frac{\sqrt{-1}
\vartheta_1'(0) \vartheta_1(\phi_{-r z \wbar}/2)}
{2\vartheta_1(\phi_{-r}/2) \vartheta_1(\phi_{z \wbar}/2)},
\]
where the notation (\ref{eqn:x_phi}) has been used.
This gives
\[
 \partial_z \partial_{\wbar} \log S_{\A_q}(z, w; r)
= - 
\Big(
\partial_{\xi}^2 \log \vartheta_1(\xi) 
\Big|_{\xi=\phi_{-r z \wbar}/2}
- 
\partial_{\xi}^2 \log \vartheta_1(\xi)
\Big|_{\xi=\phi_{z \wbar}/2}
\Big)/(4 z \wbar).
\]
We use the equality 
$\wp(2 \omega_1 z/\pi )
=(\pi/(2 \omega_1))
\{\vartheta_1'''(0)/(3 \vartheta_1'(0))
-\partial_z^2 \log \vartheta_1(z)\}$ 
(see Eq.~(23.6.14) in \cite{NIST10}).
In the setting (\ref{eqn:omega1}) we have
\begin{equation}
\partial_z \partial_{\wbar} \log S_{\A_q}(z, w; r) 
=\big( \wp(\phi_{-r z \wbar})
- \wp(\phi_{z \wbar})
\big)/(z \wbar).
\label{eqn:ddlogS1}
\end{equation}

Now we use (\ref{eqn:addition_f}) in
Lemma \ref{thm:fundamental_eqs} given 
in Section \ref{sec:Weierstrass_elliptic}
\cite{Coo00}.
Combining with (\ref{eqn:S_qt_JK})
in Proposition \ref{thm:S_qt_JK}, (\ref{eqn:ddlogS1}) gives
\begin{align}
\partial_z \partial_{\wbar} \log S_{\A_q}(z, w; r)
&= 
f^{\rm JK}(z \wbar, - r z \wbar) 
f^{\rm JK}(z \wbar, - (r z \wbar)^{-1})/(z \wbar)
\nonumber\\
&= 
S_{\A_q}(z, w; r z \wbar)
S_{\A_q}(z, w; (r z \wbar)^{-1})/(z \wbar).
\label{eqn:ddlogS4}
\end{align}
The expression (\ref{eqn:S_qt_JK}) 
of $S_{\A_q}(\cdot, \cdot; r)$ 
in Proposition \ref{thm:S_qt_JK} gives 
\begin{align}
S_{\A_q}(z, w; (r z \wbar)^{-1})
&= f^{\rm JK}(z \wbar, -(r z \wbar)^{-1})
= - f^{\rm JK}((z \wbar)^{-1}, - r z \wbar)
\nonumber\\
&= - S_{\A_q}(z^{-1}, w^{-1}; r z \wbar),
\label{eqn:variation1}
\end{align}
where (\ref{eqn:JK2b}) is used.
On the other hand, 
the expression (\ref{eqn:S_qt_theta}) 
of $S_{\A_q}(\cdot, \cdot; r)$
in Proposition \ref{thm:S_Aq_theta} gives
$S_{\A_q}(z, w; r z \wbar)
= q_0^2 \theta(-r (z \wbar)^2)/
\theta(-r z \wbar, z \wbar)$, 
and
\[
S_{\A_q}(z^{-1}, w^{-1}; r z \wbar)
= 
\frac{q_0^2 \theta(-r z \wbar (z \wbar)^{-1})}
{\theta(-r z \wbar, (z \wbar)^{-1})}
= 
\frac{q_0^2 \theta(-r)}
{\theta(-r z \wbar, (z \wbar)^{-1})}
= -z \wbar 
\frac{q_0^2 \theta(-r)}
{\theta(-z \wbar r, z \wbar)},
\]
where (\ref{eqn:theta_inversion}) was used.
Hence, 
$S_{\A_q}(z^{-1}, w^{-1}; r z \wbar)
= - z \wbar \{\theta(-r)/\theta( -r (z \wbar)^2)\}
S_{\A_q}(z, w; r z \wbar)$
and (\ref{eqn:variation1}) gives
$S_{\A_q}(z, w; (r z \wbar)^{-1})
= z \wbar \{\theta(-r)/\theta( -r (z \wbar)^2)\}
S_{\A_q}(z, w; r z \wbar)$.
Then (\ref{eqn:ddlogS4}) proves the proposition.
\end{proof}

\subsection
{Relation between $K_{\A_q}$ and $S_{\A_q}$}
\label{sec:relation}
We prove the following relation
between the Bergman kernel $K_{\A_q}$
and the Szeg\H{o} kernel $S_{\A_q}$ of an annulus.
\begin{prop}
\label{thm:relation}
The equality 
\begin{equation}
S_{\A_q}(z, w)^2
= K_{\A_q}(z, w) + \frac{a}{z \wbar}, 
\quad z, w \in \A_q, 
\label{eqn:SK1}
\end{equation}
holds, where
\begin{equation}
a=a(q) =
e_2 + \frac{P}{12} + \frac{1}{2 \log q} 
= -2 \sum_{n \in \N}
\frac{(-1)^n n q^n}{1-q^{2n}}
+\frac{1}{2 \log q}.
\label{eqn:aq}
\end{equation}
\end{prop}
\begin{proof}
By Proposition \ref{thm:S_qt_JK}, 
$S_{\A_q}(z, w)^2
= f^{\rm JK}_q(z \wbar, -q)^2$.
Since
$f^{\rm JK}(z, a)=f^{\rm JK}(z, a/q^2)/z$
is given by (\ref{eqn:JK2c}),
we have
$f^{\rm JK}(z \wbar, -q) 
=f^{\rm JK}(z \wbar, -q^{-1})/(z \wbar)$, 
and hence
\begin{equation}
S_{\A_q}(z, w)^2
=\frac{1}{z \wbar} 
f^{\rm JK}(z \wbar, -q) f^{\rm JK}(z \wbar, -q^{-1}).
\label{eqn:SK3}
\end{equation}
Here we use (\ref{eqn:addition_f}) in
Lemma \ref{thm:fundamental_eqs} given 
in Section \ref{sec:Weierstrass_elliptic}
\cite{Coo00}.
Then
\[
f^{\rm JK}(z \wbar, -q) f^{\rm JK}(z \wbar, -q^{-1})
=\wp(\pi+\pi \tau_q)-\wp(\phi_{z \wbar})
=e_2-\wp(\phi_{z \wbar}),
\]
where we have used the setting (\ref{eqn:omega1}),
the notation (\ref{eqn:x_phi}) and 
the evenness of $\wp(z)$.
The equality (\ref{eqn:SK3}) is thus written as
$S_{\A_q}(z, w)^2
= - \wp(\phi_{z \wbar})/(z \wbar)
+e/(z \wbar)$.
Now we use (\ref{eqn:BergmanK3}).
Then (\ref{eqn:SK1}) is obtained with $a$
given by the first expression in (\ref{eqn:aq}).
If we set $z=-q/\wbar$ in (\ref{eqn:SK1}), 
then Lemma \ref{thm:zero_S} gives an equality,
$0=K_{\A_q}(-q/\wbar, w)-a/q$.
By (\ref{eqn:K_Aq}) with a short calculation,
the second expression for $a$
in (\ref{eqn:aq}) is obtained.
\end{proof}
\vskip 0.3cm
\begin{rem}
\label{rem:K_S_q=0}
The relationship (\ref{eqn:SK1}) 
between $S_{\A_q}$ and $K_{\A_q}$
with an additional term $a$
is concluded from a general theory
(see, for instance, Exercise 3 in Section 6, Chapter VII
of \cite{Neh52}, and Section 25 of \cite{Bel16}). 
It was shown in \cite{Bol14} that $a$ is readily determined 
by Lemma \ref{thm:zero_S} as shown above,
if the equality (\ref{eqn:SK1}) is established.
Here we showed direct proof of (\ref{eqn:SK1})
using the equality (\ref{eqn:addition_f})
between $f^{\rm JK}$ and $\wp$ \cite{Coo00}. 
By the explicit formulas (\ref{eqn:aq}) for $a$,
we see that $\lim_{q \to 0} a(q)=0$.
Therefore, the relation (\ref{eqn:SK1}) is reduced 
in the limit $q \to 0$ to
$S_{\D}(z, w)^2
= K_{\D}(z, w), z, w \in \D$,
which is a special case of (\ref{eqn:S_K_D}),
as expected. 
\end{rem}

\begin{small}

\end{small}
\end{document}

%% file: katori-shirai_v2d.tex
\usepackage{amsmath}
\usepackage{amsfonts}
\usepackage{amssymb}
\usepackage{latexsym}
\usepackage{epsfig}
\usepackage{graphicx}
\usepackage{oldgerm}
\usepackage{amsthm}

\makeatletter
\renewenvironment{proof}[1][\proofname]{\par
  \normalfont
  \topsep6\p@\@plus6\p@ 
\trivlist
  \item[\hskip\labelsep{#1}]\ignorespaces
}{%
\hbox{\rule[-2pt]{3pt}{6pt}}
\endtrivlist
}
\renewcommand{\proofname}{\textit{Proof.}}
\makeatother
 \newtheoremstyle{katori-thm}
 {}
 {}
 {\it}
 {}
 {\bf}
 {}
 {5pt}
 {\thmname{#1} \thmnumber{#2}\thmnote{\hspace{2pt}(#3)}}
\newtheoremstyle{katori-remark}
{}
{}
{\normalfont}
{}
{\bfseries}
{}
{7pt}
{\thmname{#1} \thmnumber{#2}\thmnote{\hspace{2pt}(#3)}}

\theoremstyle{katori-thm}
\newtheorem{thm}{Theorem}[section]
\newtheorem{lem}[thm]{Lemma}
\newtheorem{cor}[thm]{Corollary}
\newtheorem{prop}[thm]{Proposition}

\theoremstyle{katori-remark}
\newtheorem{rem}{Remark}

\newcommand{\SSC}[1]{\section{#1}\setcounter{equation}{0}}

\makeatletter
\def\filename{\texttt{\jobname.tex}} 
\makeatother

\makeatletter
\def\katoriname{
Makoto Katori
\footnote{
Department of Physics,
Faculty of Science and Engineering,
Chuo University, 
Kasuga, Bunkyo-ku, Tokyo 112-8551, Japan;
e-mail: katori@phys.chuo-u.ac.jp
}
}
\def\shirainame{ 
Tomoyuki Shirai 
\footnote{
Institute of Mathematics for Industry, 
Kyushu University, 
744 Motooka, Nishi-ku,
Fukuoka 819-0395, Japan; 
e-mail: shirai@imi.kyushu-u.ac.jp
}
}
\makeatother

\def\per{\mathop{\mathrm{per}}}
\def\dis={\stackrel{\rm d}{=}}
\def\Det{\mathop{\mathrm{Det}}}
\def\bE{{\bf E}}
\def\mS{\mathfrak{S}}
\def\sgn{{\rm sgn}}

\def\bra{\langle}
\def\ket{\rangle}

\def\alphabar{\overline{\alpha}}

\def\zbar{\overline{z}}
\def\wbar{\overline{w}}

\def\A{\mathbb{A}}

\def\C{\mathbb{C}}
\def\D{\mathbb{D}}

\def\HH{\mathbb{H}}

\def\LL{\mathbb{L}}

\def\N{\mathbb{N}}

\def\R{\mathbb{R}}

\def\X{\mathbb{X}}

\def\Z{\mathbb{Z}}

\def\cC{\mathcal{C}}
\def\cD{\mathcal{D}}

\def\cH{\mathcal{H}}
\def\cI{\mathcal{I}}

\def\cK{\mathcal{K}}

\def\cZ{\mathcal{Z}}

\def\bE{\mathbf{E}}

\def\ssG{\mathsf{G}}


%% file: Katori_Shirai_annulusPDPP_e3.bbl
\begin{thebibliography}{99} 
\bibitem{AM02}
Agler, J., McCarthy, J. E.:
Pick Interpolation and Hilbert Function Spaces. 
Graduate Studies in Mathematics, Vol. 44,
Amer. Math. Soc., Providence, RI (2002)

\bibitem{Ahl79}
Ahlfors, L.:
Complex Analysis.
McGrow Hill,
New York (1979)

\bibitem{Aro50}
Aronszajn, N.:
Theory of reproducing kernels. 
Trans. Amer. Math. Soc.
\textbf{68}, 337--404 (1950)

\bibitem{AIMO08}
Astala, K., lwaniec, T., Martin, G., Onninen, J.: 
Schottky's theorem on conformal mappings between annulus. 
In : Agranovsky, M. et. al (eds.), 
Complex Analysis and Dynamical Systems III; 
Contemporary Mathematics
\textbf{455}, pp. 35--39,
Amer. Math. Soc., Providence, RI (2008);
http://dx.doi.org/10.1090/conm/455

\bibitem{BF08}
Bauer, R. O., Friedrich, R. M.:
On chordal and bilateral SLE in multiply
connected domain.
Math. Z. {\bf 258}, (2008) (2008)

\bibitem{Bel95}
Bell, S. R.:
Simplicity of the Bergman, Szeg\H{o}
and Poisson kernel functions.
Math. Res. Lett.
\textbf{2}, 267--277 (1995)

\bibitem{Bel16}
Bell, S. R.: 
The Cauchy Transform, Potential Theory
and Conformal Mapping.
2nd ed.
CRC Press, Boca Raton, FL (2016)

\bibitem{Ber70}
Bergman, S.:
The Kernel Function and Conformal Mapping. 
2nd ed. 
Amer. Math. Soc., Providence, RI (1970)

\bibitem{BSZ00}
Bleher, P., Shiffman, B., Zelditch, S.: 
Universality and scaling of correlations between zeros 
on complex manifolds. 
Invent. math.
\textbf{142}, 351--395 (2000)

\bibitem{BBL92}
Bogomolny, E.,Bohigas, O., Leb{\oe}uf, P.:
Distribution of roots of random polynomials. 
Phys. Rev. Lett. 
\textbf{68}, 2726--2729 (1992)

\bibitem{BBL96}
Bogomolny, E.,Bohigas, O., Leb{\oe}uf, P.:
Quantum chaotic dynamics and random
polynomials. 
J. Stat. Phys. 
\textbf{85}, 639--679 (1996)

\bibitem{Bol14}
Bolt, M.: 
Szeg\H{o} kernel transformation law 
for proper holomorphic mappings. 
Rocky Mountain J. Math.
\textbf{44}, 779--790 (2014)

\bibitem{BQ18+}
Bufetov, A. I., Qiu, Y.:
Patterson--Sullivan measures for point processes
and the reconstruction of harmonic functions.
{\sf arXiv:math.PR/1806.02306}

\bibitem{BKT18}
Byun, S.-S., Kang, N.-G., Tak, H.-J.:
Annulus SLE partition functions and martingale-observables.
{\sf arXiv:math.PR/1806.03638}

\bibitem{Car02}
Cardy, J.:
Crossing formulae for critical percolation in an annulus.
J. Phys. A: Math. Gen.
\textbf{35}, L565--L572 (2002)

\bibitem{Car06}
Cardy, J.:
The O($n$) model on the annulus.
J. Stat. Phys. 
\textbf{125}, 1--21 (2006)

\bibitem{CHSDS06}
Castin, Y., Hadzibabic, Z., Stock, S., Dalibard, J., Stringari, S.: 
Quantized vortices in the ideal Bose gas: 
A physical realization of random polynomials.
Phys. Rev. Lett. 
\textbf{96}, 040405/1-4 (2006)

\bibitem{Cha06}
Chavel, I.:
Riemannian Geometry, A Modern Introduction.
2nd ed., Cambridge University Press, Cambridge (2006)

\bibitem{Coo00}
Cooper, S.: 
The development of elliptic functions according
to Ramanujan and Venkatachaliengar. 
Res. Lett. Inf. Math. Sci.
\textbf{1}, 65--78 (2000); \\
available at 
https://mro.massey.ac.nz/handle/10179/4390

\bibitem{CH04}
Courant, R., Hilbert, D.: 
Methods of Mathematical Physics. Vol.1,
Wiley-VCH, Weinheim (2004)

\bibitem{DEP01}
DeLillo, T. K., Elcrat, A. R., Pfaltzgraff, J. A.:
Schwarz-Christoffel mapping of the annulus.
SIMA Review {\bf 43}, 469--477 (2001)

\bibitem{DL19}
Demni, N., Lazag, P.: 
The hyperbolic-type point process.
J. Math. Soc. Japan
\textbf{71}, 1137--1152 (2019)

\bibitem{DT02}
Driscoll, T. A., Trefethen, L. N.:
Schwarz--Christoffel Mapping.
Cambridge University Press,
Cambridge (2002)

\bibitem{EK95}
Edelman, A., Kostlan, E.:
How many zeros of a random polynomial are real?
Bull. Amer. Math. Soc. (N.S.)
\textbf{32}, no.1, 1--37 (1995)

\bibitem{EG09}
Evans, S.N., Gottlieb, A.: 
Hyperdeterminantal point process. 
Metrika
\textbf{69}, 85--99 (2009)

\bibitem{For06}
Forrester, P. J.:
Particles in a magnetic field and plasma analogies: 
doubly periodic boundary conditions.
J. Phys. A: Math. Gen. 
{\bf 39}, 13025--13036 (2006)

\bibitem{For10}
Forrester, P. J.: 
Log-gases and Random Matrices. 
London Math. Soc. Monographs,
Princeton University Press, Princeton (2010)

\bibitem{FK14}
Fukushima, M., Kaneko, H.:
On Villat's kernels and BMD Schwarz kernels 
in Komatu-Loewner equations.
In: Crisan, D., Hambly, B., Zariphopoulous, T. (eds.)
Stochastic Analysis and Applications 2014, 
Springer Proceedings in Mathematics and Statistics, 
Vol. 100, pp.327--348,
Springer, (2014)

\bibitem{GR04}
Gasper, G., Rahman, M.: 
Basic Hypergeometric Series. 2nd ed.,
Cambridge University Press,
Cambridge (2004)

\bibitem{HBB10}
Hagendorf, C., Bernard, D., Bauer, M.:
The Gaussian free field and SLE$_4$ on doubly connected domains.
J. Stat. Phys. 
\textbf{140}, 1--26 (2010)

\bibitem{HL08}
Hagendorf, C., Le Doussal, P.:
SLE on doubly-connected domains and the winding
of loop-erased random walks.
J. Stat. Phys. 
\textbf{133}, 231--254 (2008)

\bibitem{Han96}
Hannay, J. H.: 
Chaotic analytic zero points: Exact statistics for those 
of a random spin state.
J. Phys. A
\textbf{29}, L101--L105 (1996)

\bibitem{HKZ00}
Hedenmalm, H., Korenblum, B., Zhu, K.:
Theory of Bergman Spaces.
Graduate Texts in Mathematics 199,
Springer, New York (2000)

\bibitem{Hel84}
Helgason, S.:
Groups and Geometric Analysis,
Integral Geometry, Invariant Differential Operators,
and Spherical Functions.
Academic Press, New York (1984)

\bibitem{HKPV09}
Hough, J. B., Krishnapur, M., Peres, Y., Vir\'ag, B.: 
Zeros of Gaussian Analytic Functions
and Determinantal Point Processes. 
University Lecture Series, Vol. 51, 
Amer. Math. Soc., Providence, RI (2009)

\bibitem{IKO05}
Ishikawa, M., Kawamuko, H., Okada, S.: 
A Pfaffian--Hafnian analogue of Borchardt's identity. 
Electron. J. Combin. 
\textbf{12}, note 9, 8pp. (2005)

\bibitem{Izy17}
Izyurov, K.:
Critical Ising interfaces in multiply-connected domains.
Probab. Theory Relat. Fields
\textbf{167}, 379--415 (2017)

\bibitem{K85} 
Kahane, J.~P.: 
Some Random Series of Functions. 2nd ed., 
Cambridge University Press, Cambridge (1985)

\bibitem{KN03}
Kajihara, Y., Noumi, M.:
Multiple elliptic hypergeometric series.
An approach from the Cauchy determinant.
Indag. Math. (N.S.) \textbf{14}, 395--421 (2003)

\bibitem{Kal17}
Kallenberg, O.: 
Random Measures, Theory and Applications.
Probability Theory and Stochastic Modelling Vol.77, 
Springer, Switzerland (2017)

\bibitem{Kat19b}
Katori, M.: 
Two-dimensional elliptic determinantal point processes 
and related systems. 
Commun. Math. Phys. 
\textbf{371}, 1283--1321 (2019)

\bibitem{KS20}
Katori, M., Shirai, T.:
Scaling limit for determinantal point processes on spheres. 
RIMS K\^oky\^uroku Bessatsu
\textbf{B79}, 123--138 (2020)

\bibitem{KS21}
Katori, M., Shirai, T.: 
Partial isometries, duality, and determinantal point processes.
Random Matrices: Theory and Applications
2250025, 70 pages (2021); \\
DOI:10.1142/S2010326322500253

\bibitem{Koo14}
Koornwinder, T.H.: 
On the equivalence of two fundamental theta identities. 
Anal. Appl. (Singap.)
\textbf{12}, 711--725 (2014)

\bibitem{Kra05}
Krattenthaler, C.: 
Advanced determinant calculus: a complement. 
Linear Algebra Appl. 
\textbf{411}, 68--166 (2005)

\bibitem{Kri09}
Krishnapur, M.:
From random matrices to random analytic functions.
Ann. Probab. 
\textbf{37}, 314--346 (2009) 

\bibitem{Law89}
Lawden, D. F.:
Elliptic Functions and Applications.
Applied Mathematical Sciences, Vol.80,
Sprnger, New York (1989)

\bibitem{Leb99}
Leb{\oe}uf, P.:
Random analytic chaotic eigenstates.
J. Stat. Phys. 
\textbf{95}, 651--664 (1999)

\bibitem{Leb00}
Leb{\oe}uf, P.:
Random matrices, random polynomials and Coulomb systems.
J. Phys. IV France
\textbf{10} PR5, Pr5-45--Pr5-52 (2000)

\bibitem{Lie66}
Lieb, E. H.:
Proofs of some conjectures on permanents.
J. Math. Mech.
\textbf{16}, 127--134 (1966)

\bibitem{LV10}
Luque, J.-G., Vivo, P.:
Nonlinear random matrix statistics, symmetric functions
and hyperdeterminants.
J. Phys. A: Math. Theor.
\textbf{43}, 085213 (20pp) (2010)

\bibitem{MM92}
Marcus, M., Minc, H.:
A Survey of Matrix Theory and Matrix Inequalities.
Dover Publications, New York (1992)

\bibitem{Mat08}
Matsumoto, S.: 
Hyperdeterminantal expressions for Jack functions
of rectangular shapes. 
J. Algebra
\textbf{320}, 612--632 (2008)

\bibitem{MS13}
Matsumoto, S., Shirai, T.:
Correlation functions for zeros of a Gaussian power series
and Pfaffians.
Electron. J. Probab. 
\textbf{18}, no.49, 1--18 (2013)

\bibitem{MM06}
Mccullagh, P., M{\o}ller, J.:
The permanental process.
Adv. Appl. Prob. (SGSA)
\textbf{38}, 873--888 (2006)

\bibitem{MS94} 
Mccullough, S, Shen, L.C.: 
On the Szeg\H{o} kernel of an annulus. 
Proc. Amer. Math. Soc.
{\bf 121}, 1111--1121 (1994)

\bibitem{MS12}
Mccullough, S., Sultanic, S.: 
Agler--commutant lifting on an annulus.
Integr. Equ. Oper. Theory
\textbf{72}, 449--482 (2012)

\bibitem{MM97}
McKean, H., Moll, V.:
Elliptic Curves. Function Theory, Geometry, Arithmetic.
Cambridge University Press, Cambridge (1997)

\bibitem{Min78}
Minc, H.:
Permanent. 
Encyclopedia of Mathematics and its Applications, Vol. 6, 
Addison--Wesley, MA (1978)

\bibitem{Neh52}
Nehari, Z.:
Conformal Mapping.
Dover, New York (1952)

\bibitem{Neh52b}
Nehari, Z.:
On weighted kernels. 
J. Anal. Math. 
\textbf{2}, 126--149 (1952) 

\bibitem{NIST10}
Olver, F. W. J., Lozier, D. W., Boisvert, R. F., Clark, C. W.
(eds.) :
NIST Handbook of Mathematical Functions. 
U.S. Department of Commerce, 
National Institute of Standards and Technology, 
Washington, DC/ 
Cambridge University Press, Cambridge (2010); 
available at 
http://dlmf.nist.gov

\bibitem{Pee93}
Peetre, J.: Correspondence principle for 
the quantized annulus, Romanovski polynomials and Morse potentials. 
J. Funct. Anal.
\textbf{117}, 377--400 (1993)

\bibitem{PV05}
Peres, Y., Vir\'ag, B.: 
Zeros of the i.i.d. Gaussian power series. 
A conformally invariant determinantal process.
Acta Math. 
\textbf{194}, 1--35 (2005)

\bibitem{Rai10}
Rains, E. M.:
Transformations of elliptic hypergeometric integrals.
Ann. Math. 
\textbf{171}, 169--243 (2010)

\bibitem{Rem18}
Remy, G.:
Liouville quantum gravity on the annulus.
J. Math. Phys.
\textbf{59}, 082303 (2018)

\bibitem{RS06}
Rosengren, H., Schlosser, M.:  
Elliptic determinant evaluations and the Macdonald identities
for affine root systems.
Compositio Math. 
\textbf{142}, 937--961 (2006)

\bibitem{Sar65}
Sarason, D.: 
The $H^p$ Spaces of an Annulus. 
Memoirs of the American Mathematical Society,
No. 56,
Amer. Math. Soc., Providence, RI (1965)

\bibitem{Shi06}
Shirai, T.:
Large deviations for the fermion point process
associated with the exponential kernel.
J. Stat. Phys. 
\textbf{123}, 615--629 (2006)

\bibitem{Shi07}
Shirai, T.:
Remarks on the positivity of $\alpha$-determinants.
Kyushu J. Math.
\textbf{61}, 169--189 (2007)

\bibitem{Shi12}
Shirai, T.: 
Limit theorems for random analytic functions and their zeros.  
In: Functions in Number Theory and Their
Probabilistic Aspects -- Kyoto 2010, 
RIMS K\^{o}ky\^{u}roku Bessatsu
\textbf{34}, 335--359 (2012)

\bibitem{ST00}
Shirai, T., Takahashi, Y.:
Fermion process and Fredholm determinant.
In: Begehr, H. G. W., Gilbert, R. P., Kajiwara, J. (eds.), 
Proceedings of the Second ISAAC Congress,
Vol. 1, pp. 15--23, 
Kluwer Academic Publishers, Dordrecht (2000) 

\bibitem{ST03a}
Shirai, T., Takahashi, Y.:
Random point fields associated with certain
Fredholm determinants I:
fermion, Poisson and boson point process. 
J. Funct. Anal.
\textbf{205}, 414--463 (2003)

\bibitem{ST03b}
Shirai, T., Takahashi, Y.:
Random point fields associated with certain
Fredholm determinants II:
fermion shifts and their ergodic and Gibbs properties.  
Ann. Probab.
\textbf{31}, 1533--1564 (2003)

\bibitem{ST04}
Sodin, M., Tsirelson, B.:
Random complex zeros, I. Asymptotic normality. 
Israel J. Math.
\textbf{144}, 125--149 (2004)

\bibitem{Sos00}
Soshnikov, A.:
Determinantal random point fields. 
Russian Math. Surveys
\textbf{55}, 923--975 (2000) 

\bibitem{Spi02}
Spiridonov, V. P.:
Theta hypergeometric series. 
In: Malyshev, V.A., Vershik, A. M. (eds.), Asymptotic Combinatorics
with Applications to Mathematical Physics, 
Kluwer Academic, Dordrecht (2002), pp. 307--327

\bibitem{TV97}
Tarasov, V., Varchenko, A.:
Geometry of $q$-hypergeometric functions, quantum
affine algebras and elliptic quantum groups.
Ast\'erisque 
\textbf{246} (1997)

\bibitem{TT99}
Tegtmeyer, T. J., Thomas, A. D.: 
The Ahlfors map and Szeg\H{o} kernel for an annulus. 
Rocky Mountain J. Math.
\textbf{2}, 709--723 (1999)

\bibitem{Ven12}
Venkatachaliengar, K, Cooper, S. (ed): 
Development of Elliptic Functions According to Ramanujan.  
Monograph in Number Theory, Vol.6, 
World Scientific, Singapore (2012)

\bibitem{War02}
Warnaar, S. O.:
Summation and transformation formulas for elliptic hypergeometric series. 
Constr. Approx. 
\textbf{18}, 479--502 (2002)

\bibitem{Wei76}
Weil, A.:
Elliptic Functions According to Eisenstein and Kronecker.
Springer, Berlin (1976)

\bibitem{Zha04}
Zhan, D.:
Stochastic Loewner evolution in doubly
connected domains.
Probab. Theory Relat. Fields
\textbf{129}, 340--380 (2004)

\end{thebibliography}
